\documentclass[11pt]{article}

\makeatletter
\let\@fnsymbol\@arabic
\makeatother

\usepackage{appendix}
\usepackage{booktabs}

\usepackage{graphicx}
\usepackage{makecell}
\usepackage{amsmath}
\usepackage{amsfonts}
\usepackage{amssymb}
\usepackage{epsfig}

\usepackage{amsmath}
\usepackage{amsfonts}
\usepackage{amssymb}
\usepackage{dsfont}
\usepackage{mathrsfs}
\usepackage{bbm}
 \usepackage{hyperref}
\usepackage{amsmath}
\usepackage{amsfonts}
\usepackage{amssymb}
\usepackage{multirow}
\usepackage{mathrsfs}
\usepackage[dvipsnames,usenames]{color}

\usepackage{subcaption}
\usepackage{caption}
\usepackage{svg}
\usepackage{diagbox}

\usepackage{natbib}

\usepackage{enumerate}

\renewcommand{\Box}{\framebox{\rule{0.3em}{0.0em}}}

\newtheorem{theorem}{Theorem}[section]
\newtheorem{theorem*}{Theorem}[subsubsection]
\newtheorem{lemma}{Lemma}[section]
\newtheorem{proposition}{Proposition}[section]

\newtheorem{remark}{Remark}[section]
\newtheorem{specification}{Model specification}[section]
\newtheorem{definition}{Definition}[section]

\newtheorem{assumption}{Assumption}[section]
\newtheorem{corollary}{Corollary}[section]

\newcommand{\setd}{{ d \kern -.15em l}}
\newcommand{\hatsetd}{ d \hat{\kern -.15em l }}
\newcommand{\dd}{\mathsf {d\kern -0.07em l}} 

\newcommand{\bgeqn}{\begin{eqnarray}}
\newcommand{\edeqn}{\end{eqnarray}}
\newcommand{\bgeq}{\begin{eqnarray*}}
\newcommand{\edeq}{\end{eqnarray*}}
\newcommand{\bec}{\begin{center}}
\newcommand{\enc}{\end{center}}
\newcommand{\R}{{\rm I\!R}}

\newcommand{\inmat}[1]{\mbox{\rm {#1}}}

\newcommand{\be}{\begin{equation}}
\newcommand{\ee}{\end{equation}}

\def\e{\epsilon}

\def\bbe{{\Bbb{E}}} 

\renewcommand{\Box}{\hfill \rule{2.3mm}{2.3mm}}

\numberwithin{equation}{section}

\title{
Risk-Averse Bayesian Games with an Unknown Type Distribution: Bayesian Learning, Equilibrium Analysis, and Finite-Sample Guarantees
}

\author{
Yuan Tao\thanks{Department of Systems Engineering and Engineering Management, The Chinese University of Hong Kong. Email: ytao@se.cuhk.edu.hk
}
\,\,and\,\,
Huifu Xu\thanks{Department of Systems Engineering and Engineering Management, The Chinese University of Hong Kong. Email: hfxu@se.cuhk.edu.hk.}
}

\begin{document}

\maketitle

\begin{abstract}
Classical Bayesian games assume that the joint distribution of players' types is common knowledge, an assumption that rarely holds in practical applications.
We address this issue by studying a group of myopic players who repeatedly interact under the Bayesian Nash conjecture and learn a parametric joint type distribution from type profiles observed over time within a Bayesian learning framework.
Players may hold heterogeneous priors and may be risk averse toward both epistemic uncertainty about the distributional parameter and aleatoric uncertainty in their rivals' types.
We propose two models: a Bayesian Nash equilibrium among risk-averse Bayesian learners (BNE-RABL), in which epistemic and aleatoric risks are evaluated separately, and a BNE based on Bayesian predictive distributions (BNE-BPD), in which the two sources of uncertainty are integrated into a Bayesian predictive distribution and evaluated through a single risk measure.
Under suitable conditions, we establish the existence and uniqueness 
of both equilibria,
and derive 
non-asymptotic convergence rates of the equilibrium sequences toward the corresponding oracle BNE under the true type distribution as the game is repeatedly played and more type profiles are observed.
These results further show that the discrepancy between the BNE-RABL and BNE-BPD strategies vanishes as the sample size grows, which illustrates that the two models can be regarded as effective approximations to each other.
We apply the proposed models to a price competition problem and numerically illustrate the theoretical results.
\end{abstract}


\section{Introduction}
Bayesian games are a class of Nash games where  players have their own private information characterized by
type parameters.
Specifically, each player knows its own type but does not have complete information about the types of other players. 
For example, in a price competition problem, each firm typically has more accurate information about its own cost than about the costs of its competitors.
Accordingly, the type of each firm is naturally taken to represent its loss parameter \citep{ui2016bayesian,liu2025bayesian}.
Since each player chooses an optimal action after observing its own type, equilibrium strategies are necessarily type-contingent. A Bayesian Nash equilibrium (BNE) is a profile of such type-contingent strategies in which no player can improve its expected payoff through a unilateral deviation \citep{harsanyi1967games}.
Over the past few decades, the BNE framework has been  widely used in operations research and management science to model and analyze strategic interactions under incomplete information. 
Recent applications include price competition \citep{liu2025bayesian,du2026each},
inventory management \citep{kock2026deep,cui2024supply},
blockchain management \citep{chen2025bayesian},
auctions and contests \citep{pieroth2025deep,su2025existence}.

A central assumption in the classical BNE framework is that the joint probability distribution of the type parameters is common knowledge. However, this assumption may be restrictive in many data-driven settings due to limited information availability. First, the true joint distribution may be unknown, or the information required to specify it may be insufficient \citep{jackson1997social}. When the same strategic environment is encountered repeatedly and realized type profiles become observable, players can gradually learn the underlying distribution from accumulated data. This observation motivates a framework in which the type distribution is learned through repeated interaction and the progressive acquisition of information.
Second, players may enter the game with different information and subjective assessments, leading to heterogeneous prior beliefs about the unknown distributional parameters \cite{dekel2004learning}. 
This consideration suggests a relaxation of the common-prior assumption underlying classical Bayesian games, 
under which the true type distribution is initially unknown and players hold heterogeneous subjective priors.
Related work has likewise questioned the descriptive plausibility and necessity of common priors in incomplete-information environments  \citep{gul1998comment,kojevnikov2025existence}.

To address these challenges, we investigate a sequence of repeatedly played Bayesian games 
in which the unknown joint distribution of type parameters is parametric 
and the players learn the joint distribution through repeated interaction: 
initially, each player starts with a prior belief over the parameters of the joint distribution;
in every round, a player updates its belief using the type profiles observed in previous rounds and then selects an action under the Bayesian–Nash conjecture given its current type;
at the end of the round, the realized type profile is publicly revealed and is used in the next update.
We model the players as 
Bayesian learners
whose prior beliefs (and hence posterior beliefs) 
are
heterogeneous
to capture the differences in their available information and subjective judgment.
To avoid introducing an additional hierarchy of beliefs,
we take the players' prior/posterior beliefs and their risk attitudes to be common knowledge.

In the learning process described above,
players face two sources of uncertainty in each round.
The first is the aleatoric uncertainty inherent in the rivals' types, which is irreducible in the following sense: 
even if the true distribution of type parameters were known, a player would still not observe its rivals' realized types when choosing its action. 
The second is the epistemic uncertainty about the true type distribution itself, which stems from incomplete information and, unlike aleatoric uncertainty, is learnable: 
it diminishes as more samples of type profiles are observed \citep{hullermeier2021aleatoric}.
These two uncertainties may affect decisions in different ways. 
If a rival's type represents, say, its production cost or technology level, then an unfavorable realization can directly reduce a player's current payoff even when the distribution is fully known. 
By contrast, uncertainty about the distribution itself reflects a lack of statistical information about the environment and tends to be less important as more data accumulate. 
In the Bayesian learning framework, 
it is typically assumed that the true joint distribution of the type parameters belongs to a known parametric family which captures the aleatoric uncertainty, 
and that each player forms their own prior and posterior distributions over the parameters which capture the epistemic uncertainty.

Motivated by the discussion above, we first consider the equilibrium model which preserves this hierarchical structure of aleatoric and epistemic uncertainty explicitly.
The players are risk-averse and adopt different risk attitudes toward the aleatoric and epistemic uncertainties. 
Consequently, each player's objective function takes the form of nested risk measures: 
the inner risk measure evaluates the risk arising from aleatoric uncertainty under a given type distribution,
and the outer risk measure assesses the risk from epistemic uncertainty over 
the set of all plausible 
type distributions.
We term the equilibrium induced by this objective the Bayesian Nash equilibrium among risk-averse Bayesian learners (BNE-RABL).
Risk-averse games 
have been 
well studied,
see, e.g., \cite{pang2017two,su2025continuous}.
These works assume that the true probability distribution of the underlying uncertainty is known, while we consider the case with an unknown true distribution and employ nested risk measures to deal with epistemic and aleatoric uncertainty.
Such formulations with nested risk measures are common in the recent Bayesian optimization literature, see e.g., \cite{wu2018bayesian,lin2022bayesian,shapiro2026episodic,ma2024bayesian,milz2026stochastic}.
In contrast to these studies,
however, our inner risk measure for the aleatoric uncertainty is evaluated with respect to the conditional distribution of the rivals' type parameters, 
which yields a more complex uncertainty structure and incurs a substantial simulation cost when evaluating each player's objective function.

To simplify the structure of uncertainty within the Bayesian learning framework, 
we alternatively consider the equilibrium model where players combine the underlying type distribution and the prior/posterior distribution over the parameters 
into a Bayesian predictive distribution (BPD) \citep{geweke2010comparing,ghahramani2015probabilistic}. 
The BPD captures the aleatoric and epistemic uncertainties simultaneously, and
each player applies a single risk measure to it.  
We refer to the equilibrium induced by this formulation as the Bayesian Nash equilibrium based on Bayesian predictive distributions 
(BNE-BPD). 
Nevertheless,
the BNE-BPD and the BNE-RABL generally differ, and this discrepancy persists even 
when all risk measures are expectations.
This modeling approach, however, has received little attention in the classical Bayesian optimization literature.
Therefore, it is crucial to explore the properties of the BNE-RABL and BNE-BPD models and the relationship between them.

We emphasize that our analysis focuses on the myopic setting,
in which players' actions do not influence the realization of future type parameters and each player optimizes only its current-stage payoff. 
This setting differs from much of the existing literature on learning in repeated games, where players maximize discounted intertemporal payoffs and may strategically influence future information and play \citep{dekel2004learning,noguchi2015bayesian,norman2022possibility,mertikopoulos2016learning}. 
This literature typically studies how players learn their rivals' beliefs and strategies from their own payoff feedback and their rivals' observed actions. 
By contrast, models with myopic players rule out intertemporal strategic incentives, and thereby allow us to concentrate on the convergence analysis and limiting behavior of the equilibria among Bayesian learners. 
The assumption about myopic players is likewise widely adopted in the game theory literature, see, e.g., \cite{esponda2016berk,molavi2016learning,jiang2017distributed,harel2021rational,wu2025convergence}.
The main contributions of the paper can be summarized as follows:
\begin{itemize}

    \item[(a)] \textbf{Equilibrium models among Bayesian learners.} 
    We develop a new framework that integrates a Bayesian learning scheme into classical Bayesian game theory in order to address the players' lack of knowledge about the joint distribution of types. 
    We propose two equilibrium models: 
    the interim BNE-RABL model (Definition~\ref{def:BNE-HB}), which employs a composite risk measure that explicitly captures players' attitudes toward the two distinct forms of uncertainty, epistemic and aleatoric;
    and 
    the interim BNE-BPD model  (Definition~\ref{def:BNE-BPD}), which is based on the Bayesian predictive distribution (BPD), which consolidates aleatoric and epistemic uncertainty into a single distribution to which each player applies one risk measure. 
    We demonstrate the differences and connections among these equilibrium concepts within the Bayesian learning framework (Remark~\ref{remark:connection-difference}).
    Notably, calculating the objective function in the interim BNE-RABL model requires nested sampling of the unknown distributional parameter and then the rivals' types under it.
    This motivates us to use the interim BNE-BPD model as a tractable approximation with direct sampling from a conditional Bayesian predictive distribution.

    \item[(b)] \textbf{Existence and uniqueness of equilibrium.}
    We provide a rigorous theoretical analysis to establish the existence and uniqueness of equilibrium.
    Unlike the existing literature, e.g., \cite{meirowitz2003existence,guo2021existence,tao2025generalized,su2025existence}, where the assumptions are imposed directly on the objective function, 
    we identify the conditions on the model primitives, including the loss function, the risk measures, and the probability distributions, that guarantee existence and uniqueness (Theorems~\ref{thm:rabl-existence-schauder} and \ref{thm:rabl-contraction}).
    In doing so, we ensure that these assumptions are verifiable and applicable within the risk-averse Bayesian learning framework. 
    In particular, we establish the uniqueness of equilibrium in the case where the effect of a player's own action dominates the aggregate effect of the rivals' actions. 
    We impose refined assumptions on the gradients of the risk measures and 
    show how these assumptions are satisfied for several commonly used risk measures, including conditional value-at-risk (CVaR) and entropic risk measure in a price competition model.
    This analysis reveals how nonlinear risk measures may amplify cross-player marginal effects and thereby influence the equilibrium.

    \item[(c)]
    \textbf{Non-asymptotic convergence and the difference between BNE-RABL and BNE-BPD.}
    We prove that the equilibria among Bayesian learners converge to the BNE under the true joint type distribution as the number of rounds increases, 
    and derive a non-asymptotic convergence rate (Theorems~\ref{thm:convergence-rate-bne-rabl} and \ref{thm:convergence-rate-bne-bpd}). 
    This contrasts with the classical asymptotic convergence results in the Bayesian optimization literature, e.g., \cite{shapiro2026episodic,ma2024bayesian,milz2026stochastic}. 
    Building on these convergence results, 
    we quantify the discrepancy between the BNE-RABL and BNE-BPD strategies at any given round 
    and show that it vanishes as the number of rounds increases (Corollary~\ref{corollary:bne-rabl-and-bne-bpd}). 
    This result shows that the strategies obtained from the two models can be regarded as effective approximations to each other.
    We conduct numerical experiments in a price competition model to evaluate the performance of the Bayesian learning framework in the Bayesian game setting and to verify the convergence results (Section~\ref{sec:app_price_competition}).

\end{itemize}

The rest of the paper is organized as follows.
In Section~\ref{sec:preliminary}, we recall some notions of risk measures and Bayesian games.
In Section~\ref{sec:model}, we propose the interim BNE-RABL model and the interim BNE-BPD model and discuss the relationship between them.
Section~\ref{sec:exist-unique} investigates the existence and uniqueness of the proposed equilibrium models, and Section~\ref{sec:asymptotic-converge} analyzes the convergence behavior of both models.
In Section~\ref{sec:app_price_competition}, we apply the proposed framework to a price competition setting to verify the preceding assumptions and numerically illustrate the theoretical results.
Finally, we conclude the paper in Section~\ref{sec:concluding_remarks}.

\section{Preliminaries}

\label{sec:preliminary}

\subsection{Risk measure}
Consider a probability space $(\Omega, \mathcal{F}, \mathbb{P})$ and
real-valued random variables in $L^p:=L^p(\Omega,\mathcal{F},\mathbb{P})$, which
represent random losses. A functional $\rho:L^p\to\R$ is called a
\textit{monetary risk measure} if it satisfies the following two properties:
    (a) \textit{monotonicity}: for any $\xi,\zeta\in L^p$, $\xi\geqslant\zeta$
    almost surely implies $\rho(\xi)\geqslant\rho(\zeta)$;
    (b) \textit{cash invariance}: $\rho(\xi+c)=\rho(\xi)+c$ for any $\xi\in L^p$
    and $c\in\R$.
A monetary risk measure $\rho$ is said to be \textit{convex} if it additionally
satisfies
(c) \textit{convexity}: for any $\xi,\zeta\in L^p$ and $\lambda\in[0,1]$,
    $\rho(\lambda\xi+(1-\lambda)\zeta)\leqslant\lambda\rho(\xi)+(1-\lambda)\rho(\zeta)$,
and a convex risk measure $\rho$ is said to be \textit{coherent} if it further
satisfies
(d) \textit{positive homogeneity}: $\rho(\lambda\xi)=\lambda\rho(\xi)$ for
    any $\xi\in L^p$ and $\lambda\geqslant 0$.
We shall also invoke two additional properties. A risk measure $\rho$ is
(e) \textit{law invariant}: if $\rho(\xi)=\rho(\zeta)$ for any $\xi,\zeta\in L^p$ share
the same probability distribution;
(a') \textit{strictly monotonic}: if $\rho(\xi)>\rho(\zeta)$
for every $\xi,\zeta\in L^p$ with $\xi\geqslant\zeta$ almost surely and
$\mathbb{P}\{\omega\in\Omega:\xi(\omega)>\zeta(\omega)\}>0$.

When $\rho$ is a law invariant risk measure defined over a non-atomic space $L^p$, it admits an equivalent representation
as a \textit{risk functional} over the space of probability distributions induced by random variables. 
Specifically, there exists
a functional $\varrho:\mathcal{M}_1^p\to\R$ such that
$
    \rho(\xi) = \varrho(P_\xi) := \rho\!\left(F_{P_\xi}^{-1}(U)\right),
$
where $P_\xi$ denotes the distribution induced by $\xi$, $U$ is a random variable uniformly
distributed on $[0,1]$, $F_{P_\xi}$ is the cumulative distribution function
associated with $P_\xi$, and
$F_{P_\xi}^{-1}(t):=\inf\{s\in\R:F_{P_\xi}(s)\geqslant t\}$ is the corresponding quantile function with $t\in[0,1]$.

In this paper, we will primarily use 
a specific class of convex risk measures which can be represented in the following parametric form:
\begin{eqnarray} \label{eqn:parametric-risk-measure}
   \rho(\xi):=\inf_{s \in \mathcal{S}}\mathbb{E}[\Psi(\xi,s)],
\end{eqnarray}
where $\mathcal{S}$ is a subset of a finite-dimensional vector space and $\Psi:\R\times \mathcal{S}$ is a real-valued function.
This class of risk measures is proposed in  \cite{guigues2023risk}, 
and covers 
a number of important risk measures.
For instance, if 
$\Psi(\xi,s):=\xi$, 
then %
\eqref{eqn:parametric-risk-measure} %
recovers the expectation.
If we set
$\mathcal{S}:=\R$ and $\Psi(\xi,s) := s - u(s-\xi)$ with $u:\R \rightarrow \R$ being a proper closed concave and non-decreasing utility function, 
then 
\eqref{eqn:parametric-risk-measure} recovers the negated optimized certainty equivalent (OCE) in \cite{ben1986expected,ben2007old},
\begin{eqnarray} \label{eqn:oce}
    \rho(\xi) = \inf_{s\in \R} s - \mathbb{E}\left[ u(s-\xi) \right]. 
\end{eqnarray}
In the latter case, 
if we set $u(x)=\frac{1}{1-\beta}x$ for $x\leq 0$ and $u(x)=0$ for $x>0$, then we obtain the 
CVaR with confidence level $\beta$
\begin{eqnarray*}
    \rm{CVaR}_{\beta_i}(\xi) = \inf_{s\in\R} s + \frac{1}{1-\beta} \mathbb{E}\left[ (\xi-s)_+ \right],
\end{eqnarray*}
where $(x)_+:= \max\{0,x\}$;
in the case that
$u(x) = \frac{1}{\gamma}(1-e^{-\gamma x})$ with $\gamma>0$, we obtain the entropic risk measure 
\begin{eqnarray*}
    \rho^{\rm ent}(\xi) = \frac{1}{\gamma}\ln \mathbb{E}[e^{\gamma\xi}].
\end{eqnarray*}
Moreover, by setting $\mathcal{S}=\R$ and $\Psi(\xi,s) = u(s) - u(s - \xi)$, 
the parametric form \eqref{eqn:parametric-risk-measure} recovers the negated modified version of optimized certainty equivalent (MOCE) in \cite{wu2022preference},
\begin{eqnarray*}
    \rho(\xi) = \inf_{s\in \R} u(s) - \mathbb{E}\left[ u(s - \xi) \right]. 
\end{eqnarray*}

\subsection{Classical Bayesian game and equilibrium}
Consider a non-cooperative game with $n$ players. 
Let $N:=\{1,\ldots,n\}$. 
Each player $i\in N$ has a loss function denoted by $c_i(a_i,a_{-i},\theta_i,\theta_{-i})$, which depends on the player's action $a_i$, the rivals' actions $a_{-i}:=(a_1,\dots,a_{i-1},a_{i+1},\dots,a_n)$,
the player $i$'s type $\theta_i$,
and the rivals' types $\theta_{-i}:= (\theta_1,\dots,\theta_{i-1},\theta_{i+1},\dots,\theta_n)$.
For $i\in N$,
we assume that $\theta_i\in
\Theta_i\subset \R^{d_i}$ and $a_i\in\mathcal{A}_i\subset \R^{n_i}$.
Following the terminology of \cite{meirowitz2003existence}, a profile of types is a vector $\theta:= (\theta_1,\dots,\theta_n)\in \Theta:= \Theta_1\times \cdots \times \Theta_n$ and a profile of actions is a vector $a:= (a_1,\dots,a_n)\in \mathcal{A}:= \mathcal{A}_1\times \cdots \times \mathcal{A}_n$.
Throughout the paper, we make the following basic assumptions.

\begin{assumption}
\label{basic-assumption}
For $i\in N$,
the following are satisfied.
(a) $\Theta_i$ and $\mathcal A_i$ are nonempty, compact, and convex.
(b) %
   The loss function $c_i:\mathcal{A}_i\times \mathcal{A}_{-i} \times \Theta_i \times \Theta_{-i} \rightarrow \R$ is continuous.
   Let $\bar c_i$ be a positive constant such that 
\begin{eqnarray} \label{eqn:c-bound}
    |c_i(a_i,a_{-i},\theta_i,\theta_{-i})|
    \leq \bar c_i, \quad \forall (a_i,a_{-i},\theta_i,\theta_{-i})
\in\mathcal A_i\times\mathcal A_{-i}\times\Theta_i\times\Theta_{-i}.
\end{eqnarray}
\end{assumption}
These conditions are standard in  %
Bayesian game literature, see e.g., \cite{meirowitz2003existence} and
\cite{guo2021existence}. 

In the standard setup in Bayesian games, 
$\theta$ is a random vector endowed with a distribution $\eta$ with full support (i.e., $\rm{supp}(\eta) = \Theta$).
$\eta$ is common knowledge among all players.
This information describes the probability that a particular type profile $\theta$ is realized, which may be retrieved from empirical data in practice.
After the realization of type $\theta$, each player $i$ only knows its own type but not others, and thereby constructs a belief about $\theta_{-i}$ using a conditional probability distribution $\eta(\cdot|\theta_i)$ with the density function 
$
    p(\theta_{-i}\mid \theta_i)
    =
    \frac{p(\theta_i,\theta_{-i})}
    {\int_{\Theta_{-i}} p(\theta_i,\theta_{-i})\,d\theta_{-i}},
$
which describes the probability that player $i$'s rivals have a particular type profile $\theta_{-i}$.
Throughout the paper, we will use $\theta$ to denote both a deterministic element of $\Theta$ and a random vector $\theta(\omega)$ depending on the context.
According to the time at which the information of players' types is observed, a Bayesian game
can be divided into three stages. In the \emph{ex-ante stage}, players are uncertain
about the potential realizations of both their own type parameters $\theta_i$ and
the other players' types $\theta_{-i}$.
In the \emph{interim stage}, each player has observed its own type $\theta_i$ but
still does not know its rivals' types $\theta_{-i}$.
In the \emph{ex-post stage}, players have complete information about all
type parameters, including those of their rivals; see, e.g.,
\cite{koniorczyk2020ex,saglam2025bayesian} for a detailed discussion on the three stages.

In Bayesian games with pure strategies, the response (strategy) function of player $i$ is represented by $f_i$ mapping from the type space $\Theta_i$ to the action space $\mathcal{A}_i$.
For $i\in N$, we denote by $\mathcal{F}_i$ the set of measurable functions $f_i:\Theta_i \rightarrow \mathcal{A}_i$ 
endowed with the supremum norm
$\|f_i\|_\infty = \sup_{\theta_i\in \Theta_i}\|f_i(\theta_i)\|$, where $\|\cdot\|$
denotes the Euclidean norm in finite-dimensional spaces,
and by $\mathcal{C}_i \subset \mathcal{F}_i$ the set of continuous functions.
For simplicity of notation, we write
$
    \mathcal{F}:=\prod_{j\in N}\mathcal{F}_j,
    \mathcal{F}_{-i}:=\prod_{j\in N\setminus\{i\}}\mathcal{F}_j ,
    \mathcal{C}:=\prod_{j\in N}\mathcal{C}_j,
    \mathcal{C}_{-i}:=\prod_{j\in N\setminus\{i\}}\mathcal{C}_j .
$
With a slight abuse of notation,
we use the same symbol for the supremum norm on $\mathcal{F}$ and on $\mathcal{F}_{-i}$ when the relevant space is clear from context, i.e., for $f\in\mathcal{F}$,
$ \|f\|_\infty := \max_{j\in N}\|f_j\|_\infty,$
and    for $f_{-i}\in \mathcal{F}_{-i}$,
$\|f_{-i}\|_\infty:=\max_{j\neq i}\|f_j\|_\infty.$

We define each player's expected loss given the rivals' strategy function $f_{-i}$ by 
\begin{equation}
    \min_{a_i\in \mathcal{A}_i}
    \mathbb{E}_{\eta(\cdot\mid \theta_i)}
    \left[
        c_i\bigl(a_i,f_{-i}(\theta_{-i}),\theta_i,\theta_{-i}\bigr)
    \right], \quad \forall \theta_i \in \Theta_{i},
    \label{eq:classical_BNE_problem}
\end{equation}
where the expectation is taken with respect to the conditional probability distribution $\eta(\cdot|\theta_i)$, and $f_{-i}(\theta_{-i})$ represents the rivals' action in scenario $\theta_{-i}\in \Theta_{-i}$. 
Assuming that each player chooses its optimal strategy by minimizing expected loss under the Nash conjecture (taking rivals' optimal response functions as given), we consider a situation in which no player can benefit by unilaterally deviating from its own strategy. 
This leads to the formal definition of a Bayesian Nash equilibrium.

\begin{definition}[Interim BNE]
\label{def:bne}
A pure-strategy interim Bayesian Nash equilibrium is an $n$-tuple
$f^*:=(f_1^*,\ldots,f_n^*)\in \mathcal{F}$, mapping from
$\Theta_1\times\cdots\times\Theta_n$ to $\mathcal{A}_1\times\cdots\times\mathcal{A}_n$,
such that for each $i\in N$ and each $\theta_i\in\Theta_i$,
\begin{equation}
    f_i^*(\theta_i)
    \in
    \arg\min_{a_i\in \mathcal{A}_i}
    \mathbb{E}_{\eta(\cdot\mid \theta_i)}
    \left[
        c_i\bigl(a_i,f_{-i}^*(\theta_{-i}),\theta_i,\theta_{-i}\bigr)
    \right].
    \label{eqn:BNE}
\end{equation}
\end{definition}

\subsection{Notation}

To facilitate reading, we list some  specific 
notation to be used throughout this paper 
in Table~\ref{tab:notation}.

\begin{table}[htbp]
\centering
\caption{Notation.}
\label{tab:notation}
\begin{tabular}{ll}
\hline
\textbf{Notation} & \textbf{Meaning} \\
\hline
$a_i\in\mathcal{A}_i$ 
& Action of player $i$ and its space, $a:=(a_1,\dots,a_n) \in \mathcal{A}:= \mathcal{A}_1\times \cdots \times \mathcal{A}_n$ \\
$\theta_i\in \Theta_i$ 
& Type of player $i$ and its space, $\theta :=(\theta_1,\dots,\theta_n)\in\Theta:= \Theta_1\times \cdots \times \Theta_n$ \\
$c_i(a_i,a_{-i},\theta_i,\theta_{-i})$ 
&Loss function of player $i$\\
$\zeta, \zeta^*\in \mathcal{Z}$ 
& Parameter of the joint distribution of $\theta$ and true parameter \\
$\eta(\cdot\,;\zeta)$, $p(\theta;\zeta)$ 
& Joint distribution of $\theta$ 
under parameter $\zeta$ and its density\\
$\eta_{i}(\cdot;\zeta)$, $p_i(\theta_i;\zeta)$ 
& Marginal distribution of $\theta_i$
and its density\\
$\eta(\cdot|\theta_i;\zeta)$, $p(\theta_{-i}\mid\theta_i;\zeta)$
& Conditional distribution of $\theta_{-i}$ given $\theta_i$
and its density\\
$\mu_i^t$, $m_i^t(\zeta)$
& Player $i$'s prior and posterior beliefs 
of $\zeta$
at round $t$ and its density\\
$\nu_i^t$, $q_i^t(\theta)$&
Bayesian predictive distribution of $\theta$ by player $i$ at round $t$ and its density\\
$\nu_{i,\theta_i}^t$, $q_{i,\theta_i}^t(\theta_i)$&
Marginal distribution 
of player $i$ under $\nu_i^t$  and its density\\
$\nu_{i}^t(\cdot|\theta_i)$, $q_i^t(\theta_{-i}|\theta_i)$&
Conditional distribution of $\theta_{-i}$ 
of player $i$ under $\nu_i^t$ and its density\\
$\rho_{\eta(\cdot|\theta_i;\zeta)}^{\rm al}$ & Risk measure for aleatoric uncertainty in $\theta_{-i}$ w.r.t. $\eta(\cdot|\theta_i;\zeta)$\\
$\rho_{\mu_i^t}^{\rm ep}$ & Risk measure for epistemic uncertainty in $\zeta$ w.r.t. $\mu_i^t$ \\
$\rho_{\nu_i^t(\cdot|\theta_i)}^{\rm bp}$ & Risk measure for uncertainties in both $\theta_{-i}$ and $\zeta$ w.r.t. $\nu_i^t(\cdot|\theta_i)$ \\
$\rho_{\nu_{i,\theta_i}^t}^{\rm ex}$ & Risk measure for ex-ante uncertainty in $\theta_i$ w.r.t. $\nu_{i,\theta_i}^t$ \\
$v_{i,\mu_i^t}^t(a_i,f_{-i},\theta_i)$ & 
$\rho_{\mu_i^t}^{\rm ep}
    \left(
        \rho_{\eta(\cdot\mid\theta_i;\zeta)}^{\rm al}
        \left( c_i\left(  a_i,  f_{-i}(\theta_{-i}),  \theta_i, \theta_{-i} \right) \right)
    \right)$\\
$\vartheta_{i,\mu_i^t}^t(f_{-i},\theta_i)$, $\mathcal A_{i,\mu_i^t}^{t*}(f_{-i},\theta_i)$
& Optimal value and the set of optimal solutions for $\min\limits_{a_i\in \mathcal{A}_i}v_{i,\mu_i^t}^t(a_i,f_{-i},\theta_i)$\\
$v_{i,\nu_i^t(\cdot|\theta_i)}^{\text{bp},t}(a_i,f_{-i},\theta_i)$ & 
$\rho_{i,\nu_i^t(\cdot|\theta_i)}^{\rm bp}
        \left( c_i\left(a_i, f_{-i}(\theta_{-i}), \theta_i, \theta_{-i} \right)
    \right)$\\
$\vartheta_{i,\nu_i^t(\cdot|\theta_i)}^{\text{bp},t}(f_{-i},\theta_i)$, $\mathcal A_{i,\nu_i^t(\cdot|\theta_i)}^{\text{bp},t*}$
& Optimal value and the set of optimal solutions for $\min\limits_{a_i\in \mathcal{A}_i}v_{i,\nu_i^t(\cdot|\theta_i)}^{\text{bp},t}(a_i,f_{-i},\theta_i)$\\
\hline
\end{tabular}
\end{table}

\section{Repeated Bayesian games and learning mechanism}
\label{sec:model}

In classical BNE models \eqref{eqn:BNE}, the joint probability distribution $\eta$ is assumed to be publicly known. 
Here, we consider the case that
the true distribution 
$\eta$ 
is unknown but can be learned from empirical data.

\subsection{BNE among risk-averse Bayesian learners}
We begin by specifying some important features of the game.

\begin{specification} 
\label{model:repeated-game}
The Bayesian game is repeatedly played under the following rules.
    
\begin{itemize}

\item[(a)] \textbf{Players.}
The games are played repeatedly among the same set of players. The cost
functions, action spaces, and type spaces remain unchanged across rounds; that
is,\linebreak
$
    c_i^t(a_i,a_{-i},\theta_i,\theta_{-i})
    =
    c_i(a_i,a_{-i},\theta_i,\theta_{-i}),\
    \mathcal{A}_i^t=\mathcal{A}_i,\ 
    \Theta_i^t=\Theta_i,
$
for all $i\in N$ and $t=1,2,\ldots$. 
The type profiles $\{\theta^t\}_{t\ge 1}$ are i.i.d. samples generated by a fixed but unknown joint distribution $\eta$.

\item[(b)] \textbf{Repetition of the game and observation.}
At the beginning of round $t$, each player makes a decision according to its present belief on the joint distribution of $\theta$.
A Bayesian Nash equilibrium 
of the one-shot game 
is reached. 
At the end of the round, 
the realized type profile $\theta^t$ is publicly observed by
all players.

\item[(c)] \textbf{Information about $\eta$.}
The true distribution of $\theta$ belongs to a known parametric family  $\{\eta(\cdot;\zeta):\zeta\in\mathcal{Z}\}$ with density functions $\{p(\theta;\zeta):\zeta\in\mathcal{Z}\}$. 
There exists a single true parameter $\zeta^*\in\mathcal{Z}$ such that
the true distribution is $\eta(\cdot;\zeta^*)$.
The parametric family is common knowledge, whereas the true value $\zeta^*$ is
unknown. 
For each $i\in N$ and $\zeta\in\mathcal{Z}$, whenever the marginal density of $\theta_i$,
$
    p_i(\theta_i;\zeta):=\int_{\Theta_{-i}} p(\theta_i,\theta_{-i};\zeta)\,d\theta_{-i}
    >0,
$
we write $\eta(\cdot\mid \theta_i;\zeta)$ for the conditional distribution of
$\theta_{-i}$ given $\theta_i$ under $\eta(\cdot;\zeta)$ with density function $p(\theta_{-i}|\theta_i;\zeta):= \frac{p(\theta_{i},\theta_{-i};\zeta)}{p_i(\theta_i;\zeta)}$.

\item[(d)] \textbf{Information learning and updating mechanism.}
Each player treats $\zeta$ as a random variable and assigns a prior distribution $\mu_i^1$ with density $m_i^1$ to it, based on its own available information or subjective judgment. 
After observing the sample history $\boldsymbol{\theta}^{t-1}:=(\theta^1,\ldots,\theta^{t-1}),$
player $i$ updates its posterior distribution to $\mu_i^t$, with density
$m_i^t$, by Bayes' formula:
\begin{equation}
\label{eqn:bayesian-update1}
    m_i^t(\zeta)
    =
    \frac{p(\theta^{t-1};\zeta)m_i^{t-1}(\zeta)}
    {\int_{\mathcal{Z}}p(\theta^{t-1};\xi)m_i^{t-1}(\xi)\,d\xi}
    =
    \frac{\prod_{\tau=1}^{t-1}p(\theta^\tau;\zeta)m_i^1(\zeta)}
    {\int_{\mathcal{Z}}\prod_{\tau=1}^{t-1}p(\theta^\tau;\xi)m_i^1(\xi)\,d\xi},
    \qquad t=2,3,\ldots,
\end{equation}
where
$L_{t-1}(\zeta):=\prod_{\tau=1}^{t-1}p(\theta^\tau;\zeta)$
is the likelihood of the observed sample history under parameter value
$\zeta$. The posterior beliefs may be heterogeneous across players because
their prior distributions may differ.

\item[(e)] \textbf{Objective of each player.} At each round, player $i$ faces two distinct sources of uncertainty:
aleatoric uncertainty in the rivals' current types $\theta_{-i}^t$ and epistemic uncertainty about the true parameter $\zeta^*$.
Players may take different risk attitudes toward the two sources of uncertainty.
To formalize this distinction, we assign player $i$ two risk measures:
$\rho^{\rm al}$ for aleatoric uncertainty and $\rho^{\rm ep}$ for epistemic uncertainty. 
Given the rivals' strategy profile $f_{-i}^t\in\mathcal{F}_{-i}$ and player $i$'s realized type
$\theta_i^t$, player $i$ solves the nested risk-averse optimization problem
\begin{equation}
\label{eqn:model-bne-he-ra}
    \min_{a_i\in\mathcal{A}_i}
    \rho_{\mu_i^t}^{\rm ep}
    \left(
        \rho_{\eta(\cdot\mid \theta_i^t;\zeta)}^{\rm al}
        \left(
            c_i\left(
                a_i,
                f_{-i}^t(\theta_{-i}),
                \theta_i^t,
                \theta_{-i}
            \right)
        \right)
    \right),
\end{equation}
where the inner risk measure is used to evaluate the random loss 
associated with the aleatoric uncertainty in $\theta_{-i}$,
whereas the outer risk measure is used to evaluate the epistemic uncertainty in $\zeta$.
To avoid distracting from the main modeling ideas, we assume throughout that 
the problem is well defined. 
\hfill $\Box$

\end{itemize}   
\end{specification}

The nested formulation in \eqref{eqn:model-bne-he-ra} can be read from inside to outside. 
If the true parameter $\zeta^*$ were known, then the only remaining uncertainty for player $i$, after observing $\theta_i^t$, would be the aleatoric uncertainty in $\theta_{-i}^t$. 
The corresponding aleatoric risk would be
\[
    \rho_{\eta(\cdot\mid \theta_i^t;\zeta^*)}^{\rm al}
    \left[
        c_i\left(
            a_i,
            f_{-i}^t(\theta_{-i}^t),
            \theta_i^t,
            \theta_{-i}^t
        \right)
    \right].
\]
When $\zeta^*$ is unknown, the aleatoric risk value depends on the candidate parameter $\zeta$, and hence becomes a random quantity under player $i$'s posterior belief $\mu_i^t$. 
The outer risk measure $\rho_{\mu_i^t}^{\rm ep}$ 
aggregates this parameter-dependent
aleatoric risk across candidate values of $\zeta$, and thereby represents player $i$'s attitude toward epistemic uncertainty.
By composing these two risk measures, we arrive at the full objective function in \eqref{eqn:model-bne-he-ra}, which provides a comprehensive measure of the total risk faced by the player.
This formulation is consistent with the Bayesian composite-risk perspective \citep{wu2018bayesian,lin2022bayesian,ma2024bayesian}.

For the Bayesian game among risk-averse Bayesian learners under
Model Specification~\ref{model:repeated-game}, we define the equilibrium concept as
follows.

\begin{definition} [Interim BNE-RABL]
\label{def:BNE-HB}
    A pure-strategy interim Bayesian Nash equilibrium among risk-averse Bayesian learners (interim BNE-RABL) at round $t$ 
    under the 
   profile of posterior distributions $\mu^t:=(\mu_1^t,\cdots,\mu_n^t)$ 
    is an $n$-tuple $f^{t*}:=(f_1^{t*},\dots,f_n^{t*})\in\mathcal{F}$ mapping from $\Theta_1\times\cdots\times\Theta_n$ to $\mathcal{A}_1\times\cdots\times\mathcal{A}_n$ such that for each $i\in N$ and fixed $\theta_i\in \Theta_i$,
    \begin{equation}
\label{eqn:BNE-heter-belief}
    f_i^{t*}(\theta_i)
    \in
    \arg\min_{a_i\in\mathcal{A}_i}
    \rho_{\mu_i^t}^{\rm ep}
    \left(
        \rho_{\eta(\cdot\mid \theta_i;\zeta)}^{\rm al}
        \left(
            c_i\left(
                a_i,
                f_{-i}^{t*}(\theta_{-i}),
                \theta_i,
                \theta_{-i}
            \right)
        \right)
    \right).
\end{equation}
\end{definition}
The BNE-RABL model 
differs from the classical interim BNE model by (a) incorporating a Bayesian 
learning structure for updating 
each player's belief about $\zeta^*$
and (b) considering players' heterogeneous risk attitudes against
two sources of uncertainty. 
If the true parameter $\zeta^*$ is known, 
and each player is risk-neutral with respect to the aleatoric uncertainty of its rivals' 
type parameters, then
 \eqref{eqn:BNE-heter-belief} reduces to the classical interim BNE in Definition~\ref{def:bne} under the true distribution $\eta(\cdot;\zeta^*)$.
The new model effectively 
extends the classical  BNE models, which are primarily based on von Neumann--Morgenstern's expected utility theory, to a broader class of risk-sensitive strategic interactions.

\subsection{BNE based on Bayesian Predictive Distributions}
\label{sec:model-bpd}

The BNE-RABL model in Definition~\ref{def:BNE-HB} describes a player who explicitly distinguishes between aleatoric and epistemic uncertainty. 
The nested structure, however, is computationally demanding. 
Direct solution methods based on discretizing type spaces and using step-like strategies, such as those in \cite{athey2001single,liu2025bayesian}, suffer from the curse of dimensionality: the number of decision variables grows rapidly with the number of players, the dimension of the type vector, and the number of grid points. 
To simplify the structure of uncertainty and reduce the associated simulation cost, we introduce a new equilibrium concept based on Bayesian predictive distributions (BPD).

At the beginning of round $t$, player $i$ has a belief $\mu_i^t$ about $\zeta^*$ with density $m_i^t$
as shown in Model~Specification~\ref{model:repeated-game}~(c)-(d).
Based on $\mu_i^t$, player $i$ forms the BPD of the current type profile $\theta$ by averaging the parametric joint density over
$\zeta$:
\begin{equation}
\label{eqn:bpd-def}
    q_i^t(\theta) := \int_{\mathcal{Z}} p(\theta;\zeta)m_i^t(\zeta)\,d\zeta.
\end{equation}
Let $\nu_i^t$ denote the corresponding 
probability distribution
which synthesizes 
two 
sources of uncertainty into a single predictive distribution.

Under this setup, 
since player $i$ observes its own type before choosing an action, the relevant interim belief is the conditional distribution of $\theta_{-i}$ given $\theta_i$ under the BPD $\nu_i^t$ with density
\begin{equation}
\label{eqn:conditional-bpd}
    q_i^t(\theta_{-i}\mid \theta_i)
    :=
    \frac{q_i^t(\theta_i,\theta_{-i})}
    {\int_{\Theta_{-i}}q_i^t(\theta_i,\theta_{-i})d\theta_{-i}}.
\end{equation}
We denote 
the corresponding
conditional distribution 
by $\nu_i^t(\cdot|\theta_i)$.
Based on this interim belief, player $i$ evaluates the total predictive uncertainty about $\theta_{-i}$ through a single risk measure $\rho^{\rm bp}$ with respect to $\nu_i^t(\cdot\mid\theta_i)$. 
Therefore, given the rivals' strategy profile
$f_{-i}^t\in\mathcal{F}_{-i}$ and player $i$'s realized type $\theta_i$,
the interim decision problem is
\begin{equation}
\label{eqn:model-bpd-ra}
    \min_{a_i\in\mathcal{A}_i}
    \rho_{\nu_i^t(\cdot\mid\theta_i)}^{\rm bp}
    \left(
        c_i\left(
            a_i,
            f_{-i}^t(\theta_{-i}),
            \theta_i,
            \theta_{-i}
        \right)
    \right).
\end{equation}
Instead of computing a nested risk as $\rho_{\mu_i^t}^{\text{ep}} \circ \rho_{\eta(\cdot|\theta_i;\zeta)}^{\text{al}}$ in \eqref{eqn:model-bne-he-ra}, we integrate out the parametric uncertainty over $\zeta$ 
and consider a single risk measure in problem \eqref{eqn:model-bpd-ra} based on the Bayesian predictive distribution \eqref{eqn:bpd-def}. %
This prompts us to study the Bayesian Nash equilibrium with Bayesian predictive distribution (BNE-BPD) 
defined as follows.

\begin{definition} [Interim BNE-BPD]
\label{def:BNE-BPD}
    A pure-strategy interim BNE-BPD at round $t$ under a profile of posterior distributions $\mu^t:=(\mu_1^t,\cdots,\mu_n^t)$ 
    is an $n$-tuple $f^{t*}:=(f_1^{t*},\dots,f_n^{t*})\in\mathcal{F}$ mapping from $\Theta_1\times\cdots\times\Theta_n$ to $\mathcal{A}_1\times\cdots\times\mathcal{A}_n$ such that for each $i\in N$ 
    \begin{equation}
\label{eqn:BNE-BPD}
    f_i^{t*}(\theta_i)
    \in \arg\min_{a_i\in\mathcal{A}_i}
    \rho_{\nu_i^t(\cdot\mid\theta_i)}^{\rm bp}
    \left( c_i \left( a_i, f_{-i}^{t*}(\theta_{-i}), \theta_i, \theta_{-i} \right) \right),
\end{equation}
    where $\nu_i^t, i=1,\dots,n$ are defined by \eqref{eqn:bpd-def}.
\end{definition}

In the BNE-RABL model \eqref{eqn:BNE-heter-belief}, 
player $i$ 
tackles  
the risk of the epistemic uncertainty $\zeta$ and the risk of aleatoric uncertainty $\theta_{-i}$ separately.
By contrast, player $i$ uses
the posterior distribution 
of $\zeta$ to construct 
a single predictive joint distribution of $\theta$ and then applies a single risk measure to the combined predictive uncertainty in $\theta_{-i}$.
We elaborate on the connections and differences between the two equilibrium models in the following remark.

\begin{remark} \label{remark:connection-difference}
\begin{itemize}
\item[(i)]
Intuitively, 
the BNE-BPD model provides a surrogate for the BNE-RABL model when epistemic uncertainty is small.
Indeed, if the posterior distribution $\mu_i^t$ concentrates around the true parameter
$\zeta^*$, 
that is, the posterior variance of $\zeta$ is small,
then the predictive joint distribution $\nu_i^t$ is close to
the true joint distribution $\eta(\cdot;\zeta^*)$, and hence the induced conditional
distribution $\nu_i^t(\cdot \mid \theta_i)$ is close to the true conditional
distribution $\eta(\cdot \mid \theta_i;\zeta^*)$.
If, in addition, the risk measure $\rho^{\rm bp}$ in BNE-BPD is chosen to coincide with
the aleatoric risk measure $\rho^{\rm al}$ in BNE-RABL, then both BNE-BPD and BNE-RABL
converge to the true BNE with risk-averse objectives as $t\to\infty$.
Conversely, when the posterior variance of $\zeta$ remains large, $\nu_i^t$
may deviate substantially from $\eta(\cdot;\zeta^*)$, 
in which case
both BNE-BPD and BNE-RABL may
deviate significantly from the true BNE.
We return to this issue in Section~\ref{sec:asymptotic-converge}.

    \item[(ii)]
    Nevertheless, the BNE-BPD and the BNE-RABL generally differ.
To see the fundamental 
difference between the two models, 
we consider the risk-neutral case.
The objective function in \eqref{eqn:BNE-heter-belief} can be written as
\begin{eqnarray}
\label{eqn:rabl-risk-neutral-objective}
&&\mathbb{E}_{\mu_i^t}
    \left[
        \mathbb{E}_{\eta(\cdot\mid \theta_i;\zeta)}
        \left[
            c_i\left(
                a_i,
                f_{-i}^{t}(\theta_{-i}),
                \theta_i,
                \theta_{-i}
            \right)
        \right]
    \right]\nonumber\\
    &&\quad =\int_{\mathcal Z}
    \left(
        \int_{\Theta_{-i}}
            c_i\bigl(a_i,f_{-i}^t(\theta_{-i}),\theta_i,\theta_{-i}\bigr)
            \frac{p(\theta_i,\theta_{-i};\zeta)}
            {\int_{\Theta_{-i}} p(\theta_i,\theta_{-i};\zeta)\,d\theta_{-i}}
        \,d\theta_{-i}
    \right)
    m_i^t(\zeta)\,d\zeta.
\end{eqnarray}
Under some moderate conditions,
\eqref{eqn:rabl-risk-neutral-objective} can be reformulated by Fubini's theorem as
\bgeqn 
\label{eqn:bpd-conditional}
    \int_{\Theta_{-i}}
        c_i\bigl(a_i,f_{-i}^t(\theta_{-i}),\theta_i,\theta_{-i}\bigr)
        \widetilde q_i^t(\theta_{-i}\mid\theta_i)
    \,d\theta_{-i},
\edeqn 
where
\begin{eqnarray}
\label{eqn:rabl-conditional}
   \widetilde q_i^t(\theta_{-i}\mid\theta_i):= \int_{\mathcal Z}
        \frac{p(\theta_i,\theta_{-i};\zeta)}
        {\int_{\Theta_{-i}} p(\theta_i,\theta_{-i};\zeta)\,d\theta_{-i}}
        m_i^t(\zeta)\,d\zeta = \int_{\mathcal Z}
        p(\theta_{-i}|\theta_i; \zeta)
        m_i^t(\zeta)\,d\zeta.
\end{eqnarray}
By contrast, 
the objective function 
in the BNE-BPD model \eqref{eqn:BNE-BPD}
can be written as 
\begin{equation}
\label{eqn:bpd-risk-neutral-objective}
    \mathbb{E}_{\nu_i^t(\cdot\mid\theta_i)}
    \left[ c_i \left( a_i, f_{-i}^{t}(\theta_{-i}), \theta_i, \theta_{-i} \right) \right] =\int_{\Theta_{-i}}
        c_i\bigl(a_i,f_{-i}^t(\theta_{-i}),\theta_i,\theta_{-i}\bigr)
        q_i^t(\theta_{-i}\mid\theta_i)
    \,d\theta_{-i}.
\end{equation}
By \eqref{eqn:bpd-def}-\eqref{eqn:conditional-bpd},
\begin{eqnarray}
\label{eqn:bpd-condition}
q_i^t(\theta_{-i}\mid\theta_i)
& = & 
    \frac{\int_{\mathcal Z}p(\theta_i,\theta_{-i};\zeta)m_i^t(\zeta)\,d\zeta} 
    {\int_{\Theta_{-i}} \int_{\mathcal Z} p(\theta_i,\theta_{-i};\zeta) m_i^t(\zeta) d\zeta d\theta_{-i}} \nonumber\\
&=& \int_\mathcal{Z} p(\theta_{-i}\mid\theta_i;\zeta)
\frac{p_i(\theta_i;\zeta)}
{\int_\mathcal{Z}\left[\int_{\Theta_{-i}}
p(\theta_i,\theta_{-i};\zeta)d\theta_{-i}\right]m_i^t(\zeta)d\zeta}
m_i^t(\zeta)d\zeta \nonumber\\
&=& \int_\mathcal{Z} p(\theta_{-i}\mid\theta_i;\zeta)
\frac{p_i(\theta_i;\zeta)}
{\int_\mathcal{Z} p_i(\theta_i;\zeta)m_i^t(\zeta)d\zeta}
m_i^t(\zeta)d\zeta.
\end{eqnarray}
Comparing the right-hand side (rhs) 
of \eqref{eqn:rabl-conditional} and the rhs
of \eqref{eqn:bpd-condition},
we see that 
both expressions calculate 
an average of $p(\theta_{-i}\mid\theta_i;\zeta)$
with respect to the posterior 
density $m_i^t$. However, the latter is reweighted by 
$\frac{p_i(\theta_i;\zeta) }{\int_\mathcal{Z}p_i(\theta_i;\zeta) m_i^t(\zeta)d\zeta}$.
Intuitively, this likelihood reweighting uses the player's own type $\theta_i$ as an additional signal about $\zeta$, and assigns greater weight to parameter values under which the observed $\theta_i$ is more likely and thereby alters the inferred distribution of the rivals' types.

(iii)
Evaluating the objective function of the interim BNE-RABL in Definition~\ref{def:BNE-HB} requires a two-stage simulation: 
for each player $i$ and each type $\theta_i$, one must first sample the parameter $\zeta$ of the underlying joint distribution from the posterior $\mu_i^t$, 
and then, 
for each pair $(\theta_i,\zeta)$, sample the rivals' types $\theta_{-i}$ from $\eta(\cdot\mid\theta_i;\zeta)$.
The interim BNE-BPD in Definition~\ref{def:BNE-BPD} avoids this nested procedure:
for each player $i$ and each type $\theta_i$, it suffices to sample the rivals' types $\theta_{-i}$ directly from the Bayesian predictive distribution $\nu_i^t(\cdot\mid\theta_i)$. 
This formulation greatly simplifies the simulation process, 
especially when $\nu_i^t(\cdot\mid\theta_i)$ admits a closed form.
In Section~\ref{sec:ex-ante}, we incorporate the ex-ante uncertainty in $\theta_i$ and derive a reformulation of the interim BNE-BPD that can further simplify the simulation process.

\end{itemize}
\end{remark}

\section{Ex-ante equilibrium in the BNE-BPD model}

\label{sec:ex-ante}

In Section \ref{sec:model}, we focus on the interim stage, where each player observes their own type parameter and exhibits risk aversion with respect to the uncertainty surrounding the rival players' type parameters. 
It is possible to develop an ex-ante model before players' observation of their own type parameters. Here, we 
provide a sketch of the ex-ante BNE model.
We concentrate on BNE-BPD, as a similar 
approach can be 
developed for BNE-RABL.

\subsection{Ex-ante BNE-BPD and its relationship with interim BNE-BPD}

Recall that at the interim stage, the problem faced by player $i$ in BNE-BPD~\eqref{eqn:BNE-BPD} is formulated in terms of the conditional distribution given the observed own-type parameter $\theta_i\in \Theta_i$, namely,
\begin{equation*}
    \min_{a_i\in\mathcal{A}_i}
    \rho_{\nu_i^t(\cdot\mid\theta_i)}^{\rm bp}
    \left( c_i \left( a_i, f_{-i}^{t*}(\theta_{-i}), \theta_i, \theta_{-i} \right) \right).
\end{equation*}
The corresponding optimal value characterizes the minimal risk that player $i$ can attain at type $\theta_i$ by adopting the equilibrium strategy $f_i^{t*}(\theta_i)$.
Since this minimal risk value depends on the player's own type parameter $\theta_i$, it may be regarded as a random function that maps the player's own (random) type to the risk induced by the uncertainty over the rivals' types, that is,
\begin{eqnarray*}
    \theta_i \mapsto \rho_{\nu_i^t(\cdot\mid\theta_i)}^{\rm bp}
    \left( c_i \left( f_i^{t*}(\theta_i), f_{-i}^{t*}(\theta_{-i}), \theta_i, \theta_{-i} \right) \right).
\end{eqnarray*}
Accordingly, the player seeks to assess the overall risk of the equilibrium strategy across all possible realizations of its own type via
\begin{eqnarray}
\label{eqn:nested-1}
    \rho_{\nu_{i,\theta_i}^t}^{\rm ex} \left( \rho_{\nu_i^t(\cdot\mid\theta_i)}^{\rm bp}
    \left( c_i \left( f_i^{t*}(\theta_i), f_{-i}^{t*}(\theta_{-i}), \theta_i, \theta_{-i} \right) \right)
    \right),
\end{eqnarray}
where $\nu_{i,\theta_i}^t$ denotes the marginal distribution of $\theta_i$ induced by the Bayesian predictive distribution $\nu_i^t$ defined in \eqref{eqn:bpd-def}, and $\rho_{\nu_{i}^t}^{\rm ex}$ is a law-invariant risk measure that captures the ex-ante uncertainty of $\theta_i$.
Based on the risk evaluated in the ex-ante stage \eqref{eqn:nested-1}, we define the ex-ante counterpart of interim BNE-BPD as follows.

\begin{definition} [Ex-ante BNE-BPD]
\label{def:exante-bnebpd}
    A pure-strategy ex-ante Bayesian Nash equilibrium with Bayesian predictive distribution (ex-ante BNE-BPD) at round $t$ is an $n$-tuple $f^{t*}:=(f_1^{t*},\dots,f_n^{t*})\in\mathcal{F}$ mapping from $\Theta_1\times\cdots\times\Theta_n$ to $\mathcal{A}_1\times\cdots\times\mathcal{A}_n$ such that for each $i\in N$ 
    \begin{equation}
\label{eqn:exante-BNE-BPD}
    f_i^{t*}
    \in \arg\min_{f_i\in\mathcal{F}_i}
    \rho_{\nu_{i,\theta_i}^t}^{\rm ex} \left( \rho_{\nu_i^t(\cdot\mid\theta_i)}^{\rm bp}
    \left( c_i \left( f_i(\theta_i), f_{-i}^{t*}(\theta_{-i}), \theta_i, \theta_{-i} \right) \right)
    \right).
\end{equation}
    where $\nu_i, i=1,\dots,n$ are defined in \eqref{eqn:bpd-def}.
\end{definition}

The next theorem establishes that the interim BNE-BPD in Definition~\ref{def:BNE-BPD}, 
obtained by solving a series of decision-making problems~\eqref{eqn:BNE-BPD} for all realizations of $\theta_i$, 
is equivalent to the ex-ante BNE-BPD in Definition~\ref{def:exante-bnebpd}, obtained by solving a single optimization problem for each player whose decision variable is the strategy function as formulated in~\eqref{eqn:exante-BNE-BPD}, provided that both equilibria exist. We will discuss the existence and uniqueness of the equilibria in Section~\ref{sec:exist-unique}.

\begin{theorem} [Relationship between ex-ante BNE-BPD and interim BNE-BPD]
\label{thm:reformulation-bpd}
    Suppose that Assumption~\ref{basic-assumption} holds
    and $\rho^{\rm{ex}}$ is strictly monotonic. 
    Then $f^{t*}$ is almost surely a measurable interim BNE-BPD  if and only if 
    it is an ex-ante BNE-BPD, i.e., for $i\in N$,
    \begin{eqnarray} \label{eqn:equivalent-BNE1}
        f_i^{t*}\in \mathop{\arg\min}_{f_i\in \mathcal{F}_i} \rho_{\nu_{i,\theta_i}^t}^{\rm{ex}}\left(
    \rho_{\nu_i^t(\cdot|\theta_i)}^{\rm{bp}} \left( c_{i}\left(f_{i}(\theta_i),f_{-i}^{t*}(\theta_{-i}),\theta_{i},\theta_{-i}\right) \right)\right),
    \end{eqnarray}
    or equivalently 
    \begin{eqnarray} \label{eqn:equivalent-BNE2}
        f^{t*}\in \arg\min_{f\in \mathcal{F}}\sum_{i=1}^n \rho_{\nu_{i,\theta_i}^t}^{\rm{ex}}\left(
    \rho_{\nu_i^t(\cdot|\theta_i)}^{\rm{bp}} \left( c_{i}\left(f_{i}(\theta_i),f_{-i}^{t*}(\theta_{-i}),\theta_{i},\theta_{-i}\right) \right)\right).
    \end{eqnarray}
    Moreover, $f^{t*}$ is a continuous interim BNE-BPD if and only if it
is a continuous ex-ante BNE-BPD, i.e., \eqref{eqn:equivalent-BNE1} and
\eqref{eqn:equivalent-BNE2} hold with the feasible sets $\mathcal{F}_i$ and
$\mathcal{F}$ replaced by the continuous strategy spaces $\mathcal{C}_i$ and $\mathcal{C}$, respectively.
\end{theorem}

\noindent
\textbf{Proof.}
Optimality condition \eqref{eqn:equivalent-BNE2} is immediately derived by \eqref{eqn:equivalent-BNE1}. 
We only prove the equivalence between the interim BNE-BPD and ex-ante BNE-BPD. 

We first establish the equivalence for measurable BNE-BPD. 
For the ``only if'' part,  $f^{t*}$ is almost surely an interim BNE-BPD, i.e., for $i\in N$ and for $\theta_i\in \Theta_i$ almost surely,
\begin{eqnarray}\label{proof:equivalent-1}
    f_i^{t*}(\theta_i)\in \arg\min_{a_i\in \mathcal{A}_i} 
    \rho_{\nu_i^t(\cdot|\theta_i)}^{\rm{bp}} \left( c_{i}\left(a_i,f_{-i}^{t*}(\theta_{-i}),\theta_{i},\theta_{-i}\right) \right).
\end{eqnarray}  
The monotonicity of $\rho_{\nu_{i,\theta_i}^t}^{\rm{ex}}$ implies that $f_i^{t*}$ is also an optimal solution to  
\begin{eqnarray} \label{proof:equivalent-2}
    \min_{f_i\in \mathcal{F}_i} \rho_{\nu_{i,\theta_i}^t}^{\rm{ex}} \left(
    \rho_{\nu_i^t(\cdot|\theta_i)}^{\rm{bp}} \left( c_{i}\left(f_{i}(\theta_i),f_{-i}^{t*}(\theta_{-i}),\theta_{i},\theta_{-i}\right) \right)\right),
\end{eqnarray}
and thus $f^{t*}$ is an ex-ante BNE-BPD.

Conversely, for the ``if'' part, $f^{t*}$ is an ex-ante BNE-BPD, i.e., $f_i^{t*}$ is the optimal solution to \eqref{proof:equivalent-2}.
By Proposition 2.1 in \cite{shapiro2017interchangeability},
the strict monotonicity implies that,
for $\theta_i\in \Theta_i$ almost surely,  \eqref{proof:equivalent-1} holds.
Therefore, $f^{t*}$ is almost surely an interim BNE-BPD.

We next turn to the continuous BNE-BPD. 
The ``only if'' part follows by the same argument as in the measurable case.
It remains to prove the ``if'' part.
Assume for the sake of a contradiction that $f^{t*}$ is a continuous ex-ante BNE-BPD but not a continuous interim BNE-BPD. 
Then, there exist some $i\in N$ and $\tilde{f}_i^t\in \mathcal{C}_i$ such that for some $\tilde{\theta}_i\in \Theta_i$,
\begin{eqnarray}
    \rho_{\nu_i^t(\cdot|\tilde{\theta}_i)}^{\rm{bp}} \left( c_{i}\left(\tilde{f}_i^{t}(\tilde{\theta}_i),f_{-i}^{t*}(\theta_{-i}),\tilde{\theta}_{i},\theta_{-i}\right) \right)
    < \rho_{\nu_i^t(\cdot|\tilde{\theta}_i)}^{\rm{bp}} \left( c_{i}\left(f_i^{t*}(\tilde{\theta_i}),f_{-i}^{t*}(\theta_{-i}),\tilde{\theta}_{i},\theta_{-i}\right) \right).
\end{eqnarray}
Since $c$ is continuous on $\mathcal{A}_i\times \mathcal{A}_{-i}\times \Theta_i\times \Theta_{-i}$, there exists a neighborhood $\mathbb{B}(\tilde{\theta}_i)\subset \Theta_i$ of $\tilde{\theta}_i$ such that 
\begin{eqnarray} \label{eqn:proof-equivalent}
    \rho_{\nu_i^t(\cdot|{\theta}_i)}^{\rm{bp}} \left( c_{i}\left(\tilde{f}_i^{t}({\theta}_i),f_{-i}^{t*}(\theta_{-i}),\theta_{i},\theta_{-i}\right) \right)
    < \rho_{\nu_i^t(\cdot|{\theta}_i)}^{\rm{bp}} \left( c_{i}\left(f_i^{t*}({\theta}_i),f_{-i}^{t*}(\theta_{-i}),\theta_{i},\theta_{-i}\right) \right), \forall \theta_i \in \mathbb{B}(\tilde{\theta}_i).
\end{eqnarray}
Consequently, we can construct a continuous function $\hat{f}_i^t(\theta_i)$ such that $\hat{f}_i^t(\theta_i) =
f_i^{t*}(\theta_i)$ for $\theta_i \notin \mathbb{B}(\tilde{\theta}_i)$
and $\hat{f}_i^{t}(\theta_i) =
\tilde{f}_i^t(\theta_i)$ otherwise. Note that in general 
the function values of $\tilde{f}_i^t$ do not necessarily 
meet those of $f^{t*}_i$ at the boundaries of $\mathbb{B}(\tilde{\theta}_i)$, in which case, we may revise the function values of $\hat{f}_i^t$ near the boundaries within 
$\mathbb{B}(\tilde{\theta}_i)$ such that the inequality \eqref{eqn:proof-equivalent} is preserved.
By the strict monotonicity of $\rho^{\rm{ex}}$, we have 
\begin{eqnarray*}
    \rho_{\nu_{i,\theta_i}^t}^{\rm{ex}}\left(
    \rho_{\nu_i^t(\cdot|\theta_i)}^{\rm{bp}} \left( c_{i}\left(\hat{f}_{i}^t(\theta_i),f_{-i}^{t*}(\theta_{-i}),\theta_{i},\theta_{-i}\right) \right)\right) 
    < \rho_{\nu_{i,\theta_i}^t}^{\rm{ex}}\left(
    \rho_{\nu_i^t(\cdot|\theta_i)}^{\rm{bp}} \left( c_{i}\left({f}_{i}^{t*}(\theta_i),f_{-i}^{t*}(\theta_{-i}),\theta_{i},\theta_{-i}\right) \right)\right),
\end{eqnarray*}
which contradicts the assumption that $f^{t*}$ is a continuous ex-ante equilibrium.
Hence, any continuous ex-ante BNE-BPD is also a continuous interim BNE-BPD.
\hfill $\Box$

In the ex-ante BNE-BPD problem in Definition~\ref{def:exante-bnebpd}, 
each player clearly separates the risk arising from the ex-ante uncertainty of its own type
and that arising from the uncertainty of its rivals' types. 
Crucially, the player's risk attitude toward its rivals' types is prescribed as $\rho^{\rm bp}$ for every realization of its own type. 
Therefore, for a fixed rivals' strategy profile $f_{-i}$, once its own type $\theta_i$ is realized, the player's action is determined solely by this prespecified inner risk measure $\rho^{\rm bp}$.
The role of the outer risk measure is merely to aggregate, at the ex-ante
stage, the risks over all potential realizations of $\theta_i$. 
As long as the outer risk measure faithfully incorporates the risk induced by the actions under each realized type, which is guaranteed by strict monotonicity,
it preserves the optimality of the action chosen at each type almost surely. 
Hence the strictly monotonic outer risk measure has no actual effect on the players' response functions, and thus none on the resulting BNE-BPD.

By contrast,
\cite{liu2026games} studies risk-averse Bayesian games through an ex-ante formulation, but proceeds in the opposite direction.
He takes an ex-ante joint risk measure as the primitive object and derives a type-dependent interim risk measure after the player's own type is observed.
In his model, the players do not distinguish between uncertainty of their own type and that of their rivals' types at the ex-ante stage. 
Each player adopts a single joint risk measure
$\rho^{\rm joint}$ to evaluate the risk arising from both sources of uncertainty simultaneously, 
leading to the primitive objective
\begin{eqnarray} \label{eqn:liu-exante}
    \rho_{\eta}^{\rm joint}\!\left(c\bigl(f_i(\theta_i),f_{-i}(\theta_{-i}),\theta_i,\theta_{-i}\bigr)\right),
\end{eqnarray}
where $\eta$ is the joint distribution of the type profile $(\theta_1,\dots,\theta_n)$.
Once its own type is realized, the player revises this risk attitude in the interim stage according to the realized type $\theta_i$ as a type-dependent conditional risk measure $\rho^{{\rm re},\theta_i}$,
and then considers the interim objective for each type $\theta_i$,
\begin{eqnarray}\label{eqn:liu-interim}
    \rho_{\eta(\cdot\mid\theta_i)}^{{\rm re},\theta_i}\!\left(c\bigl(f_i(\theta_i),f_{-i}(\theta_{-i}),\theta_i,\theta_{-i}\bigr)\right),
\end{eqnarray}
where the revised measure $\rho^{{\rm re},\theta_i}$ may vary across realizations of $\theta_i$ and $\eta(\cdot\mid\theta_i)$ denotes the conditional distribution of $\theta_{-i}$. 
\cite{liu2026games} shows that, 
when the players revise their risk attitudes in the interim stage using the decomposition theorem for risk functionals in \cite{pflug2016time}, the ex-ante equilibrium induced by \eqref{eqn:liu-exante}
and the interim equilibrium induced by \eqref{eqn:liu-interim} coincide under suitable conditions.
In their model, the ex-ante formulation with risk measure $\rho^{\rm joint}$ for joint uncertainty is thus the primitive object, and the interim equilibrium arises as its type-contingent revision.

Our model differs from \cite{liu2026games} in both modeling and conceptual perspectives. On the modeling side, each player clearly separates the uncertainty of its own type from that of its rivals' types and 
assigns a distinct risk measure to each, rather than aggregating them through a single joint measure.
Since each player's action is taken at the interim stage where its own type has been observed, 
the strategy is primarily a response to the remaining uncertainty arising from its rivals' types.
On the conceptual side, we take the standard interim equilibrium as the primitive equilibrium concept, 
and introduce the ex-ante formulation only to evaluate the risk of the resulting type-contingent equilibrium strategy before the player's own type is realized. 
Therefore, our ex-ante equilibrium model
is
not intended to revise the interim equilibrium, but
provides a formulation that
preserves the interim equilibrium almost surely under strict monotonicity.

\subsection{Further simplifications under appropriate choices of risk measures}
\label{sec:further-simplications}

Given the equivalence between interim BNE-BPD and ex-ante BNE-BPD established in 
Theorem~\ref{thm:reformulation-bpd},
we can 
next derive 
a reformulation 
of the latter, which may 
simplify
the simulation procedure under appropriate choices of risk measures.
For instance, 
the nested composition of risk measures $\rho_{\nu_{i,\theta_i}^t}^{\rm ex}\circ\rho_{\nu_i^t(\cdot\mid\theta_i)}^{\rm bp}$
may be consolidated into a single expectation taken over the profile of all players' types. 
This will
facilitate 
a 
single-stage simulation
where 
the entire type profile $\theta$ is sampled at once from the Bayesian predictive distribution $\nu_i^t$, and
avoid sampling the rivals' types $\theta_{-i}$ for each $\theta_i$.

Consider the case where $\rho_{\nu_{i,\theta_i}^t} = \mathbb{E}_{\nu_{i,\theta_i}^t}$ and $\rho_{\nu_i^t(\cdot|\theta_i)}^{\rm{bp}}$ takes the
parametric form \eqref{eqn:parametric-risk-measure},
which covers many important risk measures as 
discussed in Section \ref{sec:preliminary}.
We have the following equivalence result.

\begin{corollary} \label{corollary:reformuation-parametric}
Consider
the setting of 
Theorem~\ref{thm:reformulation-bpd} and 
the specific case that $\rho_{\nu_{i,\theta_i}^t}^{\rm{ex}} = \mathbb{E}_{\nu_{i,\theta_i}^t}$, for $i\in N$.
    Then the following assertions hold.
    \begin{itemize}
        \item[(i)]$f^{t*}$ is almost surely a measurable interim BNE-BPD if and only if 
    it is an ex-ante BNE-BPD, 
    i.e., for $i\in N$,
    \begin{eqnarray} \label{eqn:equivalent-BNE-expectation1}
        f_i^{t*}\in \mathop{\arg\min}_{f_i\in \mathcal{F}_i} \mathbb{E}_{\nu_{i,\theta_i}^t} \left[
    \rho_{\nu_i^t(\cdot|\theta_i)}^{\rm{bp}} \left( c_{i}\left(f_{i}(\theta_i),f_{-i}^{t*}(\theta_{-i}),\theta_{i},\theta_{-i}\right) \right)\right].
    \end{eqnarray}

        \item[(ii)] Assume in addition that 
$\rho^{\rm{bp}}(\xi):=\inf_{s \in \mathcal{S}}\mathbb{E}[\Psi(\xi,s)]$.
Then 
$f^{t*}$ is almost surely a measurable interim BNE-BPD
if and only if,
for $i\in N$,
\begin{eqnarray} \label{eqn:equivalent-BNE-expectation2}
    f_i^{t*}\in \mathop{\arg\min}_{f_i\in \mathcal{F}_i} \inf_{\varphi(\cdot) \in \mathcal{M}_\mathcal{S}} \mathbb{E}_{\nu_{i}^t} \left[ \Psi \left( c(f_i(\theta_i), f_{-i}^{t*}(\theta_{-i}),\theta_i, \theta_{-i}), \varphi(\theta_i) \right) \right],
\end{eqnarray}
where $\mathcal{M}_\mathcal{S}$ is the set of measurable functions from $\Theta_i$ to $\mathcal{S}$.

    \end{itemize}
\end{corollary}

\noindent
\textbf{Proof.} Part (i). 
Since 
the mathematical expectation operator is strictly monotonic, 
the equivalence follows from Theorem \ref{thm:reformulation-bpd}.

Part (ii). With the specific form of $\rho^{\rm{bp}}(\xi)$, 
we can rewrite
\eqref{eqn:equivalent-BNE-expectation1}  as
\begin{eqnarray} \label{proof:equivalent-BNE-expectation1}
    f_i^{t*}\in \mathop{\arg\min}_{f_i\in \mathcal{F}_i} \mathbb{E}_{\nu_{i,\theta_i}^t} \left[
    \inf_{s \in \mathcal{S}} \mathbb{E}_{\nu_i^t(\cdot|\theta_i)} \left[ \Psi\left( c_{i}\left(f_{i}(\theta_i),f_{-i}^{t*}(\theta_{-i}),\theta_{i},\theta_{-i}\right),s \right)\right]\right].
\end{eqnarray}
By the interchangeability principle (see, e.g., Theorem~2.2 in \cite{hiai1977integrals},
Theorem~9.108 in \cite{shapiro2014lectures},
or Lemma~3 in \cite{liu2025preference}),
\begin{eqnarray} \label{proof:equivalent-BNE-expectation2}
    && \mathbb{E}_{\nu_{i,\theta_i}^t} \left[
    \inf_{s \in \mathcal{S}} \mathbb{E}_{\nu_i^t(\cdot|\theta_i)} \left[ \Psi\left( c_{i}\left(f_{i}(\theta_i),f_{-i}^{t*}(\theta_{-i}),\theta_{i},\theta_{-i}, s\right) \right)\right]\right] \nonumber\\
    &=& \inf_{\varphi(\cdot) \in \mathcal{M}_\mathcal{S}} \mathbb{E}_{\nu_{i,\theta_i}^t} \left[ \mathbb{E}_{\nu_i^t(\cdot|\theta_i)} \left[ \Psi\left( c_{i}\left(f_{i}(\theta_i),f_{-i}^{t*}(\theta_{-i}),\theta_{i},\theta_{-i}, \varphi(\theta_i)\right) \right)\right]\right] \nonumber\\
    &= &\inf_{\varphi(\cdot) \in \mathcal{M}_\mathcal{S}} \mathbb{E}_{\nu_{i}^t} \left[  \Psi\left( c_{i}\left(f_{i}(\theta_i),f_{-i}^{t*}(\theta_{-i}),\theta_{i},\theta_{-i}, \varphi(\theta_i)\right) \right)\right].
\end{eqnarray}
Substituting \eqref{proof:equivalent-BNE-expectation2} into \eqref{proof:equivalent-BNE-expectation1}, we 
obtain
\eqref{eqn:equivalent-BNE-expectation2}.
The equivalence argument follows from Part~(i).
\hfill $\Box$

Corollary~\ref{corollary:reformuation-parametric} implies that, if one seeks to compute the interim BNE-BPD by solving the corresponding ex-ante BNE-BPD, it suffices to set the outer risk measure as the expectation operator, namely, $\rho_{\nu_{i,\theta_i}^t}^{\rm ex}=\mathbb{E}_{\nu_{i,\theta_i}^t}$. Moreover, when the inner risk measure $\rho_{\nu_i^t(\cdot\mid\theta_i)}^{\rm bp}$ admits the parametric form in \eqref{eqn:parametric-risk-measure}, the nested composition $\rho_{\nu_{i,\theta_i}^t}^{\rm ex}\circ\rho_{\nu_i^t(\cdot\mid\theta_i)}^{\rm bp}$ reduces to a single expectation under $\nu_i^t$ over the entire type profile. 
Consequently, the objective function can be evaluated through a single-stage simulation in which the full type vector $\theta$ is sampled directly from the Bayesian predictive distribution $\nu_i^t$. This avoids the need to repeatedly sample rivals' types $\theta_{-i}$ conditional on each realization of $\theta_i$. As discussed in Section~\ref{sec:preliminary}, a variety of classical risk measures can be recovered within this framework through appropriate choices of the function $\Psi$.

\begin{corollary} \label{corollary:oce-cvar-ent}
    Consider the setting and conditions of Corollary~\ref{corollary:reformuation-parametric}.
    If the inner risk measure is given by an OCE of a utility function $u:\R\rightarrow\R$ defined in \eqref{eqn:oce},
    then the following assertions hold.
    \begin{itemize}
        \item[(i)] $f^{t*}$ is almost surely a measurable interim BNE-BPD
        if and only if,
        for $i\in N$,
        \begin{eqnarray} \label{eqn:equivalent-BNE-oce}
            f_i^{t*}\in \mathop{\arg\min}_{f_i\in \mathcal{F}_i} \inf_{\varphi(\cdot) \in \mathcal{M}_\mathcal{S}} \mathbb{E}_{\nu_{i}^t} \left[ \varphi(\theta_i) - u \left(   \varphi(\theta_i) -c(f_i(\theta_i), f_{-i}^{t*}(\theta_{-i}),\theta_i, \theta_{-i}) \right) \right].
        \end{eqnarray}

        \item[(ii)] If, in particular, the inner risk measure is the CVaR with $u(x)=\frac{1}{1-\beta}x$ for $x\leq 0$ and $u(x)=0$ for $x>0$, 
        then $f^{t*}$ is almost surely a measurable interim BNE-BPD
        if and only if,
        for $i\in N$,
        \begin{eqnarray} \label{eqn:equivalent-BNE-cvar}
            f_i^{t*}\in \mathop{\arg\min}_{f_i\in \mathcal{F}_i} \inf_{\varphi(\cdot) \in \mathcal{M}_\mathcal{S}} \mathbb{E}_{\nu_{i}^t} \left[ \varphi(\theta_i) + \frac{1}{1-\beta} \left( c(f_i(\theta_i), f_{-i}^{t*}(\theta_{-i}),\theta_i, \theta_{-i}) - \varphi(\theta_i) \right)_+ \right].\qquad
        \end{eqnarray}

        \item[(iii)] If, in particular, the inner risk measure is the entropic risk measure with $u(x)=\frac{1}{\gamma}(1-e^{-\gamma x})$ and $\gamma>0$, 
        then $f^{t*}$ is almost surely a measurable interim BNE-BPD
        if and only if,
        for $i\in N$,
        \begin{eqnarray} \label{eqn:equivalent-BNE-ent}
            f_i^{t*}\in \mathop{\arg\min}_{f_i\in \mathcal{F}_i} \mathbb{E}_{\nu_{i}^t} \left[ \exp\left( \gamma c(f_i(\theta_i), f_{-i}^{t*}(\theta_{-i}),\theta_i, \theta_{-i}) \right)  \right].
        \end{eqnarray}
    \end{itemize}

\end{corollary}

\noindent
\textbf{Proof.}
Parts (i) and (ii) follow directly from substituting the corresponding forms of $\Psi$.
It remains to prove part (iii).
Set $u(x)=\frac{1}{\gamma}(1-e^{-\gamma x})$ in  \eqref{eqn:equivalent-BNE-oce}. 
Then
\begin{eqnarray*}
    f_i^{t*}&\in & \mathop{\arg\min}_{f_i\in \mathcal{F}_i} \inf_{\varphi(\cdot) \in \mathcal{M}_\mathcal{S}} \mathbb{E}_{\nu_{i}^t} \left[ \varphi(\theta_i) - \frac{1}{\gamma} \left(1- e^{ -\gamma \left( \varphi(\theta_i) - c(f_i(\theta_i), f_{-i}^{t*}(\theta_{-i}),\theta_i, \theta_{-i}) \right) } \right) \right] \\
    &=& \mathop{\arg\min}_{f_i\in \mathcal{F}_i} \inf_{\varphi(\cdot) \in \mathcal{M}_\mathcal{S}} \mathbb{E}_{\nu_{i}^t} \left[ \varphi(\theta_i) + \frac{1}{\gamma \exp\left(\gamma  \varphi(\theta_i)\right) } 
    \exp\left(\gamma c(f_i(\theta_i), f_{-i}^{t*}(\theta_{-i}),\theta_i, \theta_{-i}) \right)  \right] - \frac{1}{\gamma}.
\end{eqnarray*}
Note that for any optimal $\varphi^*\in \mathcal{M}_\mathcal{S}$,  
\begin{eqnarray*}
f_i^{t*}&\in & \mathop{\arg\min}_{f_i\in \mathcal{F}_i} 
\mathbb{E}_{\nu_{i}^t} \left[ \varphi^*(\theta_i) + \frac{1}{\gamma \exp\left(\gamma  \varphi^*(\theta_i)\right) } 
    \exp\left(\gamma c(f_i(\theta_i), f_{-i}^{t*}(\theta_{-i}),\theta_i, \theta_{-i}) \right)  \right] - \frac{1}{\gamma}\\
    & = &  
    \mathop{\arg\min}_{f_i\in \mathcal{F}_i}  
    \mathbb{E}_{\nu_{i,\theta_i}^t} \left[    \varphi^*(\theta_i) + \frac{1}{\gamma \exp\left(\gamma  \varphi^*(\theta_i)\right) } \mathbb{E}_{\nu_{i}^t(\cdot|\theta_i)} \left[ \exp\left(\gamma c(f_i(\theta_i), f_{-i}^{t*}(\theta_{-i}),\theta_i, \theta_{-i})  \right)\right]\right].
\end{eqnarray*}
By the interchangeability principle, for a.e. $\theta_i \in \Theta_i$,
\begin{eqnarray} \label{eqn:proof-ent}
    f_i^{t*}(\theta_i) &\in& \mathop{\arg\min}_{a_i\in \mathcal{A}_i}\varphi^*(\theta_i) + \frac{1}{\gamma \exp\left(\gamma  \varphi^*(\theta_i)\right) } \mathbb{E}_{\nu_{i}^t(\cdot|\theta_i)} \left[ \exp\left(\gamma c(a_i, f_{-i}^{t*}(\theta_{-i}),\theta_i, \theta_{-i})  \right)\right]\nonumber\\
    &=& 
    \mathop{\arg\min}_{a_i\in \mathcal{A}_i} \mathbb{E}_{\nu_{i}^t(\cdot|\theta_i)} \left[ \exp\left(\gamma c(a_i, f_{-i}^{t*}(\theta_{-i}),\theta_i, \theta_{-i})  \right)\right], 
\end{eqnarray}
which means that the optimal solution $f_i^{t*}$ is independent of the choice of $\varphi^*$.
Consequently, 
we can establish the equivalence between 
\eqref{eqn:proof-ent} and  \eqref{eqn:equivalent-BNE-ent}
underpinned by the interchangeability principle. 
\hfill $\Box$

Corollary~\ref{corollary:oce-cvar-ent} shows that, 
for 
the OCE-type risk measures,
the objective function of the ex-ante BNE-BPD problem can be written as a single expectation over the type profile under $\nu_i^t$. 
The same 
reformulations 
may be obtained 
for other risk measures which 
admit a representation in terms of expectations or  CVaR.
The reformulation is similar to that used
in the risk-averse contextual optimization of \cite{tao2025risk}, where 
the objective is expressed as a nested risk measure with respect to
contextual uncertainty
(corresponding to
$\theta_i$) and 
the problem data uncertainty
(corresponding to
$\theta_{-i}$).
In their setting,
sample data are typically available only for the joint distribution of two types of uncertainties in that 
the conditional sample data for problem data uncertainty are typically unavailable for almost every 
context.
By contrast, the Bayesian learning framework adopted in this paper enables us to learn the underlying continuous distribution, 
and hence to sample from the marginal distribution of $\theta_i$ and the conditional distribution of $\theta_{-i}$ separately. 
Thus,
the reformulation is not motivated by the unavailability of conditional samples, rather it is
intended
to simplify the nested simulation to a single-stage simulation over the joint distribution, which may reduce the computational cost.

\section{Existence and uniqueness of BNE-RABL and BNE-BPD}
\label{sec:exist-unique}

In this section, we move on to discuss the existence and uniqueness of equilibria in the BNE models introduced in Section~\ref{sec:model}.

\subsection{Existence of interim BNE-RABL
}
\label{sec:exist-bne-rabl}

We begin with the interim BNE-RABL. To this end, we first show that each player's decision-making problem is well defined under appropriate conditions. For ease of notation, let the interim objective function of player $i$ in a fixed round $t$ be denoted by
\begin{equation}
\label{eqn:v-rabl-def}
    v_{i,\mu_i^t}^t(a_i,f_{-i},\theta_i)
    :=
    \rho_{\mu_i^t}^{\rm ep}
    \left(
        \rho_{\eta(\cdot\mid\theta_i;\zeta)}^{\rm al}
        \left(
            c_i\left(
                a_i,
                f_{-i}(\theta_{-i}),
                \theta_i,
                \theta_{-i}
            \right)
        \right)
    \right).
\end{equation}
The corresponding optimal value and the set of optimal solutions are denoted by
\begin{equation}
\label{eqn:vartheta-rabl-def}
    \vartheta_{i,\mu_i^t}^t(f_{-i},\theta_i)
    :=
    \min_{a_i\in\mathcal A_i}
    v_{i,\mu_i^t}^t(a_i,f_{-i},\theta_i),
    \qquad
    \mathcal A_{i,\mu_i^t}^{t*}(f_{-i},\theta_i)
    :=
    \arg\min_{a_i\in\mathcal A_i}
    v_{i,\mu_i^t}^t(a_i,f_{-i},\theta_i).
\end{equation}

\begin{assumption}
\label{assump:rabl-common-primitive} 
Let $t$ be fixed.
For each $i\in N$, 
(a)
the loss function $c_i$ is $M_i$-strongly convex in $a_i$ over $\mathcal{A}_i$ uniformly with respect to 
$(a_{-i},\theta_i,\theta_{-i})\in
\mathcal A_{-i}\times\Theta_i\times\Theta_{-i}$, i.e.,
there exists a constant $M_i>0$ such that 
\begin{eqnarray}
\label{eqn:primitive-strong-convexity-cost}
c_i(\lambda a_i+(1-\lambda)a'_i,a_{-i},\theta_i,\theta_{-i}) 
&\leq& 
\lambda c_i(a_i,a_{-i},\theta_i,\theta_{-i})
+
(1-\lambda)c_i(a_i',a_{-i},\theta_i,\theta_{-i}) \nonumber\\
&& -
\frac{M_i}{2}\lambda(1-\lambda)\|a_i-a_i'\|^2, \quad \forall  a_i,a_i'\in\mathcal A_i, 
\lambda\in [0,1] \qquad \qquad
\end{eqnarray}
uniformly for all
$(a_{-i},\theta_i,\theta_{-i})\in
\mathcal A_{-i}\times\Theta_i\times\Theta_{-i}$;
(b)
for each fixed $\theta_i \in \Theta_i$ and any $\epsilon>0$, there exists $\delta>0$ such that
\bgeqn 
    \dd_{\rm TV}
    \left(
        \eta(\cdot\mid\theta_i;\zeta),
        \eta(\cdot\mid\theta_i';\zeta)
    \right)
    \leq \epsilon,
\edeqn
for all $\theta_i'\in \Theta_i$ with $\|\theta_i'- \theta_i\| \leq \delta,$ uniformly for all $\zeta\in \mathcal{Z}$,
where $\dd_{\rm TV}$
denotes the total variation metric  defined by
\begin{eqnarray} \label{eqn:tv-definition-2}
    \dd_{\rm TV} (P,Q):= \frac{1}{2}\sup_{h\in \mathcal{M}, \|h\|_\infty \leq 1} \bigg| \mathbb{E}_P[h(\xi)] - \mathbb{E}_Q[h(\xi)] \bigg|,
\end{eqnarray}
and $\mathcal{M}$ is the set of all measurable functions (see, e.g., \cite{athreya2006measure});
(c)
the risk measures $\rho^{\rm al}$ and $\rho^{\rm ep}$ are law-invariant convex risk measures;
(d)
let $\varrho^{\rm al}$ be the risk functional corresponding to $\rho^{\rm al}$
and $\mathscr{P}(\R)$ be the set of probability measures over $\R$.
$\varrho^{\rm al}: \mathscr{P}(\R)\to \R$ 
is continuous,
i.e., for every $\epsilon>0$, there exists $\delta>0$ such that for any 
$P,Q\in \mathscr{P}(\R)$, 
\bgeqn 
\left| \varrho^{\rm al}(P) - \varrho^{\rm al}(Q)  \right| \leq \epsilon, \quad \inmat{whenever} \; \dd_{\rm TV} (P,Q) \leq \delta. 
\edeqn 
\end{assumption}

Condition~(a) requires the strong convexity of the loss function in player $i$'s own action,
where $M_i$ characterizes 
a lower bound on the curvature of the loss function in the player's own action.
Condition (b) ensures continuity of the conditional distribution with respect to the player’s own type. 
This condition can be verified directly from conditional densities. For example, if the conditional density
$p(\theta_{-i}\mid\theta_i;\zeta)$ exists and, for every $\epsilon>0$, there is
$\delta>0$ such that
\[
    \frac12
    \int_{\Theta_{-i}}
    \left|
        p(\theta_{-i}\mid\theta_i;\zeta)
        -
        p(\theta_{-i}\mid\theta_i';\zeta)
    \right|d\theta_{-i}
    <\epsilon,
\]
for all $\theta_i'$ with $\|\theta_i'- \theta_i\| \leq \delta,$ uniformly for all $\zeta\in \mathcal{Z}$,
then condition~(b) holds.
Condition~(d) is concerned with the continuity of $\rho^{\rm al}$ with respect to variation in the underlying probability distribution under the total variation metric.
This condition is fulfilled by a broad class of commonly used risk measures, in particular those admitting an expectation-based representation, such as the expectation, the CVaR with confidence level strictly less than one, and the entropic risk measure.

The continuity condition (d) is widely employed in the risk management literature, typically with $Q$ deviating from $P$ under the Kantorovich--Wasserstein metric; see, e.g.,~\cite{claus2016advancing,wang2021quantitative}.
Here, we adopt the total variation metric primarily in order to establish the continuity of $\rho_{\eta(\cdot|\theta_i;\zeta)}^{\rm al}(c(a_i, f_{-i}(\theta_{-i}), \theta_i, \theta_{-i}))$ in~$\theta_i$.
This is because continuity under the Wasserstein metric would require
\(c(a_i, f_{-i}(\theta_{-i}), \theta_i, \theta_{-i})\) to be Lipschitz
continuous in \(\theta_{-i}\), which in turn necessitates Lipschitz
continuity of \(f_{-i}\). However, establishing the latter is difficult
in the present setting.
By contrast, continuity with respect to the total variation metric only requires the measurability of $c(a_i, f_{-i}(\theta_{-i}), \theta_i, \theta_{-i})$ in $\theta_{-i}$, as we will see in the proof of the next lemma.

\begin{lemma}
\label{lem:rabl-common-regularity}
Let Assumption~\ref{assump:rabl-common-primitive} hold. 
Then, for each
$i\in N$, the following assertions hold.

\begin{itemize}
    \item[(i)] The objective function 
    $v_{i,\mu_i^t}^t(a_i,f_{-i},\theta_i)$ in \eqref{eqn:v-rabl-def} is continuous on $\mathcal{A}_{i}\times \mathcal{F}_{-i}\times \Theta_i$.

    \item[(ii)] For $f_{-i} \in \mathcal{F}_{-i}$ and $\theta_i\in \Theta_i$, $v_{i,\mu_i^t}^t(a_i,f_{-i},\theta_i)$ is $M_i$-strongly  convex on $\mathcal{A}_i$.

\end{itemize}

\end{lemma}

\noindent
\textbf{Proof.}
\underline{Part (i).}
For any $(a_i,f_{-i},\theta_i)$ and $(a_i',f_{-i}',\theta_i')$, it follows from the Lipschitz continuity of the risk measure w.r.t.~the supremum norm that
\begin{eqnarray} \label{eqn:proof-continuity-0}
    && \left| v_{i,\mu_i^t}^t(a_i,f_{-i},\theta_i) -  v_{i,\mu_i^t}^t(a_i',f_{-i}',\theta_i') \right|\nonumber\\
    &\leq& \left\|  \rho_{\eta(\cdot\mid\theta_i;\zeta)}^{\rm al} \left( c_i(a_i, f_{-i}(\theta_{-i}),\theta_i,\theta_{-i}) \right)  
    - \rho_{\eta(\cdot\mid\theta_i';\zeta)}^{\rm al}
\left( c_i(a_i', f_{-i}'(\theta_{-i}),\theta_i',\theta_{-i}) \right)
    \right\|_{\mathcal{Z},\infty} \nonumber\\
    &\leq & \left\|  \rho_{\eta(\cdot\mid\theta_i;\zeta)}^{\rm al} \left( c_i(a_i, f_{-i}(\theta_{-i}),\theta_i,\theta_{-i}) \right)  
    - \rho_{\eta(\cdot\mid\theta_i;\zeta)}^{\rm al}
\left( c_i(a_i', f_{-i}'(\theta_{-i}),\theta_i',\theta_{-i}) \right)
    \right\|_{\mathcal{Z},\infty}\nonumber\\
    && + \left\|  \rho_{\eta(\cdot\mid\theta_i;\zeta)}^{\rm al} \left( c_i(a_i', f_{-i}'(\theta_{-i}),\theta_i',\theta_{-i}) \right)  
    - \rho_{\eta(\cdot\mid\theta_i';\zeta)}^{\rm al}
\left( c_i(a_i', f_{-i}'(\theta_{-i}),\theta_i',\theta_{-i}) \right)
    \right\|_{\mathcal{Z},\infty}\nonumber\\
    &\leq & \left\|   c_i(a_i, f_{-i}(\theta_{-i}),\theta_i,\theta_{-i})   
    -  c_i(a_i', f_{-i}'(\theta_{-i}),\theta_i',\theta_{-i})
    \right\|_{\Theta_{-i},\infty}\nonumber\\
    && + \left\|  \rho_{\eta(\cdot\mid\theta_i;\zeta)}^{\rm al} \left( c_i(a_i', f_{-i}'(\theta_{-i}),\theta_i',\theta_{-i}) \right)  
    - \rho_{\eta(\cdot\mid\theta_i';\zeta)}^{\rm al}
\left( c_i(a_i', f_{-i}'(\theta_{-i}),\theta_i',\theta_{-i}) \right)
    \right\|_{\mathcal{Z},\infty}.
\end{eqnarray}
Under Assumption~\ref{basic-assumption}(b), for any $\epsilon>0$, there exists $\delta_1>0$ such that 
$    
\|a_i-a_i'\|+\|f_{-i}-f'_{-i}\|_\infty+\|\theta_i-\theta_i'\|\leq \delta_1
$
implies
\begin{eqnarray} \label{eqn:proof-continuity-1}
    \sup_{\theta_{-i}\in \Theta_{-i}}\left|   c_i(a_i, f_{-i}(\theta_{-i}),\theta_i,\theta_{-i})   
    -  c_i(a_i', f_{-i}'(\theta_{-i}),\theta_i',\theta_{-i})
    \right|\leq \frac{\epsilon}{2}.
\end{eqnarray}
Moreover, by Assumption~\ref{assump:rabl-common-primitive}(d),
for the given $\epsilon>0$,
there exists $\delta_2>0$ such that when
\begin{eqnarray} \label{eqn:proof-continuity-3}
    && \dd_{TV}\left( \eta(\cdot\mid\theta_i;\zeta) \circ c_i(a_i', f_{-i}'(\cdot),\theta_i',\cdot)^{-1}, \eta(\cdot\mid\theta_i';\zeta) \circ c_i(a_i', f_{-i}'(\cdot),\theta_i',\cdot)^{-1} \right)
    < \delta_2,
\end{eqnarray}
we have
\begin{eqnarray}\label{eqn:proof-continuity-2}
    &&\left|  \rho_{\eta(\cdot\mid\theta_i;\zeta)}^{\rm al} \left( c_i(a_i', f_{-i}'(\theta_{-i}),\theta_i',\theta_{-i}) \right)  
    - \rho_{\eta(\cdot\mid\theta_i';\zeta)}^{\rm al}
\left( c_i(a_i', f_{-i}'(\theta_{-i}),\theta_i',\theta_{-i}) \right)
    \right|\nonumber\\
    &=&  \left|  \varrho^{\rm al} \left( \eta(\cdot\mid\theta_i;\zeta) \circ c_i(a_i', f_{-i}'(\cdot),\theta_i',\cdot)^{-1} \right)  
    - \varrho^{\rm al} \left( \eta(\cdot\mid\theta_i';\zeta) \circ c_i(a_i', f_{-i}'(\cdot),\theta_i',\cdot)^{-1} \right)
    \right|\leq \frac{\epsilon}{2},\qquad 
\end{eqnarray}
uniformly for $\zeta \in \mathcal{Z}$.
In what follows,
we derive sufficient conditions for 
\eqref{eqn:proof-continuity-3}.
Under Assumption~\ref{assump:rabl-common-primitive} (b),
 there exists $\delta_3>0$ such that
when $\|\theta_i- \theta_i'\| < \delta_3$,
we have
$
    \dd_{TV}\left( \eta(\cdot\mid\theta_i;\zeta), \eta(\cdot\mid\theta_i';\zeta) \right) 
    < \delta_2.
$
Let $T(\theta_{-i}):= c(a_i', f_{-i}'(\theta_{-i}), \theta_i', \theta_{-i})$.
\begin{eqnarray} \label{eqn:nonexpansive}
    \dd_{\rm TV} (\eta(\cdot|\theta_i;\zeta)\circ T^{-1},\eta(\cdot|\theta_i';\zeta)\circ T^{-1})
    &=& \sup_{h\in \mathcal{M}, \|h\|_\infty \leq 1} \bigg|\mathbb{E}_{\eta(\cdot|\theta_i;\zeta)\circ T^{-1}}[h(\xi)] - \mathbb{E}_{\eta(\cdot|\theta_i';\zeta)\circ T^{-1}}[h(\xi)]\bigg|\nonumber\\
    &=& \sup_{h\in \mathcal{M}, \|h\|_\infty \leq 1} \bigg|\mathbb{E}_{\eta(\cdot|\theta_i;\zeta)}[h(T(\xi))] - \mathbb{E}_{\eta(\cdot|\theta_i';\zeta)}[h(T(\xi))]\bigg|\nonumber\\
    &\leq & \sup_{\tilde{h}\in \mathcal{M}, \|\tilde{h}\|_\infty \leq 1} \bigg|\mathbb{E}_{\eta(\cdot|\theta_i;\zeta)}[\tilde{h}(\xi)] - \mathbb{E}_{\eta(\cdot|\theta_i';\zeta)}[\tilde{h}(\xi)]\bigg| \nonumber\\
    &=& \dd_{\rm TV} (\eta(\cdot|\theta_i;\zeta),\eta(\cdot|\theta_i';\zeta)) \leq \delta_2,
\end{eqnarray}
which shows that \eqref{eqn:proof-continuity-3} is satisfied when $\|\theta_i- \theta_i'\| < \delta_3$.
Combining \eqref{eqn:proof-continuity-0}, \eqref{eqn:proof-continuity-1}, and \eqref{eqn:proof-continuity-2}, we obtain $\left| v_{i,\mu_i^t}^t(a_i,f_{-i},\theta_i) -  v_{i,\mu_i^t}^t(a_i',f_{-i}',\theta_i') \right| \leq\epsilon$, which implies the continuity of $v_{i,\mu_i^t}^t(a_i,f_{-i},\theta_i)$.

\underline{Part (ii).}
Let $a_i, a_i' \in \mathcal{A}_i$ and $\lambda\in [0,1]$.
By Assumption~\ref{assump:rabl-common-primitive}(a),
\begin{eqnarray*}
    &&c_i(\lambda a_i+(1-\lambda)a_i',
     f_{-i}(\theta_{-i}),\theta_i,\theta_{-i})\\
&&\quad\leq
\lambda c_i(a_i,f_{-i}(\theta_{-i}),\theta_i,\theta_{-i})
+
(1-\lambda)c_i(a_i',f_{-i}(\theta_{-i}),\theta_i,\theta_{-i})
-\frac{M_i}{2}\lambda(1-\lambda)\|a_i-a_i'\|^2.
\end{eqnarray*}
Using the monotonicity, cash invariance and the convexity of $\rho_{\eta(\cdot\mid\theta_i;\zeta)}^{\rm al}$, we have
\begin{eqnarray*}
    &&\rho_{\eta(\cdot\mid\theta_i;\zeta)}^{\rm al}
\left( c_i(\lambda a_i+(1-\lambda)a_i',
     f_{-i}(\theta_{-i}),\theta_i,\theta_{-i}) \right)\\
&\leq& \rho_{\eta(\cdot\mid\theta_i;\zeta)}^{\rm al}
\bigg(
\lambda c_i(a_i,f_{-i}(\theta_{-i}),\theta_i,\theta_{-i})
+
(1-\lambda)c_i(a_i',f_{-i}(\theta_{-i}),\theta_i,\theta_{-i})\\
&& \qquad\qquad
-\frac{M_i}{2}\lambda(1-\lambda)\|a_i-a_i'\|^2\bigg)\\
&\leq& \lambda \rho_{\eta(\cdot\mid\theta_i;\zeta)}^{\rm al}
\bigg(
c_i(a_i,f_{-i}(\theta_{-i}),\theta_i,\theta_{-i})\bigg)
+
(1-\lambda) \rho_{\eta(\cdot\mid\theta_i;\zeta)}^{\rm al}\bigg(
c_i(a_i',f_{-i}(\theta_{-i}),\theta_i,\theta_{-i})\bigg)\\
&& -\frac{M_i}{2}\lambda(1-\lambda)\|a_i-a_i'\|^2.
\end{eqnarray*}
Using the same argument for $\rho_{\mu_i^t}^{\rm ep}$, we can establish 
\begin{eqnarray*}
    &&\rho_{\mu_i^t}^{\rm ep} \left(\rho_{\eta(\cdot\mid\theta_i;\zeta)}^{\rm al}
        \left(c_i \left(\lambda a_i+(1-\lambda)a_i', f_{-i}(\theta_{-i}), \theta_i, \theta_{-i} \right) \right)
    \right)\\
    &\leq & \rho_{\mu_i^t}^{\rm ep} 
    \Bigg(
       \lambda \rho_{\eta(\cdot\mid\theta_i;\zeta)}^{\rm al}
\bigg(
 c_i(a_i,f_{-i}(\theta_{-i}),\theta_i,\theta_{-i})\bigg)
+
(1-\lambda) \rho_{\eta(\cdot\mid\theta_i;\zeta)}^{\rm al}\bigg(
c_i(a_i',f_{-i}(\theta_{-i}),\theta_i,\theta_{-i})\bigg)\\
&& -\frac{M_i}{2}\lambda(1-\lambda)\|a_i-a_i'\|^2
    \Bigg)\\
    &\leq &  \lambda\rho_{\mu_i^t}^{\rm ep} 
    \Bigg(
       \rho_{\eta(\cdot\mid\theta_i;\zeta)}^{\rm al}
\bigg(
 c_i(a_i,f_{-i}(\theta_{-i}),\theta_i,\theta_{-i})\bigg)
    \Bigg)
+
(1-\lambda)
\rho_{\mu_i^t}^{\rm ep} 
    \Bigg(
 \rho_{\eta(\cdot\mid\theta_i;\zeta)}^{\rm al}\bigg(
c_i(a_i',f_{-i}(\theta_{-i}),\theta_i,\theta_{-i})\bigg)\Bigg)\\
&& -\frac{M_i}{2}\lambda(1-\lambda)\|a_i-a_i'\|^2,
\end{eqnarray*}
which implies the $M_i$-strong convexity of $v_{i,\mu_i^t}^t(a_i,f_{-i},\theta_i)$. 
\hfill $\Box$

This lemma establishes the continuity of each player's objective function and the uniqueness of the optimal solution for each type, and hence ensures the well-definedness of the decision-making problem faced by each player.
We now turn to the existence of BNE-RABL.
A crucial step in this approach is to establish the equicontinuity of the set of optimal response functions.

\begin{lemma}
\label{lem:rabl-schauder-best-response}
Let
$
    \Psi_{i,\mu_i^t}^t(f_{-i})(\theta_i)
    :=
    \mathcal A_{i,\mu_i^t}^{t*}(f_{-i},\theta_i),
$
and 
$
    \Psi_{\mu^t}^t(f):=(\Psi_{1,\mu_1^t}^t(f_{-1}),\ldots,\Psi_{n,\mu_n^t}^t(f_{-n})).
$
Under Assumptions~\ref{basic-assumption} and \ref{assump:rabl-common-primitive}, the following assertions hold.

\begin{enumerate}[(i)]
\item For every $i\in N$, $\Psi_{i,\mu^t}^t(f_{-i})(\theta_i)$ is single-valued for every $(f_{-i},\theta_i)\in \mathcal{F}_{-i} \times \Theta_i$ and continuous over $\mathcal{F}_{-i}\times \Theta_i$.

\item For every $i\in N$, the family
$
    \{\Psi_{i,\mu^t}^t(f_{-i}):f_{-i}\in\mathcal C_{-i}\}
$
is equicontinuous on $\Theta_i$.

\item The operator $\Psi_{\mu^t}^t:\mathcal C\to\mathcal C$ is continuous under the supremum norm.
\end{enumerate}
\end{lemma}

\noindent
\textbf{Proof.}
\underline{Part (i).}
By the $M_i$-strong convexity of $v_{i,\mu_i^t}^t(a_i,f_{-i},\theta_i)$ over $\mathcal{A}_i$ in Lemma~\ref{lem:rabl-common-regularity}(ii), we immediately obtain that $\Psi_{i,\mu^t}^t(f_{-i})(\theta_i)$ is single-valued for every $\theta_i\in\Theta_i$.
Moreover, by the continuity of $v_{i,\mu_i^t}^t(a_i,f_{-i},\theta_i)$ established in Lemma~\ref{lem:rabl-common-regularity}(i) and the compactness of $\mathcal{A}_i$ in Assumption~\ref{assump:rabl-common-primitive}(a), 
Berge's maximum theorem \citep{berge1963topological} ensures the continuity of $\Psi_{i,\mu^t}^t(f_{-i})(\theta_i)$ over $\mathcal{F}_{-i}\times \Theta_i$.

\underline{Part (ii).}
By the compactness of $\Theta_i$, Assumption~\ref{assump:rabl-common-primitive}(b) guarantees the continuity of $\eta(\cdot|\theta_i;\zeta)$ in $\theta_i$ uniformly with respect to $\zeta$ under the total variation metric.
Likewise, the compactness of $\mathcal{A}\times \Theta_{-i}$ ensures the continuity of $c_i$ over $\Theta_i$ uniformly with respect to $\theta_{-i}$.
Using the same argument as in the proof of Lemma~\ref{lem:rabl-common-regularity}(i), we can show that  
for every $\epsilon >0$, there exists $\delta>0$ such that 
when $\|\theta_i- \theta_i'\|< \delta$, we have
$$\sup_{a_i \in \mathcal{A}_i}\left| v_{i,\mu_i^t}^t(a_i,f_{-i},\theta_i)-v_{i,\mu_i^t}^t(a_i,f_{-i},\theta_i') \right| <\epsilon,$$
uniformly for all $f_{-i}\in \mathcal{C}_{-i}$. We skip the details.

On the other hand, for any $f_{-i}\in \mathcal{C}_{-i}$, let 
$a_i^*:=\Psi_{i,\mu^t}^t(f_{-i})(\theta_i)$ and $a_i^{\prime *}:=\Psi_{i,\mu^t}^t(f_{-i})(\theta_i')$.
Since $v_{i,\mu_i^t}^t(a_i,f_{-i},\theta_i)$ is $M_i$-strongly convex,
\begin{eqnarray*}
    v_{i,\mu_i^t}^t(a_i^{\prime *},f_{-i},\theta_i)
    \geq
    v_{i,\mu_i^t}^t(a_i^*,f_{-i},\theta_i)
    +
    \frac{M_i}{2}\|a_i^{\prime *}-a_i^*\|^2,
\end{eqnarray*}
and thus
\begin{eqnarray*}
    \frac{M_i}{2}\|a_i^{\prime *}-a_i^*\|^2
    &\leq&
    v_{i,\mu_i^t}^t(a_i^{\prime *},f_{-i},\theta_i)
    -
    v_{i,\mu_i^t}^t(a_i^*,f_{-i},\theta_i)\\
    &\leq&
    \left|
    v_{i,\mu_i^t}^t(a_i^{\prime *},f_{-i},\theta_i)
    -
    v_{i,\mu_i^t}^t(a_i^{\prime *},f_{-i},\theta_i')
    \right|\\
    &&+
    \left|
    v_{i,\mu_i^t}^t(a_i^*,f_{-i},\theta_i')
    -
    v_{i,\mu_i^t}^t(a_i^*,f_{-i},\theta_i)
    \right|
    <2\epsilon.
\end{eqnarray*}
Hence,  for any $f_{-i}\in \mathcal{C}_{-i}$
\begin{eqnarray} \label{eqn:proof-best-response-schauder}
    \left\| \Psi_{i,\mu^t}^t(f_{-i})(\theta_i)-\Psi_{i,\mu^t}^t(f_{-i})(\theta_i') \right\| = \|a_i^{\prime *}-a_i^*\| \leq 2 \sqrt{\frac{\epsilon}{M_i}},
\end{eqnarray}
which proves the equicontinuity of $\Psi_{i,\mu^t}^t$ in $\theta_i$.

\underline{Part (iii)}.
Based on the continuity of $\Psi_{i,\mu^t}^t(f_{-i})(\theta_i)$ over $\mathcal{F}_{-i}\times \Theta_i$ in part (i) and the equicontinuity of $\Psi_{i,\mu^t}^t$ over $\Theta_i$ in part (ii), we show that $\Psi_{\mu^t}^t: \mathcal{C}\rightarrow \mathcal{C}$ is a continuous operator.
Since $\Theta_i$ is compact under Assumption~\ref{basic-assumption}(a), there exists a $\delta$-net covering $\Theta_i$, i.e., for any small positive number $\delta \in (0,1)$, there exists a finite number of points $\theta_i^1,\dots,\theta_i^M \in \Theta_i$ such that for any $\theta_i\in \Theta_i$, there is $k\in \{1,\dots,M\}$ such that $\|\theta_i - \theta_i^k\| \leq \delta$. 
Then, by the continuity of $\Psi_{i,\mu^t}^t(f_{-i})(\theta_i)$ in $(f_{-i},\theta_i)$,
for any $\epsilon>0$ and fixed $f_{-i}\in \mathcal{C}_{-i}$, we have $\sup_{k\in\{1,\dots,M\}}|\Psi_{i,\mu^t}^t(f_{-i})(\theta_i^k)-\Psi_{i,\mu^t}^t(f'_{-i})(\theta_i^k)|< \epsilon$ whenever $f'_{-i}$ is sufficiently close to $f_{-i}$.
Together with \eqref{eqn:proof-best-response-schauder}, we can select $\delta$ as in part (ii) such that for any $\theta_i\in\Theta_i$,
\begin{eqnarray*}
    &&\left\| \Psi_{i,\mu^t}^t(f_{-i}')(\theta_i)-\Psi_{i,\mu^t}^t(f_{-i})(\theta_i) \right\|\\
    &&\leq \left\| \Psi_{i,\mu^t}^t(f_{-i}')(\theta_i) - \Psi_{i,\mu^t}^t(f_{-i}')(\theta_i^k) \right\| +\left\| \Psi_{i,\mu^t}^t(f_{-i}')(\theta_i^k) - \Psi_{i,\mu^t}^t(f_{-i})(\theta_i^k) \right\|\\
    && \quad+ \left\| \Psi_{i,\mu^t}^t(f_{-i})(\theta_i^k) - \Psi_{i,\mu^t}^t(f_{-i})(\theta_i) \right\| \leq 4 \sqrt{\frac{\epsilon}{M_i}} + \epsilon,
\end{eqnarray*}
which implies the continuity of $\Psi_{i,\mu^t}^t$ in $\mathcal{C}_{-i}$.
\hfill $\Box$

We are ready to establish the existence of BNE-RABL.

\begin{theorem}
\label{thm:rabl-existence-schauder}
Under Assumptions~\ref{basic-assumption} and \ref{assump:rabl-common-primitive}, 
the BNE-RABL model \eqref{eqn:BNE-heter-belief}
has a continuous equilibrium
$f^{t*}\in\mathcal C$ in round $t$.
\end{theorem}

\noindent
\textbf{Proof.}
The set $\mathcal C=\prod_{i\in N}\mathcal C_i$ is a nonempty, closed, bounded,
and convex subset of the Banach space of continuous functions endowed with the
supremum norm. 
Since $\Psi_{i,\mu^t}^t(f_{-i})$ takes values in the compact set $\mathcal A_i$, the family
$\{\Psi_{i,\mu^t}^t(f_{-i}):f_{-i}\in\mathcal C_{-i}\}$ is uniformly bounded. 
By the equicontinuity established in Lemma~\ref{lem:rabl-schauder-best-response}(ii) and Arzel\`a--Ascoli theorem \cite[Theorem~A5]{rudin1973functional},
$\Psi_{i,\mu^t}^t(\mathcal C_{-i})$ is relatively compact in $\mathcal C_i$, and hence
$\Psi_{\mu^t}^t(\mathcal C)$ is relatively compact in $\mathcal C$.
By Lemma~\ref{lem:rabl-schauder-best-response}(iii),
$\Psi_{\mu^t}^t:\mathcal C\to\mathcal C$ is continuous.
Therefore, $\Psi_{\mu^t}^t$ is a compact operator. Schauder's fixed-point theorem ensures the existence of
$f^{t*}\in\mathcal C$ such that
$
    \Psi_{\mu^t}^t(f^{t*})=f^{t*}.
$
By the definition of $\Psi_{\mu^t}^t$ and Lemma~\ref{lem:rabl-schauder-best-response}(i), $f^{t*}$ is a continuous BNE-RABL.
\hfill $\Box$

The proof of Theorem~\ref{thm:rabl-existence-schauder} uses Schauder's fixed-point theorem, which is commonly used in the continuous BNE literature; see, e.g., \cite{guo2021existence,meirowitz2003existence,tao2025generalized}. 
The main difference is that here we apply this approach to the nested risk-averse
objective in the BNE-RABL model.
In particular, the required continuity and strong convexity of the objective function are derived from conditions on the model primitives, including the loss function, the conditional type distributions, and the risk measures, rather than imposed directly on $v_{i,\mu_i^t}^t$.
Moreover, we do not require the uniform Lipschitz or H\"older continuity of the objective, as assumed in \cite{guo2021existence,tao2025generalized}. Instead, the continuity of the conditional distributions and the continuity of the aleatoric risk functional under the total-variation distance imply uniform continuity of $v_{i,\mu_i^t}^t$ in the player's own type, which in turn yields equicontinuity of the set of optimal response functions needed for the Arzel\`a--Ascoli theorem and the continuity of the optimal response function mapping. 

\subsection{Uniqueness of BNE-RABL under an additional dominance condition
}

\label{sec:bne-rabl-banach-theorem}

We now proceed to discuss the uniqueness of BNE-RABL.
We focus on
the case where the effects of a player's own action dominate the aggregate effects of the rivals' actions.
Rather than postulating this dominance condition directly, we formulate the following assumptions on each component of the objective function. 
These assumptions yield sufficient conditions for the desired dominance property, each of which can be verified separately in applications.

\begin{assumption}
\label{assump:rabl-contraction}
Let $i,j\in N$.
(a) The loss function
$c_i(a_i,a_{-i},\theta_i,\theta_{-i})$ is differentiable in $a_i$.
(b) For each $j\neq i$, there exist constants $L_{ij}^{c,0}\geq0$ and
$L_{ij}^{c,1}\geq0$ such that 
\begin{eqnarray*}
    \left|
c_i(a_i,a_{-i}',\theta_i,\theta_{-i}) - c_i(a_i,a_{-i},\theta_i,\theta_{-i})
\right|
\leq
\sum_{j\neq i}L_{ij}^{c,0}\|a_j' - a_j\|, \; \forall a_{-i},a_{-i}' \in \mathcal{A}_{-i}
\end{eqnarray*}
and
\begin{eqnarray*}
    \left\|
\nabla_{a_i}c_i(a_i,a_{-i}',\theta_i,\theta_{-i})
-
\nabla_{a_i}c_i(a_i,a_{-i},\theta_i,\theta_{-i})
\right\|
\leq
\sum_{j\neq i} L_{ij}^{c,1}\|a_j'- a_j\| ,\; \forall a_{-i},a_{-i}' \in \mathcal{A}_{-i}.
\end{eqnarray*}
(c) 
There exist constants
$\kappa_i^{{\rm al},0}\geq0$ and $\kappa_i^{{\rm al},1}\geq0$ such that
\begin{eqnarray*}
&&\left\| \nabla_{a_i} \rho_{\eta(\cdot\mid\theta_i;\zeta)}^{\rm al}
\left[ c_i\left( a_i, f_{-i}(\theta_{-i}), \theta_i, \theta_{-i}\right) \right]
-
\nabla_{a_i} \rho_{\eta(\cdot\mid\theta_i;\zeta)}^{\rm al}
\left[ c_i\left( a_i, f_{-i}'(\theta_{-i}), \theta_i, \theta_{-i} \right) \right] \right\|\nonumber\\
&&\quad\leq
\kappa_i^{{\rm al},1} \sup_{\theta_{-i}\in\Theta_{-i}}
\left\| \nabla_{a_i} c_i\left( a_i, f_{-i}(\theta_{-i}), \theta_i, \theta_{-i} \right)
-
\nabla_{a_i} c_i\left( a_i, f_{-i}'(\theta_{-i}), \theta_i, \theta_{-i}\right) \right\|\nonumber\\
&&\qquad+ \kappa_i^{{\rm al},0} \sup_{\theta_{-i}\in\Theta_{-i}} \left| c_i\left( a_i, f_{-i}(\theta_{-i}), \theta_i, \theta_{-i} \right)
-
c_i\left( a_i, f_{-i}'(\theta_{-i}), \theta_i, \theta_{-i} \right) \right|, \; \forall f_{-i},f_{-i}'\in\mathcal F_{-i}   \qquad 
\end{eqnarray*}
uniformly for $a_i\in\mathcal A_i$, $\theta_i\in\Theta_i$, and almost surely $\zeta\in\mathcal Z$.
(d)
There exist constants
$\kappa_i^{{\rm ep},0}\geq0$ and $\kappa_i^{{\rm ep},1}\geq0$ such that for any $f_{-i},f_{-i}'\in\mathcal F_{-i}$,
\begin{eqnarray}\label{eqn:ep-risk-gradient-stability-rabl}
&&\left\|
\nabla_{a_i} \rho_{\mu_i^t}^{\rm ep} \left( \rho_{\eta(\cdot\mid\theta_i;\zeta)}^{\rm al} \left[ c_i\left( a_i, f_{-i}(\theta_{-i}), \theta_i, \theta_{-i} \right) \right] \right)
-
\nabla_{a_i}\rho_{\mu_i^t}^{\rm ep}\left( \rho_{\eta(\cdot\mid\theta_i;\zeta)}^{\rm al} \left[ c_i\left( a_i, f_{-i}'(\theta_{-i}), \theta_i, \theta_{-i}\right)\right]\right)
\right\| \nonumber\\
&\leq&
\kappa_i^{{\rm ep},1}
\sup_{\zeta\in \mathcal{Z}}
\left\| \nabla_{a_i}\rho_{\eta(\cdot\mid\theta_i;\zeta)}^{\rm al}\left[ c_i\left( a_i, f_{-i}(\theta_{-i}), \theta_i, \theta_{-i} \right)\right]
-
\nabla_{a_i} \rho_{\eta(\cdot\mid\theta_i;\zeta)}^{\rm al} \left[c_i\left( a_i, f_{-i}'(\theta_{-i}), \theta_i, \theta_{-i}\right)\right]\right\| \nonumber\\
&&+
\kappa_i^{{\rm ep},0} \sup_{\zeta\in \mathcal{Z}} \left| \rho_{\eta(\cdot\mid\theta_i;\zeta)}^{\rm al} \left[ c_i\left( a_i, f_{-i}(\theta_{-i}), \theta_i, \theta_{-i} \right)\right]
-
\rho_{\eta(\cdot\mid\theta_i;\zeta)}^{\rm al} \left[ c_i\left( a_i, f_{-i}'(\theta_{-i}), \theta_i, \theta_{-i} \right)\right]
\right|,
\end{eqnarray}
uniformly for $a_i\in\mathcal A_i$ and $\theta_i\in\Theta_i$.
(e) Let
$K_{ij}
    :=
    \kappa_i^{{\rm ep},1}
    \left(
        \kappa_i^{{\rm al},1}L_{ij}^{c,1}
        +
        \kappa_i^{{\rm al},0}L_{ij}^{c,0}
    \right)
    +
    \kappa_i^{{\rm ep},0}L_{ij}^{c,0}$
    for $i\neq j$,
and $K_{ii}:=0.$
Let $\Gamma$ be the matrix with entries
$\Gamma_{ij}:=\frac{K_{ij}}{M_i}$.
The spectral radius of $\Gamma$, written $r(\Gamma)$, satisfies $r(\Gamma)<1.$
\end{assumption}

Condition~(b) requires Lipschitz continuity of player $i$'s loss function and marginal loss with respect to the rivals' actions. 
The constants $L_{ij}^{c,0}$ and $L_{ij}^{c,1}$ reflect the sensitivity of player $i$'s loss level and marginal loss with respect to player $j$'s action, respectively.
A similar blockwise Lipschitz condition is used in Assumption~3.2 in \cite{su2025existence} to control cross-player effects in the risk-neutral Bayesian game.
Conditions~(c) and (d) quantify how the risk measures for aleatoric and epistemic uncertainty amplify changes in the loss levels and in the marginal losses. 
Compared with the risk-neutral setting in \cite{su2025existence}, these constants are specific to the risk-averse nested objective. 
They are needed to translate original cross-effects in the loss function into cross-effects in the risk-adjusted marginal objective. 
In the risk-neutral case where both risk measures reduce to expectations, 
one may take
$\kappa_i^{{\rm al},1}=\kappa_i^{{\rm ep},1}=1,$ 
and 
$\kappa_i^{{\rm al},0} =\kappa_i^{{\rm ep},0}=0.$
We will derive the values of these constants for several commonly used risk measures in Section~\ref{subsec:price-competition-application}.

Condition~(e) is the analogue of the contraction condition in Theorem~4.1 of \cite{su2025existence} for the present risk-averse setting. 
A similar contraction condition is also used in \cite{lei2020synchronous} to
establish the convergence of inexact best-response schemes for stochastic Nash games.
A convenient sufficient condition for $r(\Gamma)<1$ is
\begin{equation}
\label{eqn:contraction-condition-sufficient}
    \max_{i\in N} \sum_{j\neq i} \frac{K_{ij}}{M_i} <1.
\end{equation}
Since $\Gamma$ is a nonnegative matrix, its spectral radius is bounded above by its maximum row sum.
Recall that the constant $M_i$ can be viewed as a lower bound on the sensitivity of player $i$'s marginal loss with respect to its own action; 
the constants $L_{ij}^{c,0}$ and $L_{ij}^{c,1}$ bound the sensitivity of player $i$'s loss level and marginal loss with respect to player $j$'s action;
and the constants
$\kappa_i^{{\rm al},0}$, $\kappa_i^{{\rm al},1}$,
$\kappa_i^{{\rm ep},0}$, and $\kappa_i^{{\rm ep},1}$ describe how the two risk
measures amplify perturbations in loss levels and marginal losses. 
Therefore, $\Gamma_{ij}$ represents the ratio between the risk-adjusted cross-player marginal effect from player $j$ to player $i$ and the marginal effect from player $i$'s own action,
and thus condition~\eqref{eqn:contraction-condition-sufficient} requires the aggregate risk-adjusted cross-player effect to be dominated by the player's own-action effect.
We will illustrate how these conditions can be verified in the price competition application in Section~\ref{subsec:price-competition-application}.
We refer readers to Section~4 in \cite{su2025existence} for a detailed discussion of the contraction condition in the risk-neutral case.

\begin{lemma}
\label{lem:rabl-induced-cross-gradient}
Suppose Assumption~\ref{assump:rabl-contraction} holds. 
Let  $i\in N$.
Then 
\begin{eqnarray}
    \label{eqn:induced-cross-gradient-rabl}
\left\| \nabla_{a_i} v_{i,\mu_i^t}^t(a_i,f_{-i},\theta_i) - \nabla_{a_i} v_{i,\mu_i^t}^t(a_i,f_{-i}',\theta_i) \right\| \leq \sum_{j\neq i}
K_{ij} \|f_j-f_j'\|_\infty, \qquad \forall f_{-i},f_{-i}'\in\mathcal C_{-i},
\end{eqnarray}
uniformly for $a_i\in\mathcal A_i$, $\theta_i\in\Theta_i$.
If, in addition, Assumption~\ref{assump:rabl-common-primitive}(a) holds, then 
\begin{eqnarray}
    \label{eqn:component-contraction-rabl}
    \left\| \Psi_{i,\mu^t}^t(f)(\theta_i)-\Psi_{i,\mu^t}^t(f')(\theta_i) \right\|
    \leq \sum_{j\neq i} \Gamma_{ij} \|f_j-f_j'\|_\infty .
\end{eqnarray}
\end{lemma}

\noindent
\textbf{Proof.}
By Assumption~\ref{assump:rabl-contraction}(b-c), for any $\zeta \in \mathcal{Z}$, we immediately arrive at
\begin{eqnarray}
    \label{eqn:proof-inner-gradient-bound}
&&\left\|
\nabla_{a_i} \rho_{\eta(\cdot\mid\theta_i;\zeta)}^{\rm al} \left( c_i\left( a_i, f_{-i}(\theta_{-i}), \theta_i, \theta_{-i} \right) \right)
-
\nabla_{a_i} \rho_{\eta(\cdot\mid\theta_i;\zeta)}^{\rm al} \left( c_i\left( a_i, f_{-i}'(\theta_{-i}), \theta_i, \theta_{-i} \right) \right)
\right\|\nonumber\\
&&\qquad\leq
\sum_{j\neq i} \left( \kappa_i^{{\rm al},1}L_{ij}^{c,1} + \kappa_i^{{\rm al},0}L_{ij}^{c,0} \right) \|f_j-f_j'\|_\infty,
\end{eqnarray}
and 
\begin{eqnarray} \label{eqn:proof-inner-value-bound}
&&\left|
\rho_{\eta(\cdot\mid\theta_i;\zeta)}^{\rm al}
\left( c_i\left( a_i, f_{-i}(\theta_{-i}), \theta_i, \theta_{-i} \right)\right)
-
\rho_{\eta(\cdot\mid\theta_i;\zeta)}^{\rm al}
\left( c_i\left( a_i, f_{-i}'(\theta_{-i}), \theta_i, \theta_{-i}\right) \right) \right| \nonumber\\
&\leq& \sup_{\theta_{-i}\in \Theta_{-i}} \left| c_i\left( a_i, f_{-i}(\theta_{-i}), \theta_i, \theta_{-i} \right) -  c_i\left( a_i, f_{-i}'(\theta_{-i}), \theta_i, \theta_{-i}\right) \right| \nonumber\\
&\leq &
\sum_{j\neq i}
L_{ij}^{c,0}\|f_j-f_j'\|_\infty .    
\end{eqnarray}
Finally, by \eqref{eqn:ep-risk-gradient-stability-rabl} in Assumption~\ref{assump:rabl-contraction}(d), and using
\eqref{eqn:proof-inner-gradient-bound} and \eqref{eqn:proof-inner-value-bound}, we immediately
arrive at \eqref{eqn:induced-cross-gradient-rabl}.

For $\theta_i\in\Theta_i$, and $f_{-i},f_{-i}'\in\mathcal F_{-i}$, let
$a_i^*:=\Psi_{i,\mu^t}^t(f_{-i})(\theta_i),$ and $a_i^{\prime *}:=\Psi_{i,\mu^t}^t(f_{-i}')(\theta_i).$
If $a_i^*=a_i^{\prime *}$, then \eqref{eqn:component-contraction-rabl} is trivial.
Suppose now that $a_i^*\neq a_i^{\prime *}$.
Since $a_i^*$ minimizes $v_{i,\mu_i^t}^t(a_i,f_{-i},\theta_i)$
over  $\mathcal A_i$, the first-order optimality condition gives
\begin{eqnarray} \label{eqn:proof-expansive-1}
        \nabla_{a_i}v_{i,\mu_i^t}^t(a_i^*,f_{-i},\theta_i)^\top
    (a_i^{\prime *}-a_i^*)\geq0 .
\end{eqnarray}
Similarly, since $a_i^{\prime *}$ minimizes $v_{i,\mu_i^t}^t(a_i,f_{-i}',\theta_i)$
over $\mathcal A_i$, we have
\begin{eqnarray} \label{eqn:proof-expansive-2}
    \nabla_{a_i}v_{i,\mu_i^t}^t(a_i^{\prime *},f_{-i}',\theta_i)^\top
    (a_i^*-a_i^{\prime *})\geq0 .
\end{eqnarray}
Combining \eqref{eqn:proof-expansive-1} and \eqref{eqn:proof-expansive-2}, we have
\begin{eqnarray} \label{eqn:proof-expansive-3}
\left(
\nabla_{a_i}v_{i,\mu_i^t}^t(a_i^*,f_{-i},\theta_i)
-
\nabla_{a_i}v_{i,\mu_i^t}^t(a_i^{\prime *},f_{-i}',\theta_i)
\right)^\top
(a_i^*-a_i^{\prime *}) \leq 0.
\end{eqnarray}
By the strong convexity of
$v_{i,\mu_i^t}^t(a_i,f_{-i}',\theta_i)$ in $a_i$ under Assumption~\ref{assump:rabl-common-primitive}(a),
its gradient is strongly monotonic. 
Hence
\begin{eqnarray} \label{eqn:proof-expansive-4}
&&
\left(
\nabla_{a_i}v_{i,\mu_i^t}^t(a_i^*,f_{-i}',\theta_i)
-
\nabla_{a_i}v_{i,\mu_i^t}^t(a_i^{\prime *},f_{-i}',\theta_i)
\right)^\top
(a_i^*-a_i^{\prime *})  \geq
M_i\|a_i^*-a_i^{\prime *}\|^2 .
\end{eqnarray}
On the other hand,
\begin{eqnarray} \label{eqn:proof-expansive-5}
&&
\left(
\nabla_{a_i}v_{i,\mu_i^t}^t(a_i^*,f_{-i}',\theta_i)
-
\nabla_{a_i}v_{i,\mu_i^t}^t(a_i^{\prime *},f_{-i}',\theta_i)
\right)^\top
(a_i^*-a_i^{\prime *}) \nonumber\\
&=&
\left(
\nabla_{a_i}v_{i,\mu_i^t}^t(a_i^*,f_{-i}',\theta_i)
-
\nabla_{a_i}v_{i,\mu_i^t}^t(a_i^*,f_{-i},\theta_i)
\right)^\top
(a_i^*-a_i^{\prime *}) \nonumber\\
&&+
\left(
\nabla_{a_i}v_{i,\mu_i^t}^t(a_i^*,f_{-i},\theta_i)
-
\nabla_{a_i}v_{i,\mu_i^t}^t(a_i^{\prime *},f_{-i}',\theta_i)
\right)^\top
(a_i^*-a_i^{\prime *})\nonumber\\
&\leq & \left\|
\nabla_{a_i}v_{i,\mu_i^t}^t(a_i^*,f_{-i}',\theta_i)
-
\nabla_{a_i}v_{i,\mu_i^t}^t(a_i^*,f_{-i},\theta_i)
\right\|
\|a_i^*-a_i^{\prime *}\|,
\end{eqnarray}
where the last inequality is due to \eqref{eqn:proof-expansive-3}.
Combining \eqref{eqn:proof-expansive-4} and \eqref{eqn:proof-expansive-5},
we have 
\begin{eqnarray*}
    \|a_i^*-a_i^{\prime *}\|
    \leq
    \frac{1}{M_i}
    \left\|
    \nabla_{a_i}v_{i,\mu_i^t}^t(a_i^*,f_{-i}',\theta_i)
    -
    \nabla_{a_i}v_{i,\mu_i^t}^t(a_i^*,f_{-i},\theta_i)
    \right\|.
\end{eqnarray*}
Applying \eqref{eqn:induced-cross-gradient-rabl} yields
\[
    \left\|
        \Psi_{i,\mu^t}^t(f)(\theta_i)-\Psi_{i,\mu^t}^t(f')(\theta_i)
    \right\|
    =
    \|a_i^*-a_i^{\prime *}\|
    \leq
    \sum_{j\neq i}
    \frac{K_{ij}}{M_i}
    \|f_j-f_j'\|_\infty,
\]
which is
\eqref{eqn:component-contraction-rabl}.
\hfill $\Box$

Based on the contraction result in Lemma~\ref{lem:rabl-induced-cross-gradient}, 
we now establish the uniqueness of the BNE-RABL.

\begin{theorem}
\label{thm:rabl-contraction}
Under Assumptions~\ref{assump:rabl-common-primitive}(a) and \ref{assump:rabl-contraction},
the BNE-RABL model  \eqref{eqn:BNE-heter-belief}
has at most one equilibrium in round $t$.
\end{theorem}

\noindent
\textbf{Proof.}
Let $f$ and $f'$ be two equilibria under the profile of posterior distributions $\mu^t$. Then
$f=\Psi_{\mu^t}^t(f)$ and $f'=\Psi_{\mu^t}^t(f')$.
Let $d=(d_1,\ldots,d_n)^\top\in\mathbb R_+^n$ and
$ d_i:=\|f_i-f_i'\|_\infty,$ for $i\in N.$
By Lemma~\ref{lem:rabl-induced-cross-gradient}, for every $i\in N$,
\begin{eqnarray*}
    d_i = \left\| \Psi_{i,\mu_i^t}^t(f) - \Psi_{i,\mu_i^t}^t(f') \right\|_\infty 
    \leq \sum_{j\ne i} \Gamma_{ij} \|f_j-f_j'\|_\infty = (\Gamma d)_i,
\end{eqnarray*}
and thus
\begin{eqnarray} \label{proof:uniqueness}
    0\le d\le\Gamma d \leq \Gamma^k d,\quad \forall \; k =1,2,3,\dots,
\end{eqnarray}
where the inequalities hold componentwise.
By Assumption~\ref{assump:rabl-contraction}(e), $r(\Gamma)<1$ implies $\lim_{k\to\infty}\Gamma^k=0$.
Combining this result with \eqref{proof:uniqueness}, we conclude that $d = 0$ and thereby $f = f'$.
\hfill$\Box$


\subsection{Existence and uniqueness of BNE-BPD}
\label{sec:exist_uni_BNE-BPD}

In this section, we study the existence and uniqueness of the BNE-BPD model. 
To avoid repetition, we state only the corresponding assumptions and results, 
and focus on establishing the continuity of the conditional BPD, 
which differs from the preceding analysis and serves as the key ingredient for the existence of the BNE-BPD.

\begin{assumption} \label{assumption:bpd-existence}
    Let $t$ be fixed. 
        (a) $p_i(\theta_i;\zeta)$ is continuous in $\theta_i$ for all $\zeta\in \mathcal{Z}$, 
        and $\underline{q}_i^t := \inf_{\theta_i\in \Theta_i} \int_{\mathcal{Z}} p_i(\theta_i;\zeta) m_i^t(\zeta)d\zeta>0$.
        (b) $\rho^{\rm bp}$ is a law-invariant, convex risk measure, and the induced risk functional $\varrho^{\rm bp}$ is continuous in the underlying probability distribution
        under the total variation metric.
    
\end{assumption}

Assumption~\ref{assumption:bpd-existence} 
corresponds to Assumption~\ref{assump:rabl-common-primitive}(c)-(d). 
The newly introduced condition~(a) is a pair of conditions on the marginal densities of the underlying distribution and the predictive distribution, 
which will be used to derive the continuity of the conditional predictive distribution in $\theta_i$ (see the lemma below) 
and thereby the continuity of the objective function in~$\theta$.

\begin{lemma} \label{lemma:bpd-continuity}
    Suppose that  Assumptions~\ref{basic-assumption}(a) and \ref{assumption:bpd-existence} hold.
    Then for $i\in N$, $\nu_i^t(\cdot|\theta_i)$ is continuous in $\theta_i$ under the total variation metric.
\end{lemma}

\noindent
\textbf{Proof.}
Let 
\begin{eqnarray*}
    w_{i}^t(\theta_i;\zeta) := \frac{p_i(\theta_i;\zeta)}{\int_{\Theta_{-i}} q_i^t(\theta_i,\theta_{-i}) d\theta_{-i}}
    = \frac{p_i(\theta_i;\zeta)} {\int_{\Theta_{-i}}\int_{\mathcal{Z}}
        p(\theta_i,\theta_{-i};\zeta)m_i^t(\zeta)\,d\zeta\,d\theta_{-i}}.
\end{eqnarray*}
Then the conditional density function of the BPD defined in \eqref{eqn:conditional-bpd} can be written as
\begin{eqnarray} \label{eqn:conditional-q_i^t}
    q_i^t(\theta_{-i}\mid \theta_i)
    &:=&
    \frac{q_i^t(\theta_i,\theta_{-i})}
    {\int_{\Theta_{-i}}q_i^t(\theta_i,\theta_{-i})d\theta_{-i}}
    =
    \frac{
        \int_{\mathcal{Z}}p(\theta_i,\theta_{-i};\zeta)m_i^t(\zeta)\,d\zeta
    }{
        \int_{\Theta_{-i}}\int_{\mathcal{Z}}
        p(\theta_i,\theta_{-i};\zeta)m_i^t(\zeta)\,d\zeta\,d\theta_{-i}
    } \nonumber\\
    &=& \int_{\mathcal{Z}}p(\theta_{-i}|\theta_{i};\zeta) w_{i}^t(\theta_i;\zeta) m_i^t(\zeta)\,d\zeta.
\end{eqnarray}
Let $\theta_i',\theta_i\in \Theta_i$. For every measurable function $h:\Theta_{-i}\rightarrow\R$ with $\|h\|_\infty \leq 1$, 
we can use Fubini-Tonelli theorem and 
\eqref{eqn:conditional-q_i^t} 
to establish
\begin{eqnarray}
\label{eqn:lemma5.4-proof1}
    &&\left| \mathbb{E}_{\nu_i^t(\cdot|\theta_i)} [h(\theta_{-i})] - \mathbb{E}_{\nu_i^t(\cdot|\theta_i')} [h(\theta_{-i})] \right|\nonumber\\
    &=& \left| \int_{\Theta_{-i}} h(\theta_{-i}) \int_{\mathcal{Z}}p(\theta_{-i}|\theta_{i};\zeta) w_{i}^t(\theta_i;\zeta) m_i^t(\zeta)\,d\zeta d\theta_{-i} \right.
    \nonumber\\
    && \left. -\int_{\Theta_{-i}} h(\theta_{-i}) \int_{\mathcal{Z}}p(\theta_{-i}|\theta_{i}';\zeta) w_{i}^t(\theta_i';\zeta) m_i^t(\zeta)\,d\zeta d\theta_{-i} \right|\nonumber\\
    &=& 
    \left|
    \int_{\mathcal{Z}} \mathbb{E}_{\eta(\cdot|\theta_i;\zeta)} \left[h(\theta_{-i}) \right] w_{i}^t(\theta_i;\zeta)  m_i^t(\zeta)\,d\zeta 
    -
    \int_{\mathcal{Z}} \mathbb{E}_{\eta(\cdot|\theta_i';\zeta)} \left[h(\theta_{-i}) \right] w_{i}^t(\theta_i';\zeta)  m_i^t(\zeta)\,d\zeta 
    \right|\nonumber\\
    &\leq & 
    \int_{\mathcal{Z}} \left| \mathbb{E}_{\eta(\cdot|\theta_i;\zeta)} \left[h(\theta_{-i}) \right] - \mathbb{E}_{\eta(\cdot|\theta_i';\zeta)} \left[h(\theta_{-i}) \right] \right| w_{i}^t(\theta_i;\zeta)  m_i^t(\zeta)\,d\zeta \nonumber\\
    &&
    +
    \int_{\mathcal{Z}} |\mathbb{E}_{\eta(\cdot|\theta_i';\zeta)} \left[h(\theta_{-i}) \right]| \left| w_{i}^t(\theta_i;\zeta) - w_{i}^t(\theta_i';\zeta) \right|  m_i^t(\zeta)\,d\zeta \nonumber\\
    &\leq& \sup_{\zeta\in\mathcal{Z}} \dd_{TV}(\eta(\cdot|\theta_i;\zeta), \eta(\cdot|\theta_i';\zeta)) + \int_{\mathcal{Z}} \left| w_{i}^t(\theta_i;\zeta) - w_{i}^t(\theta_i';\zeta) \right|  m_i^t(\zeta)\,d\zeta,
\end{eqnarray}
where, in the last inequality, the first term comes from the fact that $\int_{\mathcal{Z}} w_{i}^t(\theta_i;\zeta)  m_i^t(\zeta)\,d\zeta = 1$ and the second term from $\|h\|_\infty \leq 1$.
In what follows, we estimate 
the two terms at the rhs of the inequality in  \eqref{eqn:lemma5.4-proof1}.
Assumption~\ref{assump:rabl-common-primitive}(b) 
ensures that $\eta(\cdot|\theta_i;\zeta)$ is continuous in $\theta_i$ and hence that the first term can be arbitrarily small when $\theta_i'$ is sufficiently close to $\theta_i$.
For the second term,
\begin{eqnarray}
\label{eqn:lemma5.4-proof2}
    &&\int_{\mathcal{Z}} \left| w_{i}^t(\theta_i;\zeta) - w_{i}^t(\theta_i';\zeta) \right|  m_i^t(\zeta)\,d\zeta\nonumber\\
    &=& \int_{\mathcal{Z}} \left| \frac{p_i(\theta_i;\zeta)}{\int_{\Theta_{-i}} q_i^t(\theta_i,\theta_{-i}) d\theta_{-i}} - \frac{p_i(\theta_i';\zeta)}{\int_{\Theta_{-i}} q_i^t(\theta_i',\theta_{-i}) d\theta_{-i}} \right|  m_i^t(\zeta)\,d\zeta\nonumber\\
    & \leq & \frac{1}{\int_{\Theta_{-i}} q_i^t(\theta_i,\theta_{-i}) d\theta_{-i}}\int_{\mathcal{Z}} \left| p_i(\theta_i;\zeta) - p_i(\theta_i';\zeta)\right| m_i^t(\zeta)\,d\zeta \nonumber\\
    &&+  \left| \frac{1}{\int_{\Theta_{-i}} q_i^t(\theta_i,\theta_{-i}) d\theta_{-i}}
    - \frac{1}{\int_{\Theta_{-i}}  q_i^t(\theta_i',\theta_{-i}) d\theta_{-i}}  \right|
    \int_{\mathcal{Z}} p_i(\theta_i';\zeta)  m_i^t(\zeta)\,d\zeta \nonumber\\
    &\leq&  \frac{1}{\int_{\Theta_{-i}} q_i^t(\theta_i,\theta_{-i}) d\theta_{-i}}\int_{\mathcal{Z}} \left| p_i(\theta_i;\zeta) - p_i(\theta_i';\zeta)\right| m_i^t(\zeta)\,d\zeta +  \left| 
    \frac
    {\int_{\Theta_{-i}}  q_i^t(\theta_i',\theta_{-i}) d\theta_{-i} 
    }
    {\int_{\Theta_{-i}} q_i^t(\theta_i,\theta_{-i}) d\theta_{-i}  
    } -1
    \right| 
    \nonumber\\
    &\leq& \frac{2}{\int_{\Theta_{-i}} q_i^t(\theta_i,\theta_{-i}) d\theta_{-i}}\int_{\mathcal{Z}} \left| p_i(\theta_i;\zeta) - p_i(\theta_i';\zeta)\right| m_i^t(\zeta)\,d\zeta\nonumber\\
    &\leq & \frac{2}{\underline{q}_i^t}\sup_{\mathcal{Z}} \left| p_i(\theta_i;\zeta) - p_i(\theta_i';\zeta)\right|,
\end{eqnarray}
where the second and third inequalities use the fact that $\int_{\Theta_{-i}}  q_i^t(\theta_i',\theta_{-i}) d\theta_{-i} =  \int_{\mathcal{Z}}  p_i(\theta_i';\zeta) m_i^t(\zeta)\,d\zeta$, and
the last inequality is due to $\int_{\mathcal{Z}}m_i^t(\zeta)\,d\zeta = 1$ and $\underline{q}_i^t = \inf_{\theta_i\in \Theta_i} \int_{\Theta_{-i}} q_i^t(\theta_i,\theta_{-i}) d\theta_{-i}$. 
Under Assumption~\ref{assumption:bpd-existence}(a), 
the rhs of \eqref{eqn:lemma5.4-proof2} can be arbitrarily small when $\|\theta_i'-\theta_i\|$ is sufficiently small.
Note that
\begin{eqnarray} \label{eqn:lemma5.4-proof3}
    \dd_{TV} \left(\nu_i^t(\cdot|\theta_i'), \nu_i^t(\cdot|\theta_i) \right) &=& \sup_{h\in\mathcal{M},\|h\|_\infty\leq 1}\frac{1}{2}\left|
    \mathbb{E}_{\nu_i^t(\cdot|\theta_i')} [h(\theta_{-i})] - \mathbb{E}_{\nu_i^t(\cdot|\theta_i)} [h(\theta_{-i})]
    \right| 
    .
\end{eqnarray}
Combining \eqref{eqn:lemma5.4-proof1}, \eqref{eqn:lemma5.4-proof2}, and \eqref{eqn:lemma5.4-proof3},
we obtain  the continuity of $\nu_i^t(\cdot|\theta_i)$ in $\theta_i$.
\hfill $\Box$

With Assumption~\ref{assumption:bpd-existence} and Lemma~\ref{lemma:bpd-continuity} in place, 
we have obtained a counterpart to Assumption~\ref{assump:rabl-common-primitive}(b).
The existence of the BNE-BPD can therefore be established by following the same line of Section~\ref{sec:exist-bne-rabl}.
To avoid repetition, we present only the theorem and omit the proof.

\begin{theorem}
\label{thm:bpd-existence}
Suppose that Assumptions~\ref{assump:rabl-common-primitive}(a)-(b) and \ref{assumption:bpd-existence} hold. Then, for the fixed round $t$, the BNE-BPD model \eqref{eqn:BNE-BPD}
has a continuous equilibrium.
\end{theorem}

Moreover, as in Section~\ref{sec:bne-rabl-banach-theorem} where the effect of a player's own action dominates the aggregate effect of the rivals' actions, 
we can 
establish the uniqueness of the equilibrium under the following assumption.

\begin{assumption} \label{assumption:bpd-uniqueness}
    Let $i\in N$. 
        (a) For $\rho^{\rm bp}$, there exist constants
$\kappa_i^{{\rm bp},0}\geq0$ and $\kappa_i^{{\rm bp},1}\geq0$ such that for any $f_{-i},f_{-i}'\in\mathcal F_{-i}$,
\begin{eqnarray*}
&&\left\| \nabla_{a_i} \rho_{\nu_i^t(\cdot|\theta_i)}^{\rm bp}
\left( c_i\left( a_i, f_{-i}(\theta_{-i}), \theta_i, \theta_{-i}\right) \right)
-
\nabla_{a_i} \rho_{\nu_i^t(\cdot|\theta_i)}^{\rm bp}
\left( c_i\left( a_i, f_{-i}'(\theta_{-i}), \theta_i, \theta_{-i} \right) \right) \right\|\nonumber\\
&&\quad\leq
\kappa_i^{{\rm bp},1} \sup_{\theta_{-i}\in\Theta_{-i}}
\left\| \nabla_{a_i} c_i\left( a_i, f_{-i}(\theta_{-i}), \theta_i, \theta_{-i} \right)
-
\nabla_{a_i} c_i\left( a_i, f_{-i}'(\theta_{-i}), \theta_i, \theta_{-i}\right) \right\|\nonumber\\
&&\qquad+ \kappa_i^{{\rm bp},0} \sup_{\theta_{-i}\in\Theta_{-i}} \left| c_i\left( a_i, f_{-i}(\theta_{-i}), \theta_i, \theta_{-i} \right)
-
c_i\left( a_i, f_{-i}'(\theta_{-i}), \theta_i, \theta_{-i} \right) \right|,    
\end{eqnarray*}
uniformly for $a_i\in\mathcal A_i$ and $\theta_i\in\Theta_i$.
        (b) Let
$K_{ij}^{\rm bp}
    :=
        \kappa_i^{{\rm bp},1}L_{ij}^{c,1}
        +
        \kappa_i^{{\rm bp},0}L_{ij}^{c,0}$
    for $i\neq j$,
and $K_{ii}^{\rm bp}:=0.$
Let $\Gamma^{\rm bp}$ be the matrix with entries
$\Gamma_{ij}^{\rm bp}:=\frac{K_{ij}^{\rm bp}}{M_i}$.
The spectral radius of $\Gamma^{\rm bp}$ satisfies $r(\Gamma^{\rm bp})<1.$
\end{assumption}

Assumption~\ref{assumption:bpd-uniqueness} retains the conditions on the loss function in Assumption~\ref{assump:rabl-contraction}.
Since the objective function in BNE-BPD involves only a single risk measure $\rho^{\rm bp}$, which is applied directly to the loss function, we adapt the condition on $\rho^{\rm al}$ to $\rho^{\rm bp}$ and modify the definition of $K^{\rm bp}$ accordingly.
We are now ready to state the uniqueness of BNE-BPD.

\begin{theorem}
\label{thm:bpd-contraction}
Suppose Assumptions~\ref{assump:rabl-common-primitive}(a), \ref{assump:rabl-contraction}(a)-(b) and \ref{assumption:bpd-uniqueness}
hold. 
Then, for a fixed round $t$, the BNE-BPD model in Definition \ref{def:BNE-BPD} has at most one equilibrium.
\end{theorem}

\section{Non-asymptotic convergence of BNE}
\label{sec:asymptotic-converge}

In this section, 
we study the non-asymptotic convergence of the equilibria generated 
by 
 the Bayesian learning process in repeated Bayesian games. 
This result relies on the convergence of the posterior distribution in Bayesian learning.
For asymptotic convergence,  \cite{gelman1995bayesian} provide the result that both the posterior mean and variance typically 
converge at a rate of $O(\frac{1}{t})$.
More recently, 
\cite{ma2024bayesian} and \cite{milz2026stochastic} prove  asymptotic convergence of the optimal values and optimal policies for Markov decision processes and stochastic optimal control within a Bayesian learning framework.
\cite{mou2024diffusion} establish the non-asymptotic convergence result that the posterior contraction rate around the true parameter is $t^{-1/2}$. Here, we use 
an intermediate result in the proof of Theorem~3.1 in \cite{mou2024diffusion} as an assumption to 
derive 
non-asymptotic convergence of BNE. %

\begin{assumption}%
\label{assumption:conv_rate_post_dist}
    For each player $i\in N$ and any $\epsilon\in (0,1)$,
    there exist positive constants $c_1>0$, $c_2(\epsilon)>0$, and $T(\epsilon)\geq 2$ such that, for any $t>T(\epsilon)$,
    \begin{eqnarray}
        \label{eq:Assump-6.1}
        \inmat{Prob}_{\eta(\cdot;\zeta^*)}^{t-1}\left( \mathbb{E}_{\mu_i^t}\|\zeta - \zeta^*\|_2^2 \leq c_2(\epsilon) (t-1)^{-c_1} \right) \geq 1- \epsilon,
    \end{eqnarray}
    where the probability measure $\inmat{Prob}_{\eta(\cdot;\zeta^*)}^{ t-1}$ is understood as the $(t-1)$-product probability measure of $\eta(\cdot;\zeta^*)$ over the measurable space $\Theta\times \cdots \times \Theta $ with its product Borel $\sigma$-algebra.
\end{assumption}

\subsection{Non-asymptotic convergence of BNE-RABL}
\label{sec:convergence-bne-rabl}

We first introduce the notion of gradient regularity of the risk measures. 

\begin{definition} \label{def:gradient-regular}
    A risk measure is said to be \emph{gradient regular} 
    on a class of bounded differentiable losses if, 
    for any probability measure $P$ and any loss function $l:\mathcal{A}\times \R \rightarrow \R$ differentiable in $a$, 
    there exists a measurable function $\omega_{P}:\R\rightarrow \R$ such that
    \begin{eqnarray} \label{eqn:rm-gradient-representation}
        \nabla_a \rho_P(l(a,\xi)) = \mathbb{E}_P [\omega_{P}(\xi) \nabla_a l(a,\xi)],
    \end{eqnarray}
    where $\mathbb{E}_P[\omega_{P}(\xi)] = 1$ and $\omega_{P}(\xi)\in [0,\bar{\omega}]$ for some $\bar{\omega}> 0$. 
\end{definition}

The regularity conditions required in Definition~\ref{def:gradient-regular} 
are satisfied by many widely used risk measures and have been studied in the risk optimization literature. For instance, Corollary~3.3 in \cite{ruszczynski2006optimization} shows that, if $l$ is convex, continuous, and differentiable for $P$-almost every $\xi$ at $\bar{a}\in \mathcal{A}$, and $\rho$ is a convex risk measure whose subdifferential $\partial_l \rho(\bar{l})=\{\bar{P}\}$ is a singleton, where $\bar{l}:=l(\bar{a},\cdot)$, then
\begin{eqnarray*}
    \nabla_a \rho_P(l(\bar{a},\xi)) = \mathbb{E}_{\bar{P}} \left[\nabla_a l(\bar{a},\xi)\right].
\end{eqnarray*}
Moreover, if $\bar{P}\ll P$ and $\sup_{\xi}\frac{d\bar{P}(\xi)}{dP(\xi)}< +\infty$, then setting $\omega_P(\xi):= \frac{d\bar{P}(\xi)}{dP(\xi)}$ yields
\begin{eqnarray*}
    \nabla_a \rho_P(l(\bar{a},\xi)) 
    = \int \nabla_a l(\bar{a},\xi)\, d\bar{P}(\xi) 
    = \int \frac{d\bar{P}(\xi)}{dP(\xi)} \nabla_a l(\bar{a},\xi)\, dP(\xi)
    = \mathbb{E}_{P}\left[ \omega_P(\xi) \nabla_a l(\bar{a},\xi) \right]. 
\end{eqnarray*}
Since $\mathbb{E}_P[\omega_{P}(\xi)] = \int \frac{d\bar{P}(\xi)}{dP(\xi)}\, dP(\xi) = 1$ and $\omega_{P}(\xi)\in \left[0,\sup_{\xi}\frac{d\bar{P}(\xi)}{dP(\xi)}\right]$, the risk measure $\rho$ is gradient regular in the sense of Definition~\ref{def:gradient-regular}.

To establish the convergence of the equilibrium as $t\rightarrow \infty$, 
we make the following assumptions.

\begin{assumption} \label{assumption:gradient-regular}
    For each player $i\in N$, 
    assume that $\rho^{\rm al}$ is a gradient regular risk measure with 
    $\omega_{i,\eta(\cdot\mid\theta_i;\zeta)}^{\rm al} (\theta_{-i})\in \left[0,  \bar{\omega}_{i}^{\rm al}\right]$ for all $\theta_{-i}$, 
    and $\rho^{\rm ep}$ is a gradient regular risk measure with 
    $\omega_{i,\mu_i^t}^{\rm ep}(\zeta)\in \left[0, \bar{\omega}_{i}^{\rm ep}\right]$ for all $\zeta$. 
    Moreover, there exists a constant $\lambda_i^{\rm al}$ such that 
    \begin{eqnarray} \label{eqn:gradient-regular-tvnorm}
        \mathbb{E}_{\eta(\cdot\mid\theta_i;\zeta^*)} \left[ \left|\omega_{i,\eta(\cdot\mid\theta_i;\zeta)}^{\rm al}(\theta_{-i}) - \omega_{i,\eta(\cdot\mid\theta_i;\zeta^*)}^{\rm al} (\theta_{-i}) \right| \right] 
        \leq 
        \lambda_i^{\rm al}\, \dd_{TV} \left( \eta(\cdot\mid\theta_i;\zeta), \eta(\cdot\mid\theta_i;\zeta^*) \right).
    \end{eqnarray}
\end{assumption}

 The following lemma shows that the expectation, the entropic risk measure, and the CVaR are all gradient regular, and specifies the corresponding parameters.
As the arguments are straightforward, we omit the proof.

\begin{lemma}
    \begin{itemize}
        \item[(i)] The expectation is gradient regular with $\omega_{P}(\xi) = 1$ and $\bar{\omega} = 1$. 
        Moreover, \eqref{eqn:gradient-regular-tvnorm} holds with $\lambda_i^{\rm al} = 0$.

        \item[(ii)] If $|l(a,\xi)|\leq \bar{l}$ for some $\bar{l}>0$, 
        then the entropic risk measure is gradient regular with $\omega_{P}(\xi) = \frac{e^{\gamma l(a,\xi)}}{\mathbb{E}_P\left[ e^{\gamma l(a,\xi)} \right]}$ and $\bar{\omega} = e^{2\gamma \bar{l}}$. 
        Moreover, \eqref{eqn:gradient-regular-tvnorm} holds with $\lambda_i^{\rm al} = 2e^{2\gamma \bar{l}}$.

        \item[(iii)] Suppose the underlying distribution $P$ admits a continuous density, 
        and let $q_{\beta,P,l}$ denote the $\beta$-quantile of the loss $l$ under $P$. 
        Then the CVaR is gradient regular with $\omega_{P}(\xi) = \frac{1}{1-\beta}\mathds{1}_{\{l(a,\xi)> q_{\beta,P,l}\}}$ and $\bar{\omega} = \frac{1}{1-\beta}$. 
        Moreover, if the density of $P$ around $q_{\beta,\eta(\cdot\mid\theta_i;\zeta),l}$ and $q_{\beta,\eta(\cdot\mid\theta_i;\zeta^*),l}$ is bounded above and below by 
        $\bar{p}_{\beta}$ and $\underline{p}_{\beta}$, respectively, 
        then \eqref{eqn:gradient-regular-tvnorm} holds with $\lambda_i^{\rm al} = \frac{\bar{p}_{\beta}}{(1-\beta)\, \underline{p}_{\beta}}$. 
    \end{itemize}
\end{lemma}

Next, we impose a stability condition on the underlying conditional distribution $\eta(\cdot\mid\theta_i;\zeta)$
with respect to the parameter $\zeta$
in a neighborhood of the true parameter $\zeta^*$.

\begin{assumption}
\label{assumption:lipschitz_posterior_distribution}
    For each $i\in N$ and $\theta_i\in \Theta_i$, 
    there exist constants $L_{\eta,i}>0$, $r_i\in (0,2]$, 
    and $\epsilon_i>0$ such that
    \begin{eqnarray} \label{eqn:lipschitz_posterior_distribution}
        \dd_{TV} \left(\eta(\cdot\mid\theta_i;\zeta), \eta(\cdot\mid\theta_i;\zeta^*)\right) \leq L_{\eta,i} \| \zeta - \zeta^* \|^{r_i}, \quad \forall\, \zeta\in \mathbb{B}(\zeta^*,\epsilon_i).
    \end{eqnarray}
\end{assumption}

While Assumption~\ref{assump:rabl-common-primitive}(b) requires the continuity of the conditional distribution $\eta(\cdot\mid\theta_i;\zeta)$ in $\theta_i$, 
Assumption~\ref{assumption:lipschitz_posterior_distribution} imposes continuity in the parameter $\zeta$ near $\zeta^*$. 
The latter can be verified directly from the conditional density, as it holds with $r_i=1$ whenever
\begin{eqnarray*}
    \frac{1}{2}
    \int_{\Theta_{-i}}
    \left| p(\theta_{-i}\mid\theta_i;\zeta) - p(\theta_{-i}\mid\theta_i;\zeta') \right|\, d\theta_{-i}
    \leq
    L_{\eta,i}\,\|\zeta-\zeta'\|
\end{eqnarray*}
holds uniformly in $\theta_i$. 
This condition is satisfied by a broad class of parametric families commonly employed in Bayesian learning, 
in particular members of the exponential family,
including the Gaussian distribution  \citep{devroye2018total,arbas2023polynomial} 
and the exponential distribution \citep{wainwright2008graphical,mahpud2025differentially}.

We further assume that the marginal loss is uniformly bounded for each player $i$, which is typically the case in practical problems whenever $\mathcal{A}$ and $\Theta$ are compact and $\nabla_{a_i} c$ is continuous.

\begin{assumption} \label{assumption:marginal-cost-bounded}
    For each $i\in N$, there exists $\bar{g}_i$ such that 
    \begin{eqnarray*}
        \sup_{a\in \mathcal{A}, \theta\in \Theta} \|\nabla_{a_i} c_i(a_i, a_{-i}, \theta_i, \theta_{-i})\| \leq \bar{g}_i.
    \end{eqnarray*}
\end{assumption}

Building on the preceding three assumptions, we first establish stability results for two quantities: the marginal risk in a player's own action induced by the aleatoric uncertainty, and the marginal risk induced jointly by the aleatoric and epistemic uncertainties under the formulation of nested risk measures.

\begin{lemma} \label{lemma:gradient-al-rm-tv}
    Let 
    Assumptions~\ref{assumption:gradient-regular}-\ref{assumption:marginal-cost-bounded} 
    hold. 
    Then, for each $i\in N$, $\theta_i\in \Theta_i$, and $\zeta\in \mathbb{B}(\zeta^*,\epsilon_i)$,
    \begin{eqnarray} \label{eqn:gradient-al-rm-tv}
        &&\left\| \nabla_{a_i} \rho_{\eta(\cdot\mid\theta_i;\zeta)}^{\rm al} 
        \left(c_i\left(a_i,f_{-i}(\theta_{-i}),\theta_i,\theta_{-i} \right) \right)
        - \nabla_{a_i} \rho_{\eta(\cdot\mid\theta_i;\zeta^*)}^{\rm al} 
        \left(c_i\left(a_i,f_{-i}(\theta_{-i}),\theta_i,\theta_{-i} \right) \right) \right\| 
          \nonumber \\
        && \qquad \leq L_{\eta,i} (\bar\omega_i^{\rm al}\bar g_i 
		+
		\bar g_i\lambda_i^{\rm al}) \|\zeta - \zeta^*\|^{r_i},
    \end{eqnarray}
    and 
    \begin{eqnarray} \label{eqn:gradient-ep-rm-tv} 
        &&\left\|
\nabla_{a_i}v_{i,\mu_i^t}(a_i,f_{-i},\theta_i)
-
\nabla_{a_i}v_{i,\delta_{\zeta^*}}(a_i,f_{-i},\theta_i)
\right\| \nonumber\\
&& \qquad \leq \bar{\omega}_i^{\rm ep}  L_{\eta,i} (\bar\omega_i^{\rm al}\bar g_i 
		+
		\bar g_i\lambda_i^{\rm al}) 
        E_{\mu_i^t} \left[  \|\zeta - \zeta^*\|^2 \right] ^{\frac{r_i}{2}} 
        + 2 \frac{\bar{\omega}_i^{\rm ep} \bar{g}_i}{\epsilon_i^2} 
        E_{\mu_i^t} \left[  \|\zeta - \zeta^*\|^2 \right].
    \end{eqnarray}
\end{lemma}

\noindent
\textbf{Proof.}
We begin with the aleatoric risk measure $\rho^{\rm al}$. By Assumption~\ref{assumption:gradient-regular},
\begin{eqnarray} \label{eqn:proof-gradient-aleatoric-tv}
    &&\left\| \nabla_a \rho_{\eta(\cdot\mid\theta_i;\zeta)}^{\rm al} 
        \left(c_i\left(a_i,f_{-i}(\theta_{-i}),\theta_i,\theta_{-i} \right) \right)
        - \nabla_a \rho_{\eta(\cdot\mid\theta_i;\zeta^*)}^{\rm al} 
        \left(c_i\left(a_i,f_{-i}(\theta_{-i}),\theta_i,\theta_{-i} \right) \right) \right\|\nonumber\\
    & = & \Big\| \mathbb E_{\eta(\cdot\mid\theta_i;\zeta)}
		\left[
		\omega_{i,\eta(\cdot\mid\theta_i;\zeta)}^{\rm al}(\theta_{-i})
		\nabla_{a_i}c_i(a_i,f_{-i}(\theta_{-i}),\theta_i,\theta_{-i})
		\right]\nonumber\\
    && -
        \mathbb E_{\eta(\cdot\mid\theta_i;\zeta^*)}
		\left[
		\omega_{i,\eta(\cdot\mid\theta_i;\zeta^*)}^{\rm al}(\theta_{-i})
		\nabla_{a_i}c_i(a_i,f_{-i}(\theta_{-i}),\theta_i,\theta_{-i})
		\right] \Big\| \nonumber\\
    &\leq& \Big\| \mathbb E_{\eta(\cdot\mid\theta_i;\zeta)}
		\left[
		\omega_{i,\eta(\cdot\mid\theta_i;\zeta)}^{\rm al}(\theta_{-i})
		\nabla_{a_i}c_i(a_i,f_{-i}(\theta_{-i}),\theta_i,\theta_{-i})
		\right]\nonumber\\
    && - \mathbb E_{\eta(\cdot\mid\theta_i;\zeta^*)}
		\left[
		\omega_{i,\eta(\cdot\mid\theta_i;\zeta)}^{\rm al}(\theta_{-i})
		\nabla_{a_i}c_i(a_i,f_{-i}(\theta_{-i}),\theta_i,\theta_{-i})
		\right]\Big\|\nonumber\\
    && +
        \Big\| \mathbb E_{\eta(\cdot\mid\theta_i;\zeta^*)}
		\left[
		\left( \omega_{i,\eta(\cdot\mid\theta_i;\zeta)}^{\rm al}(\theta_{-i}) - \omega_{i,\eta(\cdot\mid\theta_i;\zeta^*)}^{\rm al}(\theta_{-i})
        \right)
		\nabla_{a_i}c_i(a_i,f_{-i}(\theta_{-i}),\theta_i,\theta_{-i})
		\right] \Big\|\nonumber\\
    &\leq & \bar\omega_i^{\rm al}\bar g_i \dd_{TV}(\eta(\cdot\mid\theta_i;\zeta),\eta(\cdot\mid\theta_i;\zeta^*))
		+
		\bar g_i\lambda_i^{\rm al}\dd_{TV}(\eta(\cdot\mid\theta_i;\zeta),\eta(\cdot\mid\theta_i;\zeta^*)).
\end{eqnarray}
In the last inequality, the first term follows from the definition of the total variation metric \eqref{eqn:tv-definition-2} together with Assumption~\ref{assumption:marginal-cost-bounded}, while the second term follows from \eqref{eqn:gradient-regular-tvnorm} in Assumption~\ref{assumption:gradient-regular}. Combining \eqref{eqn:proof-gradient-aleatoric-tv} with Assumption~\ref{assumption:lipschitz_posterior_distribution} yields \eqref{eqn:gradient-al-rm-tv}.

We now turn to the epistemic risk measure $\rho^{\rm ep}$. By an analogous argument,
\begin{eqnarray} \label{eqn:proof-gradient-epistemic-tv}
   && \left\|
\nabla_{a_i}v_{i,\mu_i^t}(a_i,f_{-i},\theta_i)
-
\nabla_{a_i}v_{i,\delta_{\zeta^*}}(a_i,f_{-i},\theta_i)
\right\| \nonumber\\
&= & \left\| \mathbb E_{\mu_i^t}
		\left[
		\omega_{i,\mu_i^t}^{\rm ep}(\zeta)
		\nabla_{a_i}
		\rho^{\rm al}_{\eta(\cdot\mid\theta_i;\zeta)}
		\left[ c_i(a_i,f_{-i}(\theta_{-i}),\theta_i,\theta_{-i}) \right]
		\right] 
        - 
        \nabla_{a_i}
		\rho^{\rm al}_{\eta(\cdot\mid\theta_i;\zeta^*)}
		\left[
		c_i(a_i,f_{-i}(\theta_{-i}),\theta_i,\theta_{-i})
		\right] \right\| \nonumber\\
& = & \left\| \mathbb E_{\mu_i^t}
		\left[
		\omega_{i,\mu_i^t}^{\rm ep}(\zeta)
		  \left(
          \nabla_{a_i}
		\rho^{\rm al}_{\eta(\cdot\mid\theta_i;\zeta)}
		\left[
		c_i(a_i,f_{-i}(\theta_{-i}),\theta_i,\theta_{-i})
		\right]
		 - \nabla_{a_i}
		\rho^{\rm al}_{\eta(\cdot\mid\theta_i;\zeta^*)}
		\left[
		c_i(a_i,f_{-i}(\theta_{-i}),\theta_i,\theta_{-i})
		\right]
        \right)
        \right] \right\| \nonumber\\
&\leq & \bar{\omega}_i^{\rm ep} 
\mathbb E_{\mu_i^t}
		\left[
          \left\| \nabla_{a_i}
		\rho^{\rm al}_{\eta(\cdot\mid\theta_i;\zeta)}
		\left[
		c_i(a_i,f_{-i}(\theta_{-i}),\theta_i,\theta_{-i})
		\right]
		 - \nabla_{a_i}
		\rho^{\rm al}_{\eta(\cdot\mid\theta_i;\zeta^*)}
		\left[
		c_i(a_i,f_{-i}(\theta_{-i}),\theta_i,\theta_{-i})
		\right] \right\|
        \right],
\end{eqnarray}
where the first equality comes from the fact that
\begin{eqnarray*}
    v_{i,\delta_{\zeta^*}}(a_i,f_{-i},\theta_i) = \rho^{\rm al}_{\eta(\cdot\mid\theta_i;\zeta^*)}
		\left[
		c_i(a_i,f_{-i}(\theta_{-i}),\theta_i,\theta_{-i})
		\right],
\end{eqnarray*}
and the second equality and the inequality follow from the gradient regularity of the epistemic risk measure.

To bound the right-hand side of \eqref{eqn:proof-gradient-epistemic-tv}, we split the expectation over $\mathbb{B}(\zeta^*, \epsilon_i)$ and its complement $\mathcal{Z}\backslash \mathbb{B}(\zeta^*, \epsilon_i)$. On $\mathbb{B}(\zeta^*, \epsilon_i)$, the bound \eqref{eqn:gradient-al-rm-tv} applies directly. On $\mathcal{Z}\backslash \mathbb{B}(\zeta^*, \epsilon_i)$, we control each gradient term individually using the uniform gradient bound in Assumption~\ref{assumption:marginal-cost-bounded},
\begin{eqnarray*}
    \left\| \nabla_{a_i}
		\rho^{\rm al}_{\eta(\cdot\mid\theta_i;\zeta)}
		\left[
		c_i(a_i,f_{-i}(\theta_{-i}),\theta_i,\theta_{-i})
		\right]
        \right\|
    & = & \left\| \mathbb E_{\eta(\cdot\mid\theta_i;\zeta)}
		\left[
		\omega_{i,\eta(\cdot\mid\theta_i;\zeta)}^{\rm al}(\theta_{-i})
		\nabla_{a_i}c_i(a_i,f_{-i}(\theta_{-i}),\theta_i,\theta_{-i})
		\right] \right\|\\
    & \leq & \bar{g}_i \left\| \mathbb E_{\eta(\cdot\mid\theta_i;\zeta)}
		\left[
		\omega_{i,\eta(\cdot\mid\theta_i;\zeta)}^{\rm al}(\theta_{-i})\right] \right\| = \bar{g}_i,
\end{eqnarray*}
where the last equality follows from the gradient regularity of the aleatoric risk measure. Consequently,
\begin{eqnarray*}
    \eqref{eqn:proof-gradient-epistemic-tv} &\leq& 
    \bar{\omega}_i^{\rm ep}  
    \left( \mathbb{P}_{\mu_i^t} \left( \|\zeta-\zeta^*\| \leq \epsilon_i \right)
        \mathbb E_{\mu_i^t}
        \left[ 
            L_{\eta,i} (\bar\omega_i^{\rm al}\bar g_i 
		+
		\bar g_i\lambda_i^{\rm al}) 
        \|\zeta - \zeta^*\|^{r_i}
        \right]
        +
        2\bar{g}_i \mathbb{P}_{\mu_i^t} \left( \|\zeta-\zeta^*\| > \epsilon_i \right)
    \right)\nonumber\\
    &\leq& \bar{\omega}_i^{\rm ep}  L_{\eta,i} (\bar\omega_i^{\rm al}\bar g_i 
		+
		\bar g_i\lambda_i^{\rm al}) \mathbb E_{\mu_i^t} \left[  \|\zeta - \zeta^*\|^{r_i} \right] 
        + 2 \bar{\omega}_i^{\rm ep} \bar{g}_i \frac{\mathbb E_{\mu_i^t} \left[  \|\zeta - \zeta^*\|^2 \right]} {\epsilon_i^2}\\
    &\leq & \bar{\omega}_i^{\rm ep}  L_{\eta,i} (\bar\omega_i^{\rm al}\bar g_i 
		+
		\bar g_i\lambda_i^{\rm al}) \mathbb E_{\mu_i^t} \left[  \|\zeta - \zeta^*\|^{2} \right]^{\frac{r_i}{2}} 
        + 2 \frac{\bar{\omega}_i^{\rm ep} \bar{g}_i}{\epsilon_i^2} \mathbb E_{\mu_i^t} \left[  \|\zeta - \zeta^*\|^2 \right],
\end{eqnarray*}
where the second inequality follows from Markov's inequality and the third follows from the H\"older inequality. 
Hence, we obtain \eqref{eqn:gradient-ep-rm-tv} as claimed.
\hfill $\Box$

Equipped with the stability result of Lemma~\ref{lemma:gradient-al-rm-tv}, we are ready to establish the non-asymptotic convergence for the BNE-RABL. Building on the contraction property of the best-response operator $\Psi_{i,\mu^t}^t(f_{-i})(\cdot) := \mathcal A_{i,\mu_i^t}^{t*}(f_{-i},\cdot)$ established in Section~\ref{sec:bne-rabl-banach-theorem}, we obtain the following theorem.

\begin{theorem} \label{thm:convergence-rate-bne-rabl}
Consider the interim BNE-RABL in Definition~\ref{def:BNE-HB}, and suppose that Assumptions~\ref{assump:rabl-common-primitive} and \ref{assump:rabl-contraction} hold. Then the following assertions hold.
\begin{itemize}
    \item[(i)] Let $f^\mu$ and $f^{\mu'}$ be two BNE-RABLs corresponding to the posterior profiles $\mu = (\mu_1,\dots, \mu_n)$ and $\mu' = (\mu_1',\dots, \mu_n')$, respectively. Then 
\begin{eqnarray} \label{eqn:theorem-convergence-rabl}
    \|f^{\mu} - f^{\mu'}\|_\infty \leq \max_{i\in N} \sum_{j\in N} \left[(1-\Gamma)^{-1}\right]_{ij} \bar{d}_j
\end{eqnarray}
with
$
    \bar{d}_i
    := \frac{1}{M_i} \sup_{\theta_i\in \Theta_i}
    \left\|
\nabla_{a_i}v_{i,\mu_i'}(f_i^{\mu}(\theta_i),f_{-i}^\mu,\theta_i)
-
\nabla_{a_i}v_{i,\mu_i}(f_i^{\mu}(\theta_i),f_{-i}^\mu,\theta_i)
\right\|,
$
where $M_i$ is defined as in Assumption~\ref{assump:rabl-common-primitive}.

\item[(ii)] Let $f^{t*}$ denote the BNE-RABL at round $t$, and $f^*$ the BNE under the true underlying distribution $\eta(\cdot;\zeta^*)$. If, in addition, Assumptions~\ref{assumption:conv_rate_post_dist}-\ref{assumption:marginal-cost-bounded} hold,
then for any $\epsilon\in (0,1)$,
    there exist positive constants $c_1>0$, $c_2(\epsilon)>0$, and $T(\epsilon)\geq 2$ such that, for any $t>T(\epsilon)$,
    \begin{eqnarray}
    \label{eq:thm6.1-part(ii)}
        \inmat{Prob}_{\eta(\cdot;\zeta^*)}^{t-1}\left( \|f^{t*} - f^*\|_\infty \leq \Big( \max_{i\in N} \sum_{j\in N} \left[(1-\Gamma)^{-1}\right]_{ij} D_j(\e) \Big) (t-1)^{-\underline{r}c_1/2} \right) \geq 1- \epsilon, \quad
    \end{eqnarray}
    where
    $D_i(\epsilon):=\frac{1}{M_i}
\left(\bar{\omega}_i^{\rm ep}  L_{\eta,i} (\bar\omega_i^{\rm al}\bar g_i 
		+
		\bar g_i\lambda_i^{\rm al}) 
        + 2 \frac{\bar{\omega}_i^{\rm ep} \bar{g}_i}{\epsilon_i^2} 
        \right)\max\left\{ c_2(\epsilon)^{\frac{r_i}{2}},c_2(\epsilon)\right\}$ and $\underline{r} = \min_{i\in N} r_i$.
\end{itemize}
\end{theorem}

\noindent
\textbf{Proof.}
\underline{Part (i).}
Since $f^\mu$ and $f^{\mu'}$ are two BNE-RABLs corresponding to the posterior profiles $\mu$ and $\mu'$ respectively,
then
$f^\mu = \Psi_{\mu}(f^\mu)$ and $f^{\mu'} = \Psi_{\mu'}(f^{\mu'})$. 
For $i\in N$,
\begin{eqnarray*} 
    \|f_i^\mu - f_i^{\mu'}\|_\infty &=& \| \Psi_{i,\mu}(f_{-i}^\mu) - \Psi_{i,\mu'}(f_{-i}^{\mu'}) \|_\infty \nonumber\\
    &\leq& \| \Psi_{i,\mu}(f_{-i}^\mu) - \Psi_{i,\mu'}(f_{-i}^{\mu}) \|_\infty + 
    \| \Psi_{i,\mu'}(f_{-i}^{\mu}) - \Psi_{i,\mu'}(f_{-i}^{\mu'}) \|_\infty \nonumber\\
    &\leq & \| \Psi_{i,\mu}(f_{-i}^\mu) - \Psi_{i,\mu'}(f_{-i}^{\mu}) \|_\infty + 
    \sum_{j\neq i}\Gamma_{ij} \|f_j^\mu - f_j^{\mu'}\|_\infty,
\end{eqnarray*}
where the last inequality comes from Lemma~\ref{lem:rabl-induced-cross-gradient}.
Let 
$ d_i:=\|f_i^\mu-f_i^{\mu'}\|_\infty$ 
and 
$ \tilde d_i:=\| \Psi_{i,\mu}(f_{-i}^\mu) - \Psi_{i,\mu'}(f_{-i}^{\mu}) \|_\infty$ for $i\in N.$
Let
$d=(d_1,\ldots,d_n)^\top$
and 
$\tilde d=(\tilde d_1,\ldots, \tilde d_n)^\top$. 
Since $\Gamma_{ii}=0$ for $i\in N$, we can write the inequalities 
above in the vector form
$
    d \leq \tilde{d} + \Gamma d,
$
and subsequently 
$
    (1-\Gamma)d \leq \tilde{d}.
$
Moreover, since $r(\Gamma)< 1$ under Assumption~\ref{assump:rabl-contraction}(e), $(1-\Gamma)^{-1}=\sum_{k=0}^\infty \Gamma^k$ which is nonnegative componentwise.
Hence
\begin{eqnarray} 
\label{eqn:proof-nonexpansive-0}
    d \leq (1-\Gamma)^{-1}\tilde{d}.
\end{eqnarray}

Next, we bound the term $\tilde{d}_i =\| \Psi_{i,\mu}(f_{-i}^\mu) - \Psi_{i,\mu'}(f_{-i}^{\mu}) \|_\infty$ in \eqref{eqn:proof-nonexpansive-0}. Fix $i\in N$ and $\theta_i\in \Theta_i$, and set $a_{i,\mu}:= \Psi_{i,\mu}(f^\mu)(\theta_i)$ and $a_{i,\mu'}:= \Psi_{i,\mu'}(f_{\mu})(\theta_i)$. If $a_{i,\mu}=a_{i,\mu'}$, then $\|a_{i,\mu}-a_{i,\mu'}\| =0$, and the bound below \eqref{eqn:proof-nonexpansive-6.5} holds trivially. 
Consider the case that $a_{i,\mu}\neq a_{i,\mu'}$. Since $a_{i,\mu}$ minimizes $v_{i,\mu_i}(a_i,f_{-i}^\mu,\theta_i)$ over $\mathcal A_i$, the first-order optimality condition gives
\begin{eqnarray} \label{eqn:proof-nonexpansive-1}
        \nabla_{a_i}v_{i,\mu_i}(a_{i,\mu},f_{-i}^\mu,\theta_i)^\top
    (a_{i,\mu'}-a_{i,\mu})\geq0 .
\end{eqnarray}
Likewise, since $a_{i,\mu'}$ minimizes $v_{i,\mu_i'}(a_i,f_{-i}^\mu,\theta_i)$,
\begin{eqnarray} \label{eqn:proof-nonexpansive-2}
    \nabla_{a_i}v_{i,\mu_i'}(a_{i,\mu'}, f_{-i}^\mu,\theta_i)^\top
    (a_{i,\mu}-a_{i,\mu'})\geq 0 .
\end{eqnarray}
Adding \eqref{eqn:proof-nonexpansive-1} and \eqref{eqn:proof-nonexpansive-2} gives rise to 
\begin{eqnarray} \label{eqn:proof-nonexpansive-3}
\left(
\nabla_{a_i}v_{i,\mu_i} (a_{i,\mu},f_{-i}^\mu,\theta_i)
-
\nabla_{a_i}v_{i,\mu_i'}(a_{i,\mu'},f_{-i}^\mu,\theta_i)
\right)^\top
(a_{i,\mu}-a_{i,\mu'}) \leq 0.
\end{eqnarray}
Moreover, since $v_{i,\mu_i'}(a_i,f_{-i}^\mu,\theta_i)$ is strongly convex in $a_i$, its gradient is strongly monotonic, whence
\begin{eqnarray} \label{eqn:proof-nonexpansive-4}
&&
\left(
\nabla_{a_i}v_{i,\mu_i'}(a_{i,\mu},f_{-i}^\mu,\theta_i)
-
\nabla_{a_i}v_{i,\mu_i'}(a_{i,\mu'},f_{-i}^\mu,\theta_i)
\right)^\top
(a_{i,\mu}-a_{i,\mu'})  \geq
M_i\|a_{i,\mu}-a_{i,\mu'}\|^2.\qquad
\end{eqnarray}
Combining these results, we obtain
\begin{eqnarray*} 
&& \left\|
\nabla_{a_i}v_{i,\mu_i'}(a_{i,\mu},f_{-i}^\mu,\theta_i)
-
\nabla_{a_i}v_{i,\mu_i}(a_{i,\mu},f_{-i}^\mu,\theta_i)
\right\|
\|a_{i,\mu}-a_{i,\mu'}\| \nonumber\\
&\geq&
\left(
\nabla_{a_i}v_{i,\mu_i'}(a_{i,\mu},f_{-i}^\mu,\theta_i)
-
\nabla_{a_i}v_{i,\mu_i}(a_{i,\mu},f_{-i}^\mu,\theta_i)
\right)^\top
(a_{i,\mu}-a_{i,\mu'}) \nonumber\\
&=&
\left(
\nabla_{a_i}v_{i,\mu_i'}(a_{i,\mu},f_{-i}^\mu,\theta_i)
-
\nabla_{a_i}v_{i,\mu_i'} (a_{i,\mu'},f_{-i}^\mu,\theta_i)
\right)^\top
(a_{i,\mu}-a_{i,\mu'}) \nonumber\\
&&+
\left(
\nabla_{a_i}v_{i,\mu_i'} (a_{i,\mu'},f_{-i}^\mu,\theta_i)
-
\nabla_{a_i}v_{i,\mu_i}(a_{i,\mu},f_{-i}^\mu,\theta_i)
\right)^\top
(a_{i,\mu}-a_{i,\mu'})\nonumber\\
&\geq & M_i\|a_{i,\mu} - a_{i,\mu'}\|^2,
\end{eqnarray*}
where the first inequality follows from the Cauchy--Schwarz inequality and the last inequality follows from \eqref{eqn:proof-nonexpansive-3} and \eqref{eqn:proof-nonexpansive-4}. Dividing both sides by $M_i \|a_{i,\mu}-a_{i,\mu'}\|$ and recalling that $a_{i,\mu}= \Psi_{i,\mu}(f^\mu)(\theta_i)=f_i^{\mu}(\theta_i)$, we obtain
\begin{eqnarray}
    \|a_{i,\mu} - a_{i,\mu'}\|
    &\leq&
    \frac{1}{M_i}
    \left\|
\nabla_{a_i}v_{i,\mu_i'}(a_{i,\mu},f_{-i}^\mu,\theta_i)
-
\nabla_{a_i}v_{i,\mu_i}(a_{i,\mu},f_{-i}^\mu,\theta_i)
\right\|\nonumber\\
& =  & \frac{1}{M_i}
    \left\|
\nabla_{a_i}v_{i,\mu_i'}(f_i^{\mu}(\theta_i),f_{-i}^\mu,\theta_i)
-
\nabla_{a_i}v_{i,\mu_i}(f_i^{\mu}(\theta_i),f_{-i}^\mu,\theta_i)
\right\|
\label{eqn:proof-nonexpansive-6.5}
\end{eqnarray}
and thus
\begin{eqnarray}
    \tilde{d}_i&=& \| \Psi_{i,\mu}(f_{-i}^\mu) - \Psi_{i,\mu'}(f_{-i}^{\mu}) \|_\infty \nonumber\\
    &\leq& \frac{1}{M_i} \sup_{\theta_i\in \Theta_i}
    \left\|
\nabla_{a_i}v_{i,\mu_i'}(f_i^{\mu}(\theta_i),f_{-i}^\mu,\theta_i)
-
\nabla_{a_i}v_{i,\mu_i}(f_i^{\mu}(\theta_i),f_{-i}^\mu,\theta_i)
\right\|. \label{eqn:proof-nonexpansive-6}
\end{eqnarray}
Substituting \eqref{eqn:proof-nonexpansive-6} into \eqref{eqn:proof-nonexpansive-0},
we arrive at \eqref{eqn:theorem-convergence-rabl}.

\underline{Part (ii).}
By setting $\mu=\mu^t$ and $\mu_i'=\delta_{\zeta^*}$ for each $i\in N$ in Part~(i), 
we have 
\begin{eqnarray*} 
    \|f^{t*} - f^{*}\|_\infty \leq \max_{i\in N} \sum_{j\in N} \left[(1-\Gamma)^{-1}\right]_{ij} \bar{d}_j,
\end{eqnarray*}
where
\begin{eqnarray}
    \bar{d}_i
    &:=& \frac{1}{M_i} \sup_{\theta_i\in \Theta_i}
    \left\|
\nabla_{a_i}v_{i,\delta_{\zeta^*}}(f_i^{\mu}(\theta_i),f_{-i}^\mu,\theta_i)
-
\nabla_{a_i}v_{i,\mu_i^t}(f_i^{\mu}(\theta_i),f_{-i}^\mu,\theta_i)
\right\|\nonumber\\
&\leq & \frac{1}{M_i}
\left(\bar{\omega}_i^{\rm ep}  L_{\eta,i} (\bar\omega_i^{\rm al}\bar g_i 
		+
		\bar g_i\lambda_i^{\rm al}) 
        \bbe_{\mu_i^t} \left[  \|\zeta - \zeta^*\|^2 \right] ^{\frac{r_i}{2}} 
        + 2 \frac{\bar{\omega}_i^{\rm ep} \bar{g}_i}{\epsilon_i^2} 
        \bbe_{\mu_i^t} \left[  \|\zeta - \zeta^*\|^2 \right]\right),
        \label{eq:d_ibar???}
\end{eqnarray}
and the inequality follows from Lemma~\ref{lemma:gradient-al-rm-tv}.
Finally, under Assumption~\ref{assumption:conv_rate_post_dist}, 
we can use \eqref{eq:Assump-6.1} to
obtain an upper bound for the rhs of equation \eqref{eq:d_ibar???}
with probability at least $1-\epsilon$:
\begin{eqnarray*}
    \bar{d}_i &\leq& \frac{1}{M_i}
\left(\bar{\omega}_i^{\rm ep}  L_{\eta,i} (\bar\omega_i^{\rm al}\bar g_i 
		+
		\bar g_i\lambda_i^{\rm al}) 
        c_2(\epsilon)^{\frac{r_i}{2}} (t-1)^{-c_1\frac{r_i}{2}} 
        + 2 \frac{\bar{\omega}_i^{\rm ep} \bar{g}_i}{\epsilon_i^2} 
        c_2(\epsilon) (t-1)^{-c_1}\right)\\
        &\leq& \frac{1}{M_i}
\left(\bar{\omega}_i^{\rm ep}  L_{\eta,i} (\bar\omega_i^{\rm al}\bar g_i 
		+
		\bar g_i\lambda_i^{\rm al}) 
        + 2 \frac{\bar{\omega}_i^{\rm ep} \bar{g}_i}{\epsilon_i^2} 
        \right) \max\left\{ c_2(\epsilon)^{\frac{r_i}{2}},c_2(\epsilon)\right\} (t-1)^{-c_1\frac{r_i}{2}}, 
\end{eqnarray*}
which yields \eqref{eq:thm6.1-part(ii)}.
\hfill $\Box$

Theorem~\ref{thm:convergence-rate-bne-rabl}(i) shows that the convergence of the BNE-RABL is governed by the convergence of the gradients of the players' objective functions, 
while Theorem~\ref{thm:convergence-rate-bne-rabl}(ii) makes this rate explicit as $t^{-\underline{r} c_1/2}$ with $\underline{r}:=\min_{i\in N} r_i$. 
In particular, when $c_1=1$ as established in~\cite{mou2024diffusion} and $r_i=1$ for all $i\in N$, the convergence rate reduces to $t^{-1/2}$, which coincides with the non-asymptotic convergence rate of the posterior distribution derived in~\cite{mou2024diffusion}.

\subsection{Non-asymptotic convergence of BNE-BPD}

We now turn to the non-asymptotic convergence of BNE-BPD in Definition~\ref{def:BNE-BPD}.
To this end, we first derive the convergence rate of the Bayesian predictive distribution under the following local stability condition for the marginal distribution $\eta_{i}(\cdot;\zeta)$.

\begin{assumption} \label{assumption:6.5}
For each $i\in N$, $p_i(\theta_i;\zeta^*)>0$ for all $\theta_i\in \Theta_i$.
     There exist positive constants $L_{p,i}$ and $\delta$ such that 
     \begin{eqnarray}\label{eqn:assumption-6.5-1}
         \left| \log p_i(\theta_i;\zeta) - \log p_i(\theta_i;\zeta^*) \right| \leq L_{p,i}\|\zeta-\zeta^*\|,\qquad \forall \zeta\in \mathbb{B}(\zeta^*,\delta), 
     \end{eqnarray}
     uniformly for all $\theta_i\in \Theta_i$.
     Moreover, there exists $\bar{l}< +\infty$ such that
     \begin{eqnarray}\label{eqn:assumption-6.5-2}
         \frac{p_i(\theta_i;\zeta)}{p_i(\theta_i;\zeta^*)} < \bar{l},\qquad \forall \zeta\in \mathcal{Z}, \theta_i\in \Theta_i.
     \end{eqnarray}
\end{assumption}

Assumption~\ref{assumption:6.5} imposes Lipschitz continuity of the log marginal density function with respect to the parameter $\zeta$ near the true parameter value $\zeta^*$. 
This condition is widely used in analyzing the convergence and stability of the posterior distribution; see, e.g., Theorem~5.39 in \cite{van2000asymptotic} and Section~3 of \cite{honorio2011lipschitz}.
In particular, it has been verified for several parametric distributions under appropriate conditions in Section~4 of \cite{honorio2011lipschitz}.
Condition \eqref{eqn:assumption-6.5-2} imposes a uniform one-sided domination condition on the marginal density over the family of parametric distributions, i.e., $p_i(\theta_i;\zeta) < \bar{l} p_i(\theta_i;\zeta^*)$ for all $\theta_i\in \Theta_{i}$.  
This condition is satisfied by many families of parametric distributions with a common compact support and densities that are jointly continuous in $\theta_{i}$ and $\zeta$.

\begin{lemma}
\label{lemma:6}
    Under Assumptions~\ref{assumption:conv_rate_post_dist}, \ref{assumption:lipschitz_posterior_distribution}, and \ref{assumption:6.5}, 
    the following assertions hold.
    \begin{itemize}
        \item[(i)] For $\epsilon\in (0,1)$, there exists a positive constant $T(\epsilon)\geq 2$ such that, for all $t>T(\epsilon)$,
        \begin{eqnarray} \label{eqn:convergence-bpd}
            \inmat{Prob}_{\eta(\cdot;\zeta^*)}^{t-1 }\left( \dd_{TV}( \nu_i^t, \eta(\cdot;\zeta^*)  ) \leq c_2^{\rm bp}(\epsilon) (t-1)^{- \min\{ \frac{1}{2}, \frac{r_i}{2} \}c_1}  \right) \geq 1-\epsilon,
        \end{eqnarray}
    \end{itemize}
    where
        $c_2^{\rm bp}(\epsilon):=\left(\exp(L_{p,i}\delta) L_{p_i} + L_{\eta,i} + \frac{2}{\delta^2} \right)  \max\{c_2(\epsilon)^{1/2}, c_2(\epsilon)^{r_i/2}, c_2(\epsilon)\}$ and  $c_2(\epsilon)$ is defined in Assumption~\ref{assumption:conv_rate_post_dist}.

    \item[(ii)] For $\epsilon\in (0,1)$, there exists a positive constant $T(\epsilon,\delta)\geq 2$ such that, for all $t>\max\{T(\epsilon), \hat{T}(\delta,\epsilon)\}$,
    \begin{eqnarray} \label{eqn:convergence-cond-bpd}
        \inmat{Prob}_{\eta(\cdot;\zeta^*)}^{t-1 }\left( 
        \dd_{TV} \left( \nu_i^t(\cdot|\theta_i), \eta(\cdot|\theta_i;\zeta^*) \right)
        \leq c_2^{\rm cbp}(\epsilon) (t-1)^{ -\frac{c_1 r_i}{2}}  \right) \geq 1-\epsilon,
    \end{eqnarray}
    where 
    $
        c_2^{\rm cbp}(\epsilon): = 2\exp(L_{p,i}\delta) 
    \left( L_{\eta,i} \exp(L_{p,i}\delta)+ 2 \frac{\bar{l}}{\delta^2}\right) c_2(\epsilon).  
    $
\end{lemma}

\noindent
\textbf{Proof.}
\underline{Part (i).}
Let $l_i(\theta_i,\zeta):= \frac{p_i(\theta_i;\zeta)}{p_i(\theta_i;\zeta^*)}$.
For $\zeta\in \mathbb{B}(\zeta^*,\delta)$,
inequality \eqref{eqn:assumption-6.5-1} implies 
\begin{eqnarray} \label{eqn:joint-BPD-proof}
    \left| \log l_i(\theta_i,\zeta) \right| 
\leq L_{p,i} \|\zeta-\zeta^*\| \leq L_{p,i}\delta,
\end{eqnarray}
and by the inequality $|e^x - 1| \leq e^{|x|} |x|,\; \forall x\in \R$, we have
\begin{eqnarray} \label{eqn:joint-BPD-proof-0}
    \left| l_i(\theta_i,\zeta) - 1  \right| 
    \leq  e^{ \left|\log l_i(\theta_i,\zeta)\right| } \left|\log l_i(\theta_i,\zeta)\right|\leq  
    \exp(L_{p,i}\delta) L_{p,i} \|\zeta-\zeta^*\|.
\end{eqnarray}
On the other hand, by the definition of total variation metric \eqref{eqn:tv-definition-2},
\begin{eqnarray} \label{eqn:joint-BPD-proof-1}
    &&\dd_{TV} (\eta(\cdot;\zeta), \eta(\cdot;\zeta^*) ) =\sup_{h\in \mathcal{M}, \|h\|_\infty\leq 1} \left| \int_{\Theta}h(\theta) (p(\theta;\zeta) - p(\theta;\zeta^*)) d\theta\right|\nonumber\\
    & &= \int_{\Theta} |p(\theta;\zeta) - p(\theta;\zeta^*)| d\theta= \int_{\Theta} |p_i(\theta_i;\zeta)p(\theta_{-i}|\theta_i;\zeta) - p_i(\theta_i;\zeta^*)p(\theta_{-i}|\theta_i;\zeta^*)| d\theta\nonumber\\
    &&\leq  \int_{\Theta} | (p_i(\theta_i;\zeta) - p_i(\theta_i;\zeta^*) )p(\theta_{-i}|\theta_i;\zeta)| d\theta
    + 
    \int_{\Theta} |p_i(\theta_i;\zeta^*) (p(\theta_{-i}|\theta_i;\zeta) - p(\theta_{-i}|\theta_i;\zeta^*))| d\theta\nonumber\\
    && \leq \int_{\Theta_i} | p_i(\theta_i;\zeta) - p_i(\theta_i;\zeta^*) | d\theta_i 
    +
    \int_{\Theta_i} p_i(\theta_i;\zeta^*) \dd_{TV}(\eta(\cdot|\theta_i;\zeta), \eta(\cdot|\theta_i;\zeta^*))d\theta_{i},  
\end{eqnarray}
where the last inequality is obtained by the Fubini-Tonelli theorem and integrating each term over $\Theta_{-i}$.
By the definition of $l_i$ and  \eqref{eqn:joint-BPD-proof-0}, 
\begin{eqnarray} \label{eqn:joint-BPD-proof-2}
    \int_{\Theta_i} | p_i(\theta_i;\zeta) - p_i(\theta_i;\zeta^*) | d\theta_i 
    &=& \int_{\Theta_i} p_i(\theta_i;\zeta^*) \left| l_i(\theta_i,\zeta) - 1  \right| d\theta_i\nonumber\\
    &\leq &\exp(L_{p,i}\delta) L_{p_i} \|\zeta-\zeta^*\|.
\end{eqnarray}
Moreover, by \eqref{eqn:lipschitz_posterior_distribution} in Assumption~\ref{assumption:lipschitz_posterior_distribution},
\begin{eqnarray} \label{eqn:joint-BPD-proof-3}
    \int_{\Theta_i} p_i(\theta_i;\zeta^*) \dd_{TV}(\eta(\cdot|\theta_i;\zeta), \eta(\cdot|\theta_i;\zeta^*))d\theta_{i} 
    \leq L_{\eta,i} \|\zeta -\zeta^*\|^{r_i}.
\end{eqnarray}
Substituting \eqref{eqn:joint-BPD-proof-2} and \eqref{eqn:joint-BPD-proof-3} into \eqref{eqn:joint-BPD-proof-1}, we arrive at
\begin{eqnarray} \label{eqn:joint-BPD-proof-4}
    \dd_{TV} (\eta(\cdot;\zeta), \eta(\cdot;\zeta^*) ) \leq \exp(L_{p,i}\delta) L_{p_i} \|\zeta-\zeta^*\| + L_{\eta,i} \|\zeta -\zeta^*\|^{r_i},
\end{eqnarray}
for all $\zeta\in \mathbb{B}(\zeta^*,\delta)$.

Finally, by the convexity of TV distance, we have 
\begin{eqnarray*}
    &&\dd_{TV}( v_i^t, \eta(\cdot;\zeta^*) ) =  \dd_{TV}\left( \int_\mathcal{Z}\eta(\cdot;\zeta)d\mu_i^t(\zeta), \eta(\cdot;\zeta^*) \right)  
    \leq  \int_\mathcal{Z} \dd_{TV}( \eta(\cdot;\zeta), \eta(\cdot;\zeta^*) ) d\mu_i^t(\zeta)\\
    && = \int_{\mathbb{B}(\zeta^*,\delta)} \dd_{TV}( \eta(\cdot;\zeta), \eta(\cdot;\zeta^*) ) d\mu_i^t(\zeta) + \int_{\mathcal{Z}\backslash \mathbb{B}(\zeta^*,\delta)} \dd_{TV}( \eta(\cdot;\zeta), \eta(\cdot;\zeta^*) ) d\mu_i^t(\zeta) \\
    && \leq \int_\mathcal{Z} \exp(L_{p,i}\delta) L_{p_i} \|\zeta-\zeta^*\| + L_{\eta,i} \|\zeta -\zeta^*\|^{r_i} d\mu_i^t + 2 \mu_i^t (\|\zeta - \zeta^*\|> \delta)  \\
    && \leq \exp(L_{p,i}\delta) L_{p_i} \mathbb{E}_{\mu_i^t} \left[ \|\zeta-\zeta^*\|^{2} \right]^{1/2} + L_{\eta,i}\mathbb{E}_{\mu_i^t} \left[ \|\zeta-\zeta^*\|^{2} \right]^{r_i/2} +  2 \frac{\mathbb{E}_{\mu_i^t} \left[ \|\zeta-\zeta^*\|^{2} \right]}{\delta^2},
\end{eqnarray*}
where the second inequality follows by \eqref{eqn:joint-BPD-proof-4} and the fact that total variation metric is less than or equal to $2$, 
and the third inequality follows by the H\"older inequality and Markov inequality.
By 
Assumption~\ref{assumption:conv_rate_post_dist},
we arrive at \eqref{eqn:convergence-bpd}.

\underline{Part (ii).}
Fix $\theta_i \in \Theta_{i}$.
The conditional density function of BPD is written as 
\begin{eqnarray} \label{eqn:cond-BPD-proof-1}
    q_i^t(\theta_{-i}\mid \theta_i)
    &:=&
    \frac{q_i^t(\theta_i,\theta_{-i})}
    {\int_{\Theta_{-i}}q_i^t(\theta_i,\theta_{-i})d\theta_{-i}}
    =
    \frac{\int_{\mathcal{Z}}p(\theta_i,\theta_{-i};\zeta)m_i^t(\zeta)\,d\zeta}
    {\int_{\Theta_{-i}}\int_{\mathcal{Z}}
        p(\theta_i,\theta_{-i};\zeta)m_i^t(\zeta)\,d\zeta\,d\theta_{-i}
    } \nonumber\\
    & = & \frac{\int_{\mathcal{Z}}p(\theta_{-i}|\theta_{i};\zeta) p_i(\theta_i;\zeta) m_i^t(\zeta)\,d\zeta}
    {\int_{\Theta_{-i}}\int_{\mathcal{Z}}p(\theta_{-i}|\theta_{i};\zeta) p_i(\theta_i;\zeta) m_i^t(\zeta)\,d\zeta d\theta_{-i}} 
    = \frac{\int_{\mathcal{Z}}p(\theta_{-i}|\theta_{i};\zeta) p_i(\theta_i;\zeta) m_i^t(\zeta)\,d\zeta}
    {\int_{\mathcal{Z}} p_i(\theta_i;\zeta) m_i^t(\zeta)\,d\zeta }\nonumber\\
    &=& \frac{\int_{\mathcal{Z}}p(\theta_{-i}|\theta_{i};\zeta) l_i(\theta_i,\zeta) m_i^t(\zeta)\,d\zeta}
    {\int_{\mathcal{Z}} l_i(\theta_i,\zeta) m_i^t(\zeta)\,d\zeta },
\end{eqnarray}
where the third equation comes from Fubini-Tonelli theorem.
By \eqref{eqn:joint-BPD-proof}, 
\begin{eqnarray} \label{eqn:cond-BPD-proof-2}
    \int_{\mathcal{Z}} l_i(\theta_i,\zeta) m_i^t(\zeta)\,d\zeta 
    &\geq&
    \int_{\mathbb{B}(\zeta^*,\delta)} l_i(\theta_i,\zeta) m_i^t(\zeta)\,d\zeta
    \geq
    \exp(-L_{p,i}\delta)\mu_i^t(\|\zeta-\zeta^*\|\leq \delta)\nonumber\\
    &\geq& \exp(-L_{p,i}\delta) \left( 1- \frac{\mathbb{E}_{\mu_i^t}[\|\zeta-\zeta^*\|^2]}{\delta^2} \right),
\end{eqnarray}
where the last inequality comes from the Markov inequality.
For $t$ sufficiently large, i.e., for $t>\hat{T}(\delta,\epsilon)$ with some $\hat{T}(\delta,\epsilon)>2$ such that $\mathbb{E}_{\mu_i^t}[\|\zeta-\zeta^*\|^2]<\frac{\delta^2}{2}$ holds with probability $1-\epsilon$, it follows that
\begin{eqnarray} \label{eqn:cond-BPD-proof-3}
    &&\dd_{TV} \left( \nu_i^t(\cdot|\theta_i), \eta(\cdot|\theta_i;\zeta^*) \right) \nonumber\\
    && = \int_{\Theta_{-i}} \left| \frac{\int_{\mathcal{Z}}p(\theta_{-i}|\theta_{i};\zeta) l_i(\theta_i,\zeta) m_i^t(\zeta)\,d\zeta}
    {\int_{\mathcal{Z}} l_i(\theta_i,\zeta) m_i^t(\zeta)\,d\zeta } - p(\theta_{-i}|\theta_i;\zeta^*) \right| d\theta_{-i}\nonumber\\
    &&\leq \frac{\exp(L_{p,i}\delta)\delta^2}{\delta^2 - \mathbb{E}_{\mu_i^t}[\|\zeta-\zeta^*\|^2]} 
    \int_{\Theta_{-i}} \left| \int_{\mathcal{Z}} \big( p(\theta_{-i}|\theta_{i};\zeta) - p(\theta_{-i}|\theta_i;\zeta^*) \big) l_i(\theta_i,\zeta) m_i^t(\zeta)\,d\zeta  \right| d\theta_{-i}\nonumber\\
    && \leq 2\exp(L_{p,i}\delta) 
     \int_{\mathcal{Z}}  \dd_{TV} \Big( \eta(\cdot|\theta_i;\zeta), \eta(\cdot|\theta_i;\zeta^*) \Big) l_i(\theta_i,\zeta) m_i^t(\zeta)\,d\zeta,
\end{eqnarray}
where the first inequality follows from \eqref{eqn:cond-BPD-proof-2}.
Splitting the integral into two parts, one over the ball $\mathbb{B}(\zeta^*,\delta)$ 
and the other over its complement  $\mathcal{Z}\backslash \mathbb{B}(\zeta^*,\delta)$,
we have 
\begin{eqnarray} \label{eqn:cond-BPD-proof-4}
    &&\int_{\mathcal{Z}}  \dd_{TV} \Big( \eta(\cdot|\theta_i;\zeta), \eta(\cdot|\theta_i;\zeta^*) \Big) l_i(\theta_i,\zeta) m_i^t(\zeta)\,d\zeta\nonumber\\
    &= & \int_{\mathbb{B}(\zeta^*,\delta)}  \dd_{TV} \Big( \eta(\cdot|\theta_i;\zeta), \eta(\cdot|\theta_i;\zeta^*) \Big) l_i(\theta_i,\zeta) m_i^t(\zeta)\,d\zeta\nonumber\\
    &&
    + \int_{\mathcal{Z}\backslash \mathbb{B}(\zeta^*,\delta)}  \dd_{TV} \Big( \eta(\cdot|\theta_i;\zeta), \eta(\cdot|\theta_i;\zeta^*) \Big) l_i(\theta_i,\zeta) m_i^t(\zeta)\,d\zeta\nonumber\\
    &\leq & \int_{\mathbb{B}(\zeta^*,\delta)}  L_{\eta,i}\|\zeta - \zeta^*\|^{r_i} \exp(L_{p,i}\delta) m_i^t(\zeta)\,d\zeta
    + \int_{\mathcal{Z}\backslash \mathbb{B}(\zeta^*,\delta)}  2 \bar{l} m_i^t(\zeta)\,d\zeta\nonumber\\
    & \leq & L_{\eta,i} \exp(L_{p,i}\delta) \mathbb{E}_{\mu_i^t} \left[ \|\zeta - \zeta^*\|^{r_i} \right] + 2 \bar{l} \mu_i^t (\|\zeta - \zeta^*\|> \delta),
\end{eqnarray}
where, in the first inequality, the first term is due to Assumption~\ref{assumption:lipschitz_posterior_distribution} and \eqref{eqn:joint-BPD-proof}, 
and the second term is due to \eqref{eqn:assumption-6.5-2} in Assumption~\ref{assumption:6.5} and the fact that total variation metric is less than or equal to $2$.
Substituting \eqref{eqn:cond-BPD-proof-4} into \eqref{eqn:cond-BPD-proof-3}, and using the H\"older inequality and Markov inequality, 
we arrive at
\begin{eqnarray*}
    &&\dd_{TV} \left( \nu_i^t(\cdot|\theta_i), \eta(\cdot|\theta_i;\zeta^*) \right) \nonumber\\
    && \leq 2\exp(L_{p,i}\delta) 
    \left( L_{\eta,i} \exp(L_{p,i}\delta) \mathbb{E}_{\mu_i^t} \left[ \|\zeta - \zeta^*\|^{2} \right]^{r_i/2} + 2 \frac{\bar{l}}{\delta^2} \mathbb{E}_{\mu_i^t} \left[ \|\zeta - \zeta^*\|^{2} \right] \right).
\end{eqnarray*}
By Assumption~\ref{assumption:conv_rate_post_dist},
we arrive at \eqref{eqn:convergence-cond-bpd}.
\hfill $\Box$

Up to this point, we have obtained the non-asymptotic convergence rates of the BPD and conditional BPD, given the convergence rate of posterior distribution specified in Assumption~\ref{assumption:conv_rate_post_dist}.
This result enables us to establish the convergence rate of BNE-BPD by arguments parallel to those in Section~\ref{sec:convergence-bne-rabl}.
To this end, we require $\rho^{\rm bp}$ to satisfy the regularity condition in Definition~\ref{def:gradient-regular}.
Note that the BPD coincides with the true underlying distribution when the posterior distribution reduces to a Dirac measure at the true parameter, i.e., $\mu = \delta_{\zeta^*}.$
Hence, in the following discussion we write $\eta(\cdot|\theta_i;\zeta^*)$ for the BPD induced by $\mu = \delta_{\zeta^*}$.

\begin{assumption} \label{assumption:gradient-regular-bp}
    For each player $i\in N$, 
    assume that $\rho^{\rm bp}$ is a gradient regular risk measure with 
    $\omega_{i,\nu_i^t(\cdot\mid\theta_i)}^{\rm bp} (\theta_{-i})\in \left[0,  \bar{\omega}_{i}^{\rm bp}\right]$ for all $\theta_{-i}$. 
    Moreover, there exists a constant $\lambda_i^{\rm bp}$ such that 
    \begin{eqnarray} \label{eqn:gradient-regular-tvnorm-bp}
        \mathbb{E}_{\eta(\cdot\mid\theta_i;\zeta^*)} \left[ \left| \omega_{i,\nu_i^t(\cdot|\theta_i)}^{\rm bp}(\theta_{-i}) - \omega_{i,\eta(\cdot|\theta_i;\zeta^*)}^{\rm bp} (\theta_{-i}) \right| \right] 
        \leq 
        \lambda_i^{\rm bp}\, \dd_{TV} \left( \nu_i^t(\cdot|\theta_i), \eta(\cdot|\theta_i;\zeta^*) \right).
    \end{eqnarray}
\end{assumption}

\begin{theorem} \label{thm:convergence-rate-bne-bpd}
Consider the interim BNE-BPD in Definition~\ref{def:BNE-BPD}. 
Under Assumptions~\ref{assumption:bpd-existence}, \ref{assumption:bpd-uniqueness},  \ref{assumption:marginal-cost-bounded}, and \ref{assumption:gradient-regular-bp}, the following assertions hold.
\begin{itemize}
    \item[(i)] For each $i\in N$, $\theta_i\in \Theta_i$, 
\begin{eqnarray*}
    &&\left \|\nabla_{a_i} v_{i,\nu_i^t(\cdot|\theta_i)}^{\text{bp},t}(a_i,f_{-i},\theta_i) - \nabla_{a_i} v_{i,\eta(\cdot|\theta_i;\zeta^*)}^{\text{bp},t}(a_i,f_{-i},\theta_i)
    \right\|\\
    &&\qquad \leq  \left( \bar\omega_i^{\rm bp}\bar g_i  + \bar g_i\lambda_i^{\rm bp}\right)
    \dd_{TV}(\nu_i^t(\cdot\mid\theta_i),\eta(\cdot\mid\theta_i;\zeta^*)).
\end{eqnarray*}

    \item[(ii)] Let $f^\nu$ and $f^{\nu'}$ be two BNE-BPDs corresponding to the posterior profiles $\mu = (\mu_1,\dots, \mu_n)$ and $\mu' = (\mu_1',\dots, \mu_n')$, respectively. Then 
\begin{eqnarray} \label{eqn:theorem-convergence-bpd}
    \|f^{\nu} - f^{\nu'}\|_\infty \leq \max_{i\in N} \sum_{j\in N} \left[(1-\Gamma^{\rm bp})^{-1}\right]_{ij} \bar{d}_j^{\rm bp}
\end{eqnarray}
with
$
    \bar{d}_i^{\rm bp}
    := \frac{1}{M_i} \sup_{\theta_i\in \Theta_i}
    \left\|
\nabla_{a_i} v_{i,\nu_i(\cdot|\theta_i)}^{\text{bp}}(f_i^\nu(\theta_i),f_{-i}^\nu,\theta_i)
-
\nabla_{a_i} v_{i,\nu_i'(\cdot|\theta_i)}^{\text{bp}}(f_i^\nu(\theta_i),f_{-i}^\nu,\theta_i)
\right\|,
$
where $M_i$ is defined as in Assumption~\ref{assump:rabl-common-primitive}.

\item[(iii)] Let $f^{t*}$ denote the BNE-BPD at round $t$, and $f^*$ the BNE under the true underlying distribution $\eta(\cdot;\zeta^*)$. 
If, in addition, Assumptions~\ref{assumption:conv_rate_post_dist}, \ref{assumption:lipschitz_posterior_distribution}, and \ref{assumption:6.5} hold,
then for any $\epsilon\in (0,1)$,
    there exist positive constants $c_1>0$, $c_2(\epsilon)>0$, and $T(\epsilon), \hat{T}(\delta,\epsilon)\geq 2$ such that, for all $t>\max\{T(\epsilon), \hat{T}(\delta,\epsilon)\}$,
    \begin{eqnarray}
        \inmat{Prob}_{\eta(\cdot;\zeta^*)}^{t-1}\left( \|f^{t*} - f^*\|_\infty \leq \Big( \max_{i\in N} \sum_{j\in N} \left[(1-\Gamma^{\rm bp})^{-1}\right]_{ij} D_j^{\rm bp}(\e,\delta) \Big) (t-1)^{-\underline{r}c_1/2} \right) \geq 1- \epsilon, \quad
    \end{eqnarray}
    where
    $D_i^{\rm bp}(\e,\delta):=\frac{2}{M_i}
\left(\bar{\omega}_i^{\rm bp} + \lambda_i^{\rm bp}
        \right) \bar{g}_i \exp(L_{p,i} \delta) 
        \left( L_{\eta,i} \exp(L_{p,i} \delta) + 2 \frac{\bar l}{\delta^2} \right)  
        c_2(\epsilon)$
        and $\underline{r} = \min_{i\in N} r_i$.
\end{itemize}
\end{theorem}

\noindent
\textbf{Proof.}
The proofs follow arguments analogous to those used in Lemma~\ref{lemma:gradient-al-rm-tv} and Theorem~\ref{thm:convergence-rate-bne-rabl}.
We skip the details.
\hfill $\Box$

So far, we have established the non-asymptotic convergence of 
the BNE-RABL in Theorem~\ref{thm:convergence-rate-bne-rabl} and the BNE-BPD in Theorem~\ref{thm:convergence-rate-bne-bpd}.
Both convergence rates are of order $(t-1)^{-\underline{r}c_1/2}$. 
Based on these results, we can quantify the discrepancy between the BNE-RABL and BNE-BPD strategies at any given round $t$; see the following corollary.

\begin{corollary}
\label{corollary:bne-rabl-and-bne-bpd}
Let $f^{t,\rm{RABL}}$ be a BNE-RABL at round $t$, and $f^{t,\rm{BPD}}$ be a BNE-BPD at round $t$.
    Suppose that the settings and conditions of Theorems~\ref{thm:convergence-rate-bne-rabl} and \ref{thm:convergence-rate-bne-bpd} hold.
    Then, for any $\epsilon>0$, there exist positive constants $c_1$, $c_2(\epsilon)$, and $T(\epsilon), \hat{T}(\delta,\epsilon)\geq 2$ such that, for all $t>\max\{T(\epsilon), \hat{T}(\delta,\epsilon)\}$,
    \begin{eqnarray*}
        \inmat{Prob}_{\eta(\cdot;\zeta^*)}^{t-1}\left( \|f^{t,\rm{RABL}} - f^{t,\rm{BPD}}\|_\infty \leq C(\delta,\epsilon) (t-1)^{-\underline{r}c_1/2} \right) \geq 1- \epsilon,
    \end{eqnarray*}
    where 
$
        C(\delta,\epsilon) =  
        \max_{i\in N} \sum_{j\in N} \left[(1-\Gamma)^{-1}\right]_{ij} D_j(\e)
        +
        \max_{i\in N} \sum_{j\in N} \left[(1-\Gamma^{\rm bp})^{-1}\right]_{ij} D_j^{\rm bp}(\e,\delta),
$
    and $D_j(\e)$ and $D_j^{\rm bp}(\e,\delta)$ are defined in Theorems~\ref{thm:convergence-rate-bne-rabl} and \ref{thm:convergence-rate-bne-bpd} respectively.
\end{corollary}

In particular, 
this corollary
shows that the strategies obtained from the two models can be regarded as effective approximations to each other.
Moreover, the discrepancy vanishes as $t$ grows, and therefore both strategies converge to a common limit, which is the oracle BNE under the true underlying distribution.

\section{Application in the price competition}

\label{sec:app_price_competition}

In this section, we apply 
the established 
Bayesian learning framework 
to
price competition.
In practice, a firm usually has more accurate information about its own current cost than its competitors, which motivates modeling a firm's current cost as private information \citep{ui2016bayesian,liu2025bayesian}. 
These costs may be statistically dependent because they are affected by common input prices and supply chain conditions, and hence it is reasonable to consider a joint distribution of players' costs \citep{lagerlof2024bertrand}.
These firms make pricing decisions over multiple selling periods,
and public disclosures may provide information about 
firms' previously realized production costs \citep{bagnoli2010oligopoly}.
For tractability, we 
therefore assume 
that the realized cost profiles are publicly observed at the end of  each period.
We further consider the situation where 
each firm minimizes the risk associated with its loss in the current period. 
The setting is standard 
in Bayesian learning-in-games models 
to isolate the learning process from 
investment, reputation, collusion, and other 
intertemporal incentives 
\citep{wu2025convergence}.
The stage-wise model is consistent with evidence of the short-horizon competition pressure documented for pricing and operating decisions
where
managers may focus on current performance targets, even at the expense of long-run value \citep{graham2005economic}.

\subsection{The setup}
\label{subsec:price-competition-application}

Consider an oligopolistic market with $n$ differentiated substitutable products,
where each product is offered by a distinct firm.
Firm $i$ chooses a price
$p_i\in\mathcal A_i:=[\underline p_i,\bar p_i],$
and its marginal cost type is
$\theta_i\in\Theta_i:=[\underline\theta_i,\bar\theta_i].$
We assume that the type profile $\theta:=(\theta_1,\ldots,\theta_n)$ is drawn from an unknown joint distribution $\eta(\cdot;\zeta^*).$
Firm $i$’s marginal cost is privately known: it is observed by firm $i$ itself but remains unknown to its competitors at the time of decision-making.
Once this stage of the game is completed, the  cost profile becomes common knowledge to all firms.
The firms learn the joint distribution of marginal costs through the repeated Bayesian learning process specified in Model Specification~\ref{model:repeated-game}.

We use an affine inverse demand function, which is standard in oligopoly models with differentiated products; see \cite{farahat2011comparison,ui2016bayesian}.
For a price vector $p:=(p_1,\ldots,p_n),$
the demand of firm $i$ is
\begin{eqnarray}
\label{eqn:linear-price-demand}
    D_i(p)
    &=&
    A_i-b_i p_i+\sum_{j\neq i}d_{ij}p_j,
\end{eqnarray}
where $b_i$ measures the own-price sensitivity of firm $i$'s demand, 
and $d_{ij}$ measures the substitution effect from firm $j$'s price to firm $i$'s demand.
The parameters $b_i$ and $d_{ij}$ are known to all players.
We assume that the action sets are chosen so that
$D_i(p)\geq0,\forall p\in\mathcal A.$
The firm $i$'s profit is then $(p_i-\theta_i)D_i(p).$
Since our equilibrium model is formulated to minimize the player's loss, 
we define firm $i$'s loss as the negative profit:
\begin{eqnarray}
\label{eqn:linear-price-loss}
    c_i(p_i,p_{-i},\theta_i,\theta_{-i})
    :=
    -(p_i-\theta_i)D_i(p_i,p_{-i}) = -(p_i-\theta_i) \left( A_i-b_i p_i+\sum_{j\neq i}d_{ij}p_j \right).
\end{eqnarray}
This loss function depends on rivals' types through their strategy profile $f_{-i}(\theta_{-i})$ via the inverse demand function.
At round $t$, firm $i$ has a posterior belief $\mu_i^t$ over the unknown parameter $\zeta$.

\subsection{Verification of key assumptions.}
\label{sec:veri-price-competition}

We first verify the conditions on the loss function in Assumption~\ref{basic-assumption}(b),  Assumption~\ref{assump:rabl-common-primitive}(a) and Assumption~\ref{assump:rabl-contraction}(b).

\begin{lemma}
\label{lem:price-competition-constants}
Consider the price competition model \eqref{eqn:linear-price-demand}--\eqref{eqn:linear-price-loss}.
Then the following assertions hold. 
\begin{itemize}
    \item[(i)] The loss function $c_i$ is bounded and uniformly continuous on $\mathcal A\times\Theta$.

    \item[(ii)] For each $i\in N$, $c_i(p_i,p_{-i},\theta_i,\theta_{-i})$ is $M_i$-strongly convex in $p_i$ with $M_i = 2b_i$.

    \item[(iii)] For each $j\neq i$, the blockwise Lipschitz constants in Assumption~\ref{assump:rabl-contraction}(b) can be chosen as $L_{ij}^{c,1} = d_{ij}$, and $L_{ij}^{c,0} = d_{ij} \sup_{p_i \in \mathcal A_i,\theta_i\in\Theta_i} |p_i-\theta_i| \leq d_{ij}(\bar{p}_i - \underline{\theta}_i).$
\end{itemize}

\end{lemma}

\noindent
\textbf{Proof.}
\underline{Part (i).} Since $\mathcal A$ and $\Theta$ are compact and $c_i$ is polynomial in $(p_i,p_{-i},\theta_i)$, $c_i$ is bounded and uniformly continuous on
$\mathcal A\times\Theta$.

\underline{Part (ii).} By \eqref{eqn:linear-price-loss}, we obtain
\begin{eqnarray*}
    c_i(p_i,p_{-i},\theta_i,\theta_{-i}) =
    b_i p_i^2 - \left( A_i+\sum_{j\neq i}d_{ij}p_j+b_i\theta_i \right) p_i
    +
    \theta_i \left( A_i+\sum_{j\neq i}d_{ij}p_j \right),
\end{eqnarray*}
which is strongly convex in $p_i$ with $M_i = 2b_i$.

\underline{Part (iii).} Next, 
\begin{eqnarray*}
    \nabla_{p_i}c_i(p_i,p_{-i},\theta_i,\theta_{-i})
    &=& 2b_i p_i - A_i - \sum_{j\neq i}d_{ij}p_j - b_i\theta_i .
\end{eqnarray*}
and thus, for $p_{-i}\neq p_{-i}'$,
\begin{eqnarray*}
\left\| \nabla_{p_i}c_i(p_i,p_{-i},\theta_i,\theta_{-i}) - \nabla_{p_i}c_i(p_i,p_{-i}',\theta_i,\theta_{-i}) \right\| \leq \sum_{j\neq i} d_{ij}|p_j-p_j'|.
\end{eqnarray*}
Thus $L_{ij}^{c,1}=d_{ij}$.
Similarly
\begin{eqnarray*}
\left| c_i(p_i,p_{-i},\theta_i,\theta_{-i}) - c_i(p_i,p_{-i}',\theta_i,\theta_{-i}) \right| 
= d_{ij}|p_i-\theta_i|\,|p_j-p_j'|.
\end{eqnarray*}
Since $\sup_{p_i \in \mathcal A_i,\theta_i\in\Theta_i} |p_i-\theta_i| \leq \bar{p}_i - \underline{\theta}_i$, we arrive at the result.
\hfill $\Box$

We now plug these constants into the contraction matrix to verify the assumptions.
To ease the notation, let 
$ G_i:= \sup_{p\in\mathcal A,\theta_i\in\Theta_i} \left\| \nabla_{p_i}c_i(p_i,p_{-i},\theta_i,\theta_{-i}) \right\|.$
This quantity is finite since $\mathcal A$ and $\Theta_i$ are compact.

\begin{proposition}
\label{cor:price-competition-uniqueness}
Consider the price competition model \eqref{eqn:linear-price-demand}--\eqref{eqn:linear-price-loss}.
Under Assumptions~\ref{assump:rabl-common-primitive}--\ref{assump:rabl-contraction}, the following assertions hold.

\begin{enumerate}[(i)]
\item If both $\rho^{\rm al}$ and $\rho^{\rm ep}$ are expectations, then $\kappa_i^{{\rm al},1}=\kappa_i^{{\rm ep},1}=1,$ and $\kappa_i^{{\rm al},0}=\kappa_i^{{\rm ep},0}=0$, and thus 
$ \Gamma_{ij}= \frac{d_{ij}}{2b_i}, i\neq j.$
In this case, a sufficient condition for uniqueness of the BNE-RABL is 
$$ \max_{i\in N} \frac{\sum_{j\neq i}d_{ij}}{2b_i} <1.$$

\item If $\rho^{\rm ep}$ is the expectation and $\rho^{\rm al}$ is the entropic risk measure defined by 
$$ \rho_{\gamma_i}^{\rm ent}(X) := \frac{1}{\gamma_i}\log\mathbb E[\exp(\gamma_i X)], $$
with $\gamma_i>0$,
then
$ 
\kappa_i^{{\rm ep},1}= 1,
\kappa_i^{{\rm ep},0}= 0,
\kappa_i^{{\rm al},1}= 1,
\kappa_i^{{\rm al},0}= \gamma_iG_i,
$
and
$\Gamma_{ij}=\frac{d_{ij}\left(1+\gamma_iG_i (\bar{p}_i - \underline{\theta}_i)\right)}{2b_i}, i\neq j.$
A sufficient condition for uniqueness is
\[
\max_{i\in N}  \frac{ \left( 1+ \gamma_iG_i(\bar p_i-\underline\theta_i) \right) \sum_{j\neq i}d_{ij}} {2b_i} <1.
\]

\item If $\rho^{\rm ep}$ is the expectation and
$\rho^{\rm al}={\rm CVaR}_{\beta_i}$ with confidence level
$\beta_i\in(0,1)$, and
for every $(p_i, f_{-i}, \theta_i)$ and conditional distribution of $\theta_{-i}$, the density function of the loss distribution induced by
$ c_i\left(p_i, f_{-i}(\theta_{-i}), \theta_i, \theta_{-i}\right)$
is uniformly bounded by $\bar h_i/2$ in a neighborhood of $\beta_i$-quantile,
then
$
    \kappa_i^{{\rm ep},1}=1,
    \kappa_i^{{\rm ep},0}=0,
    \kappa_i^{{\rm al},1}=1,
    \kappa_i^{{\rm al},0} = \frac{2G_i\bar h_i}{1-\beta_i},
$
and
$
    \Gamma_{ij} =
    \frac{ d_{ij} \left( 1+ \frac{2G_i\bar h_i(\bar p_i-\underline\theta_i)} {1-\beta_i} \right)} {2b_i},i\neq j.
$
A sufficient condition for uniqueness is
\[
\max_{i\in N} \frac{ \left( 1+ \frac{2G_i\bar h_i(\bar p_i-\underline\theta_i)} {1-\beta_i} \right) \sum_{j\neq i}d_{ij}} {2b_i} <1.
\]

\end{enumerate}
\end{proposition}

\noindent
\textbf{Proof.}
By Lemma~\ref{lem:price-competition-constants}, we have
$M_i=2b_i, L_{ij}^{c,1}=d_{ij}, L_{ij}^{c,0}=d_{ij}(\bar{p}_i - \underline{\theta}_i).$
Substituting these constants into the definition of $\Gamma_{ij}$ in Assumption~\ref{assump:rabl-contraction}(e), we have 
\begin{eqnarray*}
    \Gamma_{ij}=\frac{K_{ij}}{M_i}
    = \frac{ d_{ij} \left[ \kappa_i^{{\rm ep},1} \left( \kappa_i^{{\rm al},1} + \kappa_i^{{\rm al},0}(\bar{p}_i - \underline{\theta}_i) \right) + \kappa_i^{{\rm ep},0}(\bar{p}_i - \underline{\theta}_i) \right]} {2b_i}.
\end{eqnarray*}
Since $\Gamma$ is nonnegative, 
$\max_{i\in N}\sum_{j\neq i}\Gamma_{ij}<1$ implies $r(\Gamma)<1$ by Perron-Frobenius theorem.
Therefore, by substituting the corresponding risk-measure constants into this formula, we immediately arrive at the formula for $\Gamma_{ij}$ and the sufficient condition for uniqueness. 

Next, we derive the risk-measure constants 
$\kappa_i^{{\rm ep},1},
\kappa_i^{{\rm ep},0},
\kappa_i^{{\rm al},1},
\kappa_i^{{\rm al},0}$ for $\rho^{\rm ep}$ and $\rho^{\rm al}$ in these cases.
To make the implications transparent, consider the symmetric case where $b_i=b, d_{ij}=d,\bar{p}_i = \bar{p},$ and $\underline{\theta}_i = \underline{\theta}$.

\underline{Part (i).}
When both risk measures are expectations, the gradient of the risk value is the
expectation of the gradient.
Indeed, for any relevant probability measure,
\begin{eqnarray*}
&& \left\| \nabla_{p_i} \mathbb E_{\eta(\cdot\mid\theta_i;\zeta)} \left[ c_i\left( p_i, f_{-i}(\theta_{-i}), \theta_i, \theta_{-i} \right) \right]
-
\nabla_{p_i} \mathbb E_{\eta(\cdot\mid\theta_i;\zeta)} \left[ c_i\left( p_i, f_{-i}'(\theta_{-i}), \theta_i, \theta_{-i} \right) \right] \right\|\\
&=&
\left\| \mathbb E_{\eta(\cdot\mid\theta_i;\zeta)} \left[ \nabla_{p_i} c_i\left( p_i, f_{-i}(\theta_{-i}), \theta_i, \theta_{-i} \right)
-
\nabla_{p_i} c_i\left(  p_i, f_{-i}'(\theta_{-i}), \theta_i, \theta_{-i} \right) \right] \right\|\\
&\leq&
\sup_{\theta_{-i}\in\Theta_{-i}}
\left\| \nabla_{p_i} c_i\left( p_i, f_{-i}(\theta_{-i}), \theta_i,  \theta_{-i} \right) -
\nabla_{p_i} c_i\left( p_i, f_{-i}'(\theta_{-i}),\theta_i, \theta_{-i}\right) \right\|.
\end{eqnarray*}
Thus, the expectation does not amplify the change in the marginal loss and does not
require the loss-level term.
Therefore, we immediately have
$
    \kappa_i^{{\rm al},1} = \kappa_i^{{\rm ep},1} = 1,
    \kappa_i^{{\rm al},0} = \kappa_i^{{\rm ep},0} = 0.
$

\underline{Part (ii).}
When $\rho^{\rm ep}$ is the expectation and $\rho^{\rm al}$ is the entropic risk measure,
we have $\kappa_i^{{\rm ep},1}=1, \kappa_i^{{\rm ep},0}=0$ for expectation as in  part (i).
It remains to derive the constants for the entropic risk measure.
By the definition of the entropic risk measure,
\begin{eqnarray}
&&
\nabla_{p_i}
\rho_{\gamma_i,{\eta(\cdot\mid\theta_i;\zeta)}}^{\rm ent}
\left[c_i\left( p_i, f_{-i}(\theta_{-i}), \theta_i, \theta_{-i}\right)\right] \nonumber \\
&=&
\nabla_{p_i} \left( \frac{1}{\gamma_i}  \log  \mathbb{E}_{\eta(\cdot\mid\theta_i;\zeta)}\left[\exp\left(\gamma_i c_i\left( p_i, f_{-i}(\theta_{-i}), \theta_i, \theta_{-i} \right) \right)\right]  \right)
\nonumber \\
&=&
\frac{
\mathbb E_{\eta(\cdot\mid\theta_i;\zeta)}\left[\exp\left(\gamma_i c_i\left( p_i, f_{-i}(\theta_{-i}), \theta_i, \theta_{-i} \right) \right) \nabla_{p_i} c_i\left( p_i, f_{-i}(\theta_{-i}), \theta_i, \theta_{-i}\right)\right]}
{ \mathbb E_{\eta(\cdot\mid\theta_i;\zeta)} \left[ \exp\left( \gamma_i c_i\left( p_i, f_{-i}(\theta_{-i}), \theta_i, \theta_{-i} \right)\right)\right]}. \label{eqn:proof-kappa-riskmeasure}
\end{eqnarray}
To ease the notation, we denote $P := \eta(\cdot\mid\theta_i;\zeta)$. For fixed $p_i,\theta_i$ and $f_{-i}$, let 
\begin{eqnarray*}
    dQ_{f_{-i}}
    :=
    \frac{ \exp\left( \gamma_i c_i\left( p_i, f_{-i}(\theta_{-i}), \theta_i, \theta_{-i} \right) \right)} 
    {\mathbb E_P\left[\exp\left(\gamma_i c_i\left( p_i, f_{-i}(\theta_{-i}), \theta_i, \theta_{-i} \right) \right)\right]} dP.
\end{eqnarray*}
Since 
\begin{eqnarray*}
    \int_{\Theta_{-i}} dQ_{f_{-i}} =
    \int_{\Theta_{-i}}\frac{ \exp\left( \gamma_i c_i\left( p_i, f_{-i}(\theta_{-i}), \theta_i, \theta_{-i} \right) \right)} 
    {\mathbb E_P\left[\exp\left(\gamma_i c_i\left( p_i, f_{-i}(\theta_{-i}), \theta_i, \theta_{-i} \right) \right)\right]} dP =1,
\end{eqnarray*}
$Q_{f_{-i}}$ is also a probability distribution of $\theta_{-i}$.
By \eqref{eqn:proof-kappa-riskmeasure} and the definition of $Q_{f_{-i}}$, 
we have 
\begin{eqnarray*}
    \nabla_{p_i}
\rho_{\gamma_i,P}^{\rm ent}
\left[c_i\left( p_i, f_{-i}(\theta_{-i}), \theta_i, \theta_{-i}\right)\right] 
= \mathbb{E}_{Q_{f_{-i}}} \left[ \nabla_{p_i} c_i\left( p_i, f_{-i}(\theta_{-i}), \theta_i, \theta_{-i}\right)\right]
\end{eqnarray*}
and similarly 
\begin{eqnarray*}
    \nabla_{p_i}
\rho_{\gamma_i,P}^{\rm ent}
\left[c_i\left( p_i, f_{-i}'(\theta_{-i}), \theta_i, \theta_{-i}\right)\right] 
= \mathbb{E}_{Q_{f_{-i}'}} \left[ \nabla_{p_i} c_i\left( p_i, f_{-i}'(\theta_{-i}), \theta_i, \theta_{-i}\right)\right].
\end{eqnarray*}
Thus, we have 
\begin{eqnarray*}
&&
\left\|
\nabla_{p_i} \rho_{\gamma_i, P}^{\rm ent} \left[ c_i\left( p_i, f_{-i}(\theta_{-i}), \theta_i, \theta_{-i} \right) \right]
-
\nabla_{p_i} \rho_{\gamma_i, P}^{\rm ent} \left[ c_i\left( p_i,  f_{-i}'(\theta_{-i}), \theta_i, \theta_{-i} \right) \right]
\right\| \\
&= &
\left\|
\mathbb{E}_{Q_{f_{-i}}} \left[ \nabla_{p_i} c_i\left( p_i, f_{-i}(\theta_{-i}), \theta_i, \theta_{-i}\right)\right]
-
\mathbb{E}_{Q_{f_{-i}}} \left[ \nabla_{p_i} c_i\left( p_i, f_{-i}'(\theta_{-i}), \theta_i, \theta_{-i}\right)\right]
\right\| \\
& & +
\left\|
\mathbb{E}_{Q_{f_{-i}}} \left[ \nabla_{p_i} c_i\left( p_i, f_{-i}'(\theta_{-i}), \theta_i, \theta_{-i}\right)\right]
-
\mathbb{E}_{Q_{f_{-i}'}} \left[ \nabla_{p_i} c_i\left( p_i, f_{-i}'(\theta_{-i}), \theta_i, \theta_{-i}\right)\right]
\right\| \\
&\leq&
\sup_{\theta_{-i}\in\Theta_{-i}} \left\|\nabla_{p_i} c_i\left( p_i, f_{-i}(\theta_{-i}), \theta_i, \theta_{-i}\right)
-
\nabla_{p_i} c_i\left(p_i, f_{-i}'(\theta_{-i}),\theta_i, \theta_{-i}\right)\right\|\\
&& +2G_i\dd_{\rm TV}\left(Q_{f_{-i}},Q_{f_{-i}'}\right),
\end{eqnarray*}
where the last inequality comes from \eqref{eqn:tv-definition-2} and $G_i:= \sup_{p\in\mathcal A,\theta_i\in\Theta_i} \left\| \nabla_{p_i}c_i(p_i,p_{-i},\theta_i,\theta_{-i}) \right\|.$
By Lemma~\ref{lem:tv-bound-exponential-tilting}, 
$$\dd_{\rm TV}(Q_X,Q_{X'}) \leq\frac{\gamma_i}{2}\sup_{\theta_{-i}\in\Theta_{-i}} \left| c_i\left( a_i, f_{-i}(\theta_{-i}), \theta_i, \theta_{-i} \right)
-
c_i\left( a_i, f_{-i}'(\theta_{-i}), \theta_i, \theta_{-i} \right) \right|.$$
Therefore, we have $\kappa_i^{{\rm al},1}=1$ and $\kappa_i^{{\rm al},0}=\gamma_iG_i .$

\underline{Part (iii).}
Since $\rho^{\rm ep}$ is expectation,
we immediately have
$\kappa_i^{{\rm ep},1}=1,\kappa_i^{{\rm ep},0}=0.$
Then we derive constants for $\rho^{\rm al}={\rm CVaR}_{\beta_i}$.
Fix a relevant probability measure and two rivals' strategy profiles
$f_{-i}$ and $f_{-i}'$.
Let
\[
\Delta
=
\sup_{\theta_{-i}\in\Theta_{-i}}
\left|
c_i\left(p_i,f_{-i}(\theta_{-i}),\theta_i,\theta_{-i}\right)
-
c_i\left(p_i,f_{-i}'(\theta_{-i}),\theta_i,\theta_{-i}\right)
\right|.
\]
By the non-expansive property of monetary risk measures \cite[Lemma 4.3]{follmer2011stochastic},
\[
\left|
{\rm VaR}_{\beta_i}
\left(c_i\left(p_i,f_{-i}(\theta_{-i}),\theta_i,\theta_{-i}\right)\right)
-
{\rm VaR}_{\beta_i}
\left(c_i\left(p_i,f_{-i}'(\theta_{-i}),\theta_i,\theta_{-i}\right)\right)
\right|
\leq
\Delta .
\]
Next, for every $\theta_{-i}\in\Theta_{-i}$, we have
\begin{eqnarray}\label{eqn:proof-diff-of-indicator}
&&
\left|
\mathds{1}_{\left\{
c_i\left(p_i,f_{-i}(\theta_{-i}),\theta_i,\theta_{-i}\right)
\geq
{\rm VaR}_{\beta_i}
\left(c_i\left(p_i,f_{-i}(\theta_{-i}),\theta_i,\theta_{-i}\right)\right)
\right\}}
-
\mathds{1}_{\left\{
c_i\left(p_i,f_{-i}'(\theta_{-i}),\theta_i,\theta_{-i}\right)
\geq
{\rm VaR}_{\beta_i}
\left(c_i\left(p_i,f_{-i}'(\theta_{-i}),\theta_i,\theta_{-i}\right)\right)
\right\}}
\right|
\nonumber\\
&\leq&
\mathds{1}_{\left\{
\left|
c_i\left(p_i,f_{-i}(\theta_{-i}),\theta_i,\theta_{-i}\right)
-
{\rm VaR}_{\beta_i}
\left(c_i\left(p_i,f_{-i}(\theta_{-i}),\theta_i,\theta_{-i}\right)\right)
\right|
\leq
2\Delta
\right\}} .
\end{eqnarray}
Indeed, if
\[
c_i\left(p_i,f_{-i}(\theta_{-i}),\theta_i,\theta_{-i}\right)
-
{\rm VaR}_{\beta_i}
\left(c_i\left(p_i,f_{-i}(\theta_{-i}),\theta_i,\theta_{-i}\right)\right)
>
2\Delta,
\]
then 
\[
c_i\left(p_i,f_{-i}'(\theta_{-i}),\theta_i,\theta_{-i}\right) > {\rm VaR}_{\beta_i}
\left(c_i\left(p_i,f_{-i}'(\theta_{-i}),\theta_i,\theta_{-i}\right)\right).
\]
Similarly, if
\[
c_i\left(p_i,f_{-i}(\theta_{-i}),\theta_i,\theta_{-i}\right)
-
{\rm VaR}_{\beta_i}
\left(c_i\left(p_i,f_{-i}(\theta_{-i}),\theta_i,\theta_{-i}\right)\right)
< -2\Delta,
\]
then 
\[
c_i\left(p_i,f_{-i}'(\theta_{-i}),\theta_i,\theta_{-i}\right) < {\rm VaR}_{\beta_i}
\left(c_i\left(p_i,f_{-i}'(\theta_{-i}),\theta_i,\theta_{-i}\right)\right).
\]
Thus the two indicators in \eqref{eqn:proof-diff-of-indicator} can differ only when $c_i\left(p_i,f_{-i}(\theta_{-i}),\theta_i,\theta_{-i}\right)$ lies within a
$2\Delta$-neighborhood of ${\rm VaR}_{\beta_i}
\left(c_i\left(p_i,f_{-i}(\theta_{-i}),\theta_i,\theta_{-i}\right)\right)$.
Moreover, since the density function induced by
$ c_i\left(p_i, f_{-i}(\theta_{-i}), \theta_i, \theta_{-i}\right)$
is uniformly bounded by $\bar h_i/2$ in a neighborhood of $\beta_i$-quantile, we have
\[
P\left(
\left|
c_i\left(p_i,f_{-i}(\theta_{-i}),\theta_i,\theta_{-i}\right)
-
{\rm VaR}_{\beta_i}
\left(c_i\left(p_i,f_{-i}(\theta_{-i}),\theta_i,\theta_{-i}\right)\right)
\right|
\leq
2\Delta
\right)
\leq
2\bar h_i\Delta .
\]

On the other hand, by the definition of CVaR,
\begin{eqnarray*}
&&
\nabla_{p_i}
{\rm CVaR}_{\beta_i}
\left[
c_i\left(p_i,f_{-i}(\theta_{-i}),\theta_i,\theta_{-i}\right)
\right]\\
&=&
\frac{1}{1-\beta_i}
\mathbb E
\left[
\nabla_{p_i}c_i\left(p_i,f_{-i}(\theta_{-i}),\theta_i,\theta_{-i}\right)
\mathds{1}_{\left\{
c_i\left(p_i,f_{-i}(\theta_{-i}),\theta_i,\theta_{-i}\right)
\geq
{\rm VaR}_{\beta_i}
\left(c_i\left(p_i,f_{-i}(\theta_{-i}),\theta_i,\theta_{-i}\right)\right)
\right\}}
\right].
\end{eqnarray*}
The same formula holds when $f_{-i}$ is replaced by $f_{-i}'$.
Therefore,
\begin{eqnarray*}
&&
\left\|
\nabla_{p_i}
{\rm CVaR}_{\beta_i}
\left[
c_i\left(p_i,f_{-i}(\theta_{-i}),\theta_i,\theta_{-i}\right)
\right]
-
\nabla_{p_i}
{\rm CVaR}_{\beta_i}
\left[
c_i\left(p_i,f_{-i}'(\theta_{-i}),\theta_i,\theta_{-i}\right)
\right]
\right\|\\
&\leq&
\frac{1}{1-\beta_i}
\mathbb E
\left\|
\nabla_{p_i}c_i\left(p_i,f_{-i}(\theta_{-i}),\theta_i,\theta_{-i}\right)
-
\nabla_{p_i}c_i\left(p_i,f_{-i}'(\theta_{-i}),\theta_i,\theta_{-i}\right)
\right\|\\
&&\qquad \qquad
\mathds{1}_{\left\{
c_i\left(p_i,f_{-i}(\theta_{-i}),\theta_i,\theta_{-i}\right)
\geq
{\rm VaR}_{\beta_i}
\left(c_i\left(p_i,f_{-i}(\theta_{-i}),\theta_i,\theta_{-i}\right)\right)
\right\}}
\Bigg]\\
&&\quad + \frac{1}{1-\beta_i}
\mathbb E
\Bigg[
\left\|
\nabla_{p_i}c_i\left(p_i,f_{-i}(\theta_{-i}),\theta_i,\theta_{-i}\right)
\right\| 
\left|
\mathds{1}_{\left\{
c_i\left(p_i,f_{-i}(\theta_{-i}),\theta_i,\theta_{-i}\right)
\geq
{\rm VaR}_{\beta_i}
\left(c_i\left(p_i,f_{-i}(\theta_{-i}),\theta_i,\theta_{-i}\right)\right)
\right\}}\right.\\
&&\quad \quad \left.-
\mathds{1}_{\left\{
c_i\left(p_i,f_{-i}'(\theta_{-i}),\theta_i,\theta_{-i}\right)
\geq
{\rm VaR}_{\beta_i}
\left(c_i\left(p_i,f_{-i}'(\theta_{-i}),\theta_i,\theta_{-i}\right)\right)
\right\}}
\right|
\Bigg]\\
&\leq&
\mathbb E
\Bigg[
\left\|
\nabla_{p_i}c_i\left(p_i,f_{-i}(\theta_{-i}),\theta_i,\theta_{-i}\right)
-
\nabla_{p_i}c_i\left(p_i,f_{-i}'(\theta_{-i}),\theta_i,\theta_{-i}\right)
\right\|
\Bigg]\\
&&\quad+
\frac{G_i}{1-\beta_i}
P\left(
\left|
c_i\left(p_i,f_{-i}(\theta_{-i}),\theta_i,\theta_{-i}\right)
-
{\rm VaR}_{\beta_i}
\left(c_i\left(p_i,f_{-i}(\theta_{-i}),\theta_i,\theta_{-i}\right)\right)
\right|
\leq
2\Delta
\right)\\
&\leq&
\sup_{\theta_{-i}\in \Theta_{-i}}
\left\|
\nabla_{p_i}c_i\left(p_i,f_{-i}(\theta_{-i}),\theta_i,\theta_{-i}\right)
-
\nabla_{p_i}c_i\left(p_i,f_{-i}'(\theta_{-i}),\theta_i,\theta_{-i}\right)
\right\| + \frac{2G_i\bar h_i}{1-\beta_i}\Delta.
\end{eqnarray*}
Therefore, we may take
$
    \kappa_i^{{\rm al},1}=1,
    \kappa_i^{{\rm al},0}=\frac{2G_i\bar h_i}{1-\beta_i}.
$
\hfill $\Box$

Proposition~\ref{cor:price-competition-uniqueness} indicates 
how the uniqueness of the risk-averse BNE depends on the inter-player effect and player's own risk attitude.
In the risk-neutral case, the coefficient is given by
$ \Gamma_{ij}=\frac{d_{ij}}{2b_i},$
which is exactly the result in \cite{su2025existence}.
A larger $b_i$ indicates a stronger own-price effect on %
the residual 
demand for firm $i$,
and 
a larger $d_{ij}$ indicates a stronger cross-price effect of firm $j$'s price on the residual demand.
Thus, the condition
$
    \max_{i\in N}\frac{\sum_{j\neq i}d_{ij}}{2b_i}<1
$
requires that  the aggregate cross-price effects from rivals is dominated by twice of the firm's own-price effect. 
Under the risk measures, the corresponding amplification factors increase with risk aversion, which makes the sufficient condition for uniqueness more restrictive.
Under the entropic risk measure, the cross-price effect is multiplied by
$1+\gamma_iG_i(\bar p_i-\underline\theta_i).$
The parameter $\gamma_i$ measures firm $i$'s degree of entropic risk aversion,
$G_i$ bounds the sensitivity of firm $i$'s marginal loss with respect to its own price, and $\bar p_i-\underline\theta_i$ bounds the largest possible price-cost {\color{blue}} gap. 
Therefore, stronger risk aversion, steeper marginal losses, or a larger feasible price-cost range all amplify the strategic effect of rivals' prices. 
As $\gamma_i$ decreases to zero, the amplification factor approaches one, and the risk-neutral condition is recovered.
By contrast, as $\gamma_i\rightarrow +\infty$ the entropic risk measure approaches the essential supremum, and $\Gamma_{ij}$ diverges.
In this case, the uniqueness is no longer guaranteed by theory.
Similarly,
under CVaR, the cross-price effect is multiplied by
$1+\frac{2G_i\bar h_i(\bar p_i-\underline\theta_i)}{1-\beta_i}.$
This factor has a similar interpretation, and it is driven by tail-risk sensitivity. 
A higher confidence level $\beta_i$ places more weight on extreme losses and increases the amplification factor. 
The constant $\bar h_i$ measures how concentrated the loss distribution is around its VaR level. 
Hence, a high confidence level or a high local density around the VaR level makes uniqueness harder to establish.
Consequently, equilibrium uniqueness is more likely in markets where the firm's own price has a stronger effect on its residual demand, products are less substitutable, and the players are less risk averse.

\subsection{Numerical tests with a linear inverse demand function}
\label{sec:numerical_test}

We perform numerical tests on a two-firm price-competition game to examine non-asymptotic  convergence of BNE-RABL and BNE-BPD 
as the number of games increases,
and to compare the CPU times.
The experiments are carried out on a MacBook Air (Apple M4 chip, 16 GB RAM) using GAMS and the extended mathematical programming framework developed in \cite{kim2019solving}, 
which is a widely used software for solving
optimization and game-theoretic problems.

Consider the setting described in Section~\ref{subsec:price-competition-application}.
We assume that
each firm \(i\in\{1,2\}\) observes its own marginal cost (type) \(\theta_i\in\Theta_i:=[0.5,1.5]\) and then sets a price (action) \(p_i\in\mathcal A:=[1,3]\).
For \(i\in\{1,2\}\),
the loss function in \eqref{eqn:linear-price-loss} is given by
\begin{eqnarray}
    c_i(p_i,p_{-i},\theta_i,\theta_{-i})
:=
-(p_i-\theta_{-i})\left(3.8-1.5p_i+p_{-i}\right).
\label{eqn:cost-linear}
\end{eqnarray}
We discretize firm \(i\)'s type space using \(K\) equally spaced grid points with
$
\theta_i^k
=
0.5+\frac{k-1}{K-1}$ for $
k=1,\ldots,K,
$
and restrict its response function \(f_i\) to the set
 \(\mathcal F_{iK}\)
of piecewise-linear functions
with breakpoints\footnote{We use terminologies ``breakpoints'' and ``grid points'' interchangeably depending on the context.
The former is used in piecewise-linear function and the latter is used in discretization of distributions.
} 
$\theta_i^k, k= 1,\dots, K$.
Let $p_{ik}= f_i(\theta_i^k)$.
Then player $i$ chooses optimal 
price $p_{ik}$ when its marginal price is $\theta_i^t$.



We assume that the joint distribution of 
$\theta = (\theta_1,\theta_2)$,
denoted by \(\eta(\cdot;\zeta)\), 
is a bivariate Gaussian distribution with identical marginal mean value $\zeta$ and standard deviation \(\sigma=0.18\),
and correlation coefficient \(\rho=0.5\), truncated to \(\Theta_1\times\Theta_2\).
In this setup, the joint distribution of $\theta$ is parameterized by 
$\zeta$. The true value of $\zeta$ is 
\(\zeta^*=1\).
The two firms observe the same sequence of joint type profiles at the end of games,
use heterogeneous truncated Gaussian priors for $\zeta$,
\[
\mu_1^1
=
\mathcal N(0.65,0.15^2)\big|_{[0.5,1.5]},
\qquad
\mu_2^1
=
\mathcal N(1.35,0.15^2)\big|_{[0.5,1.5]}.
\]
After observing \(T\) profiles, firm \(i\)'s posterior is updated with
$
\mu_i^{T+1}(d\zeta)
\propto
\mu_i^1(d\zeta)
\prod_{s=1}^{T}p(\theta^s;\zeta).
$

We first consider the BNE-RABL model.
Suppose that firm $i$ evaluates epistemic uncertainty using the expectation and aleatoric uncertainty using the entropic risk measure (\(\gamma=0.1\)):
\begin{eqnarray}
p_{ik}^{T,\mathrm{RABL}}
\in
\arg\min_{p_i\in[1,3]}
\mathbb E_{\mu_i^{T+1}}
\left[
\rho_{\eta(\cdot\mid\theta_i^k;\zeta)}^{\mathrm{ent}}
\left(
c_i(p_i,p_{-i}^{T,\mathrm{RABL}},\theta_i^k,\theta_{-i})
\right)
\right],\quad \text{for } i=1,2,\; k=1,\ldots,K.\quad
\end{eqnarray}
Next, we consider the BNE-BPD model where firm $i$ 
evaluates
interim uncertainty in its rival's type using the entropic risk measure (\(\gamma=0.1\)):
\begin{eqnarray}
p_{ik}^{T,\mathrm{BPD}}
\in
\arg\min_{p_i\in[1,3]}
\rho_{\nu_i^{T+1}(\cdot\mid\theta_i^k)}^{\mathrm{ent}}
\left(
c_i(p_i,p_{-i}^{T,\mathrm{BPD}},\theta_i^k,\theta_{-i})
\right)
\quad \text{for } i=1,2,
\; k=1,\ldots,K.
\end{eqnarray}
As a benchmark, we compute the oracle BNE where both firms know the true parameter \(\zeta^*=1\) and use the entropic risk measure with \(\gamma=0.1\) to evaluate uncertainty in their rivals' types:
\begin{eqnarray}
p_{ik}^*
\in
\arg\min_{p_i\in[1,3]}
\rho_{\eta(\cdot\mid\theta_i^k;\zeta^*)}^{\mathrm{ent}}
\left(
c_i(p_i,p_{-i}^*,\theta_i^k,\theta_{-i})
\right),\quad \text{for } i=1,2,\; k=1,\ldots,K.
\end{eqnarray}

In the first experiment, we examine the convergence of BNE-RABL and BNE-BPD as \(T\) increases.
We fix the number of type-grid points with \(K=31\)
and approximate continuous distributions with discrete
distributions supported at the grid points.
We conduct \(200\) independent simulations  and report the median of each metric together with the \(10\)th--\(90\)th percentile band.
Figure~\ref{fig:experiment-convergence} reports the following metrics at \(T\in\{1,2,5,10,20,50,100,200,500\}\) and Figure~\ref{fig:experiment-strategy-error} reports the deviations of the BNE-RABL and BNE-BPD from the oracle BNE at \(T\in\{1,10,100\}\):
\begin{enumerate} [(i)]
    \item The posterior mean-squared error of $\zeta$:
    \begin{eqnarray}
    P_T
    = 
\max_{i} \mathbb{E}_{\mu_i^{T+1}}\left[\| \zeta - \zeta^*\|^2\right].
    \end{eqnarray}

    \item The response function errors of the BNE-RABL and the BNE-BPD relative to the oracle BNE:
    \begin{eqnarray}
    \max_{i,k}
    \left|
    p_{ik}^{T,\mathrm{RABL}}-p_{ik}^*
    \right|
    \quad\text{and}\quad
    \max_{i,k}
    \left|
    p_{ik}^{T,\mathrm{BPD}}-p_{ik}^*
    \right|.
    \end{eqnarray}

    \item The response function discrepancy between BNE-RABL and BNE-BPD:
    \begin{eqnarray}
    \Delta_T
    :=
    \max_{i,k}
    \left|
    p_{ik}^{T,\mathrm{RABL}}-p_{ik}^{T,\mathrm{BPD}}
    \right|.
    \end{eqnarray}
\end{enumerate}

\begin{figure}
    \centering
    \includegraphics[width=0.9\linewidth]{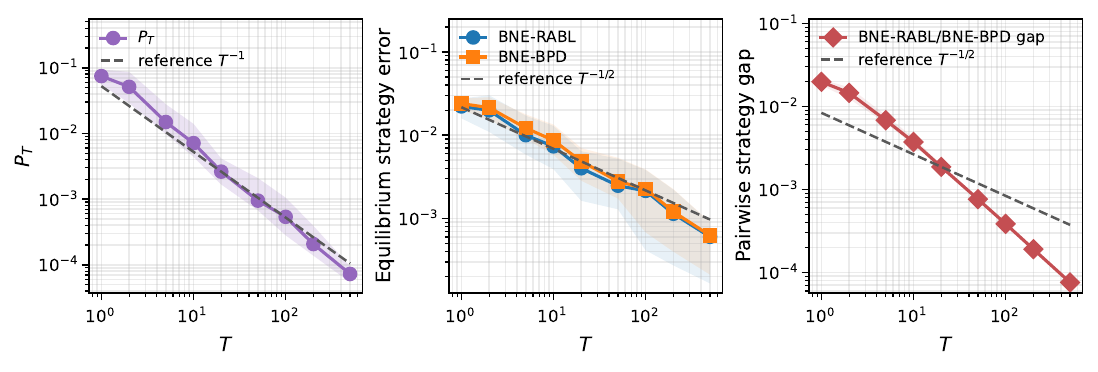}
    \caption{Posterior mean-squared  error of $\zeta$, response function errors relative to the oracle BNE, and the discrepancy between BNE-RABL and BNE-BPD. 
    The dashed lines represent the rates of convergence derived in Theorems~\ref{thm:convergence-rate-bne-rabl} and \ref{thm:convergence-rate-bne-bpd} and Corollary~\ref{corollary:bne-rabl-and-bne-bpd}.
    Specifically, each dashed line is plotted with \(m_{20}(T/20)^\alpha\), where \(\alpha=-1\) in the first \(P_T\) panel, \(\alpha=-1/2\) in the second strategy-error and third pairwise-gap panels, and \(m_{20}\) denotes the corresponding empirical median at \(T=20\).}
    \label{fig:experiment-convergence}
\end{figure}

\begin{figure}
    \centering
    \includegraphics[width=0.9\linewidth]{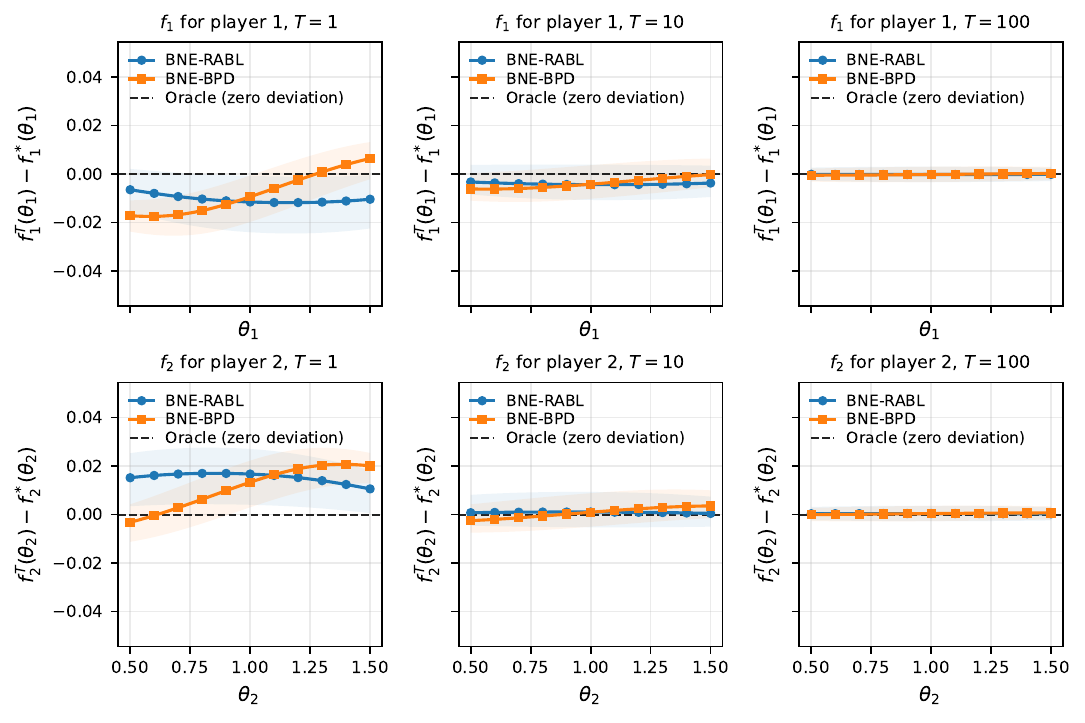}
    \caption{Deviations of the BNE-RABL and BNE-BPD strategies from the oracle strategies at \(T\in\{1,10,100\}\).}
    \label{fig:experiment-strategy-error}
\end{figure}

The results in Figures~\ref{fig:experiment-convergence}-\ref{fig:experiment-strategy-error} are consistent with the convergence results established in Section~\ref{sec:asymptotic-converge}.
In Figure~\ref{fig:experiment-convergence},
the posterior error \(P_T\) decreases from approximately \(7\times10^{-2}\) when \(T=1\) to \(8\times10^{-5}\) when \(T=500\).
The response function errors of both BNE-RABL and BNE-BPD decrease from approximately \(2\times10^{-2}\) to \(8\times10^{-4}\).
The decay of the response function errors is consistent with the convergence rates derived in Theorems~\ref{thm:convergence-rate-bne-rabl} and~\ref{thm:convergence-rate-bne-bpd}.
Moreover, the pairwise gap \(\Delta_T\) decreases from approximately \(2\times10^{-2}\) to \(8\times10^{-5}\), which suggests that the discrepancy between the equilibria generated by two models goes to $0$
faster than the theoretical 
rate
derived in Corollary~\ref{corollary:bne-rabl-and-bne-bpd}.
Likewise, Figure~\ref{fig:experiment-strategy-error} shows that the strategy profiles \(f_i^T(\theta_i)\) 
exhibit visible deviations from the oracle strategy \(f_i^*(\theta_i)\) at \(T=1\), become substantially closer at \(T=10\), and are nearly indistinguishable from the oracle profiles at \(T=100\).
These results 
highlight 
the impact of learning on 
the convergence of the equilibria.

Moreover, Figure~\ref{fig:experiment-strategy-error} shows that
the response functions generated by the two equilibrium models intersect:
the response function in the BNE-BPD model sets a lower price than that in the BNE-RABL model when player $i$'s realized marginal cost $\theta_i^T$ is low, and a higher price when $\theta_i^T$ is large.
This is because the BNE-BPD model uses the realized marginal cost $\theta_i^T$ 
to implicitly reweigh the posterior distribution of $\zeta$
via $\frac{p_i(\theta_i;\zeta)}
{\int_\mathcal{Z} p_i(\theta_i;\zeta)m_i^t(\zeta)d\zeta}$
as shown in \eqref{eqn:bpd-condition}. 
Since $\zeta$ is the mean value of the Gaussian distribution of $\theta_i$,
a smaller value of $\zeta$ assigns a higher density to a low observed value of $\theta_i$ (when $\theta_i<\zeta$);
conversely,
a low observed value of $\theta_i^T$ 
induces a higher likelihood with a smaller value of $\zeta$.
Consequently, a low realized marginal cost $\theta_i^T$ shifts the reweighted posterior in the BNE-BPD model toward small values of $\zeta$.
Since $\zeta$ is also the mean value of the Gaussian distribution of $\theta_{-i}$,
this shift also moves the predicted distribution of the rival's cost $\theta_{-i}$ downward and hence lowers the predicted distribution of the rival's price.
Under strategic complementarity characterized in \eqref{eqn:cost-linear},
player $i$ responds by lowering its own price in order to remain competitive.
Conversely, a larger realization of marginal cost $\theta_i^T$ shifts the reweighted posterior in the BNE-BPD model toward larger values of $\zeta$, which yields a higher predicted distribution of the rival's price and hence a higher own price.


To examine the observations above, we introduce the following two metrics.
\begin{enumerate} [(i)]
\item 
The expected signed deviation of the posterior used in the BNE-RABL model:
\begin{equation}
\mathcal B_{i,T}^{\mathrm{RABL}}(\theta_i)
:=
\int_{\mathcal Z}
(\zeta-\zeta^*)
m_i^{T+1}(\zeta)d\zeta,
\end{equation}
and the expected signed deviation of the reweighted posterior used in the BNE-BPD model:
\begin{equation}
\mathcal B_{i,T}^{\mathrm{BPD}}(\theta_i)
:=
\int_{\mathcal Z}
(\zeta-\zeta^*)
\frac{p_i(\theta_i;\zeta) }{\int_\mathcal{Z}p_i(\theta_i;\zeta) m_i^{T+1}(\zeta)d\zeta}
m_i^{T+1}(\zeta)d\zeta.
\label{eq}
\end{equation}
Since $\mathcal B_{i,T}^{\mathrm{RABL}}(\theta_i)$ is independent of $\theta_i$, it is a constant.

\item Deviation of the conditional expectation of the rival's marginal cost in the BNE-RABL model relative to the oracle counterpart
\begin{align}
\mathcal R_{i,T}^{\mathrm{RABL}}(\theta_i)
&:=
\int_{\Theta_{-i}}
\theta_{-i}
\widetilde q_i^{T+1}
(\theta_{-i}\mid\theta_i)\,d\theta_{-i} -
\int_{\Theta_{-i}}
\theta_{-i}
p(\theta_{-i}\mid\theta_i;\zeta^*)\,d\theta_{-i},
\label{eq:rival-type-deviation-rabl}
\end{align}
and deviation of the conditional expectation of the rival's marginal cost in the BNE-BPD model relative to the oracle counterpart:
\begin{align}
\mathcal R_{i,T}^{\mathrm{BPD}}(\theta_i)
&:=
\int_{\Theta_{-i}}
\theta_{-i}
q_i^{T+1}
(\theta_{-i}\mid\theta_i)\,d\theta_{-i} -
\int_{\Theta_{-i}}
\theta_{-i}
p(\theta_{-i}\mid\theta_i;\zeta^*)\,d\theta_{-i},
\label{eq:rival-type-deviation-bpd}
\end{align}
where $\widetilde q_i^{T+1}$ and $q_i^{T+1}$ are given in \eqref{eqn:rabl-conditional} and \eqref{eqn:bpd-condition}, respectively.

\end{enumerate}

Figures~\ref{fig:parameter_belief_deviations} and \ref{fig:implied_rival_type_deviations} report these metrics for round $T\in\{1,10,100\}$.
When player $i$'s marginal cost $\theta_i$ is low,
the reweighted posterior distribution used in the BNE-BPD model yields lower estimates of $\zeta$ and of rival's marginal cost $\theta_{-i}$ than the standard posterior distribution used in the BNE-RABL model when the player's own marginal cost $\theta_i$ is low;
conversely, when $\theta_i$ is high, the the reweighted posterior yields higher estimates than standard posterior.
The discrepancy between the two models diminishes as $T$ increases. 
These observations support the explanation proposed above for the intersection of the response functions generated by the two models.

\begin{figure}[h]
\centering
\includegraphics[width=0.9\linewidth]
{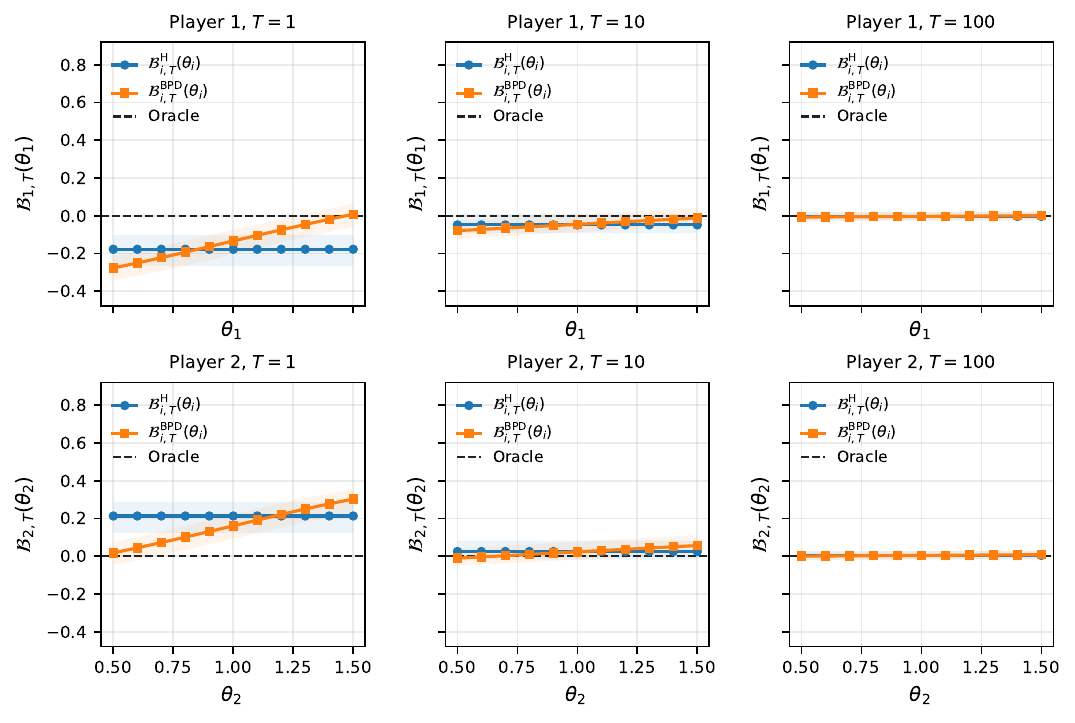}
\caption{$\mathcal B_{i,T}^{\mathrm{RABL}}(\theta_i)$ and $\mathcal B_{i,T}^{\mathrm{BPD}}(\theta_i)$ at round $T\in\{1,10,100\}$.}
\label{fig:parameter_belief_deviations}
\end{figure}

\begin{figure}[h]
\centering
\includegraphics[width=0.9\linewidth]
{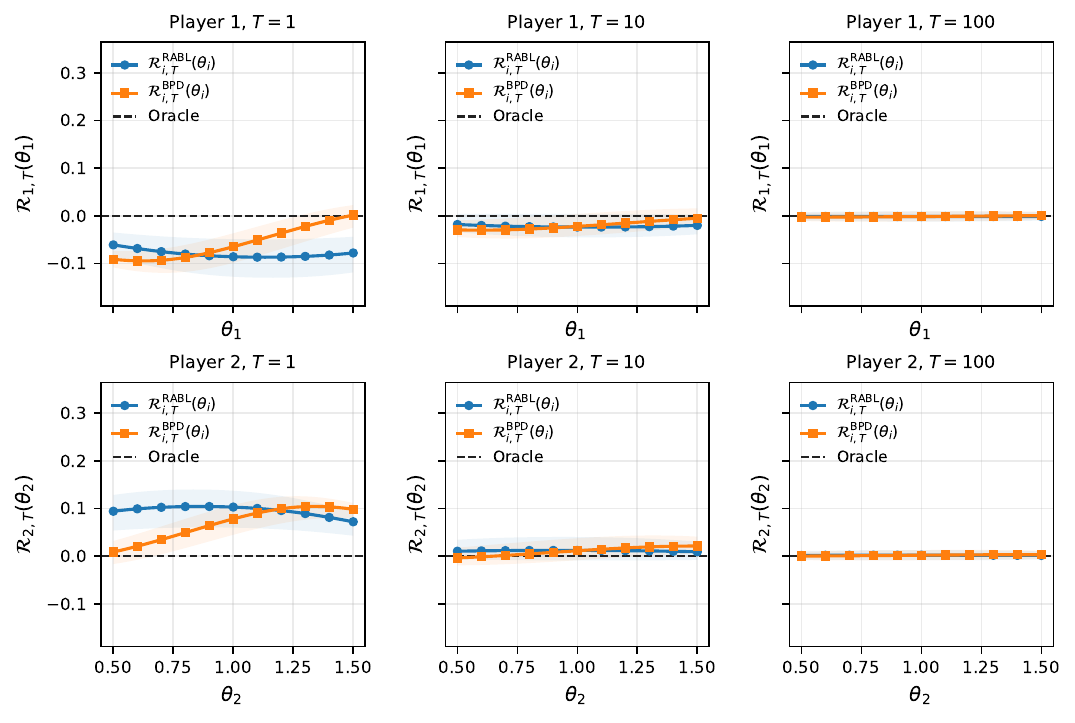}
\caption{
$\mathcal R_{i,T}^{\mathrm{RABL}}(\theta_i)$ and $\mathcal R_{i,T}^{\mathrm{BPD}}(\theta_i)$ at $T\in\{1,10,100\}$.
}
\label{fig:implied_rival_type_deviations}
\end{figure}

In the second experiment, we examine the computational time of the BNE-RABL and BNE-BPD models. 
We fix the round \(T=20\) and vary the number of grid points \(K\) from \(11\) to \(201\) to assess the effect of the discretization size on the CPU time.
The results are reported in Figure~\ref{fig:experiment-solve-time}.
The CPU time for solving the BNE-RABL model increases from approximately \(0.10\) seconds at \(K=11\) to \(18\) seconds at \(K=201\), whereas that for solving the BNE-BPD model increases from approximately \(0.05\) seconds to \(1.4\) seconds.
These results show that the BNE-BPD model is more computationally efficient than the BNE-RABL model, especially for large discretization sizes.

\begin{figure}
    \centering
    \includegraphics[width=0.6\linewidth]{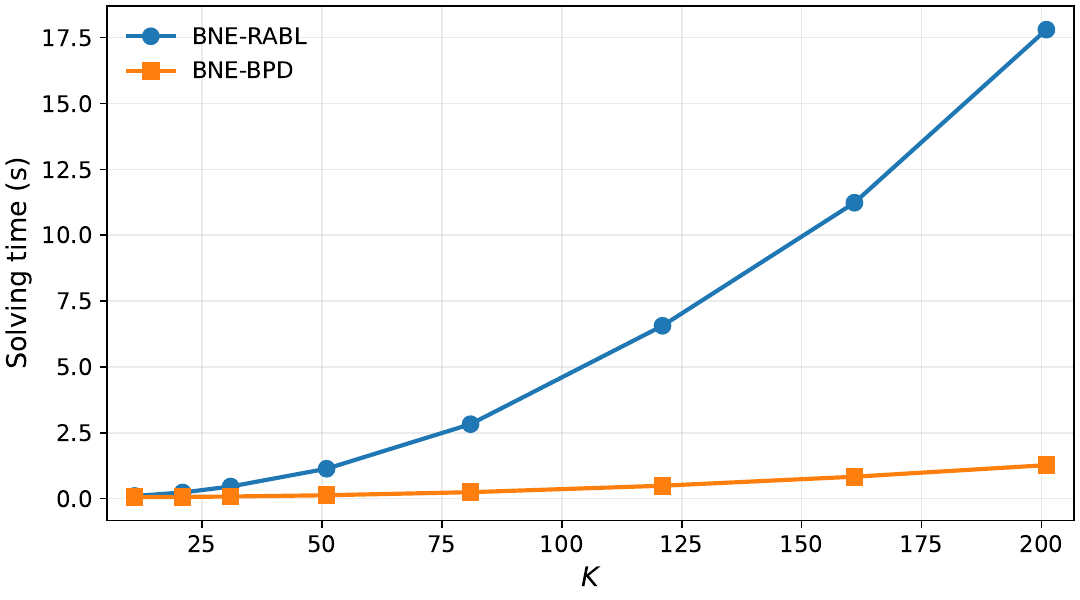}
    \caption{GAMS CPU times for BNE-RABL and BNE-BPD at different discretization sizes $K$.}
    \label{fig:experiment-solve-time}
\end{figure}

\subsection{Application in price competition: numerical tests with multinomial logit demand}

\label{sec:price-competition-mnl}

To further assess the effectiveness of the proposed Bayesian learning framework,
we replace the linear demand function \eqref{eqn:linear-price-demand} used in the preceding experiment
with a multinomial-logit (MNL) demand model, and consider a setting involving two heterogeneous customer groups with distinct preference parameters, indexed by $h\in \{1,2\}$.
We assume that the two customer groups $h=1$ and $h=2$ have proportions $q_1 =0.8$ and $q_2 = 0.2$, respectively.
For customer group $h$, firm $i$'s market share is defined as
\begin{equation}
 D_{hi}(p)
 =
 \frac{\exp(v_{hi}-\beta_{hi}p_i)}
 {1+\sum_{j=1}^{2}\exp(v_{hj}-\beta_{hj}p_j)},
 \label{eq:mnl-robustness-share}
\end{equation}
where
\[
 (v_{1i},\beta_{1i})=(3.0,0.85),\qquad
 (v_{2i},\beta_{2i})=(6.0,3.50).
\]
Firm $i$'s loss function is then given by
\begin{eqnarray}
    c_i(p_i,p_{-i},\theta_i,\theta_{-i})
 &=&-(p_i-\theta_i) \sum_{h=1}^{2}q_h D_{hi}(p) \nonumber\\
 &=& -(p_i-\theta_i) \sum_{h=1}^{2}q_h \frac{\exp(v_{hi}-\beta_{hi}p_i)} 
 {1+\sum_{j=1}^{2}\exp(v_{hj}-\beta_{hj}p_j)}.
 \label{eq:mnl-robustness-loss}
\end{eqnarray}
The information structure, 
Bayesian updating procedure,
firms' risk attitudes,
and the formulations of the BNE-RABL and BNE-BPD models remain the same as in the preceding experiment.

Figures~\ref{fig:experiment3-convergence}--\ref{fig:experiment3-strategy-error} show that 
both the response functions of the BNE-RABL and BNE-BPD models
approach the response functions of the oracle equilibrium, 
and the discrepancy between the equilibria generated by two models goes to $0$ as $T$ increases.
In Figure~\ref{fig:experiment3-solve-time}, the CPU time of solving the BNE-BPD model is significantly shorter than that of the BNE-RABL model, which shows the computational advantage of the BNE-BPD formulation.

\begin{figure}
    \centering
    \includegraphics[width=0.9\linewidth]{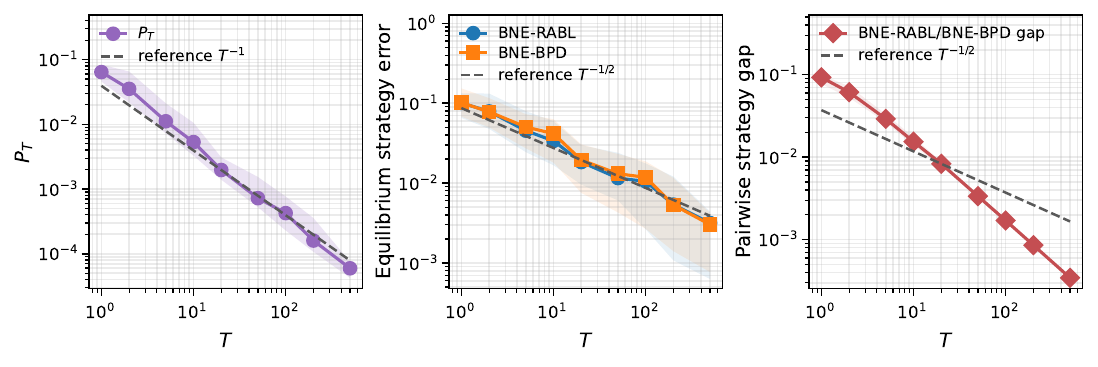}
    \caption{Posterior mean-squared  error of $\zeta$, response function errors relative to the oracle BNE, and the discrepancy between BNE-RABL and BNE-BPD.} 
    \label{fig:experiment3-convergence}
\end{figure}

\begin{figure}
    \centering
    \includegraphics[width=0.9\linewidth]{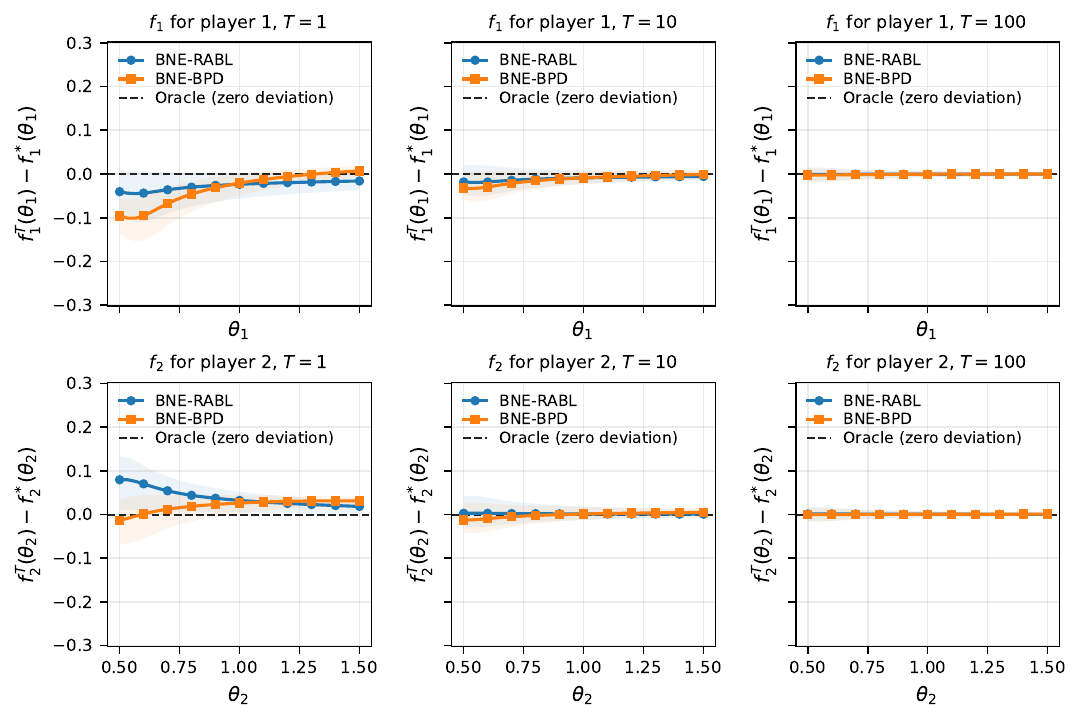}
    \caption{Deviations of the BNE-RABL and BNE-BPD strategies from the oracle strategies at \(T\in\{1,10,100\}\).}
    \label{fig:experiment3-strategy-error}
\end{figure}

\begin{figure}
    \centering
    \includegraphics[width=0.6\linewidth]{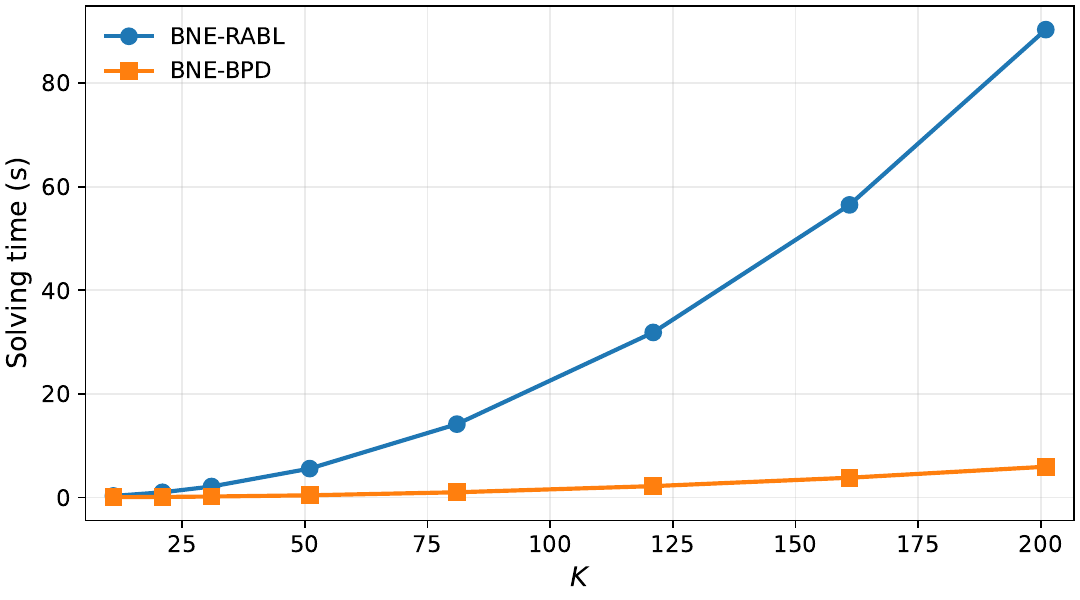}
    \caption{GAMS CPU times for BNE-RABL and BNE-BPD at different discretization sizes $K$.}
    \label{fig:experiment3-solve-time}
\end{figure}

\section{Concluding Remarks}
\label{sec:concluding_remarks}

In this paper, we study Bayesian games in which the true joint distribution of the players' types is initially unknown and must be learned from the type profiles observed over time. 
Within a Bayesian learning framework, we propose two equilibrium models, BNE-RABL and BNE-BPD. The former explicitly distinguishes the players' risk attitudes toward the two sources of uncertainty, namely aleatoric and epistemic uncertainty, whereas the latter aggregates them by marginalizing out the distributional parameter.
Our numerical experiments reveal a trade-off between tractability and modelling fidelity. On the one hand, the BNE-BPD model is substantially easier to solve, since marginalization removes the outer risk functional and thus reduces the problem to a single-level equilibrium condition. 
On the other hand, when the available information about the epistemic uncertainty is limited, 
the 
BNE-BPD encapsulates
a large dispersion across realizations of the distributional parameter $\zeta$, 
which may 
result in 
significant modelling errors. 
Once sufficient information about 
$\zeta$ has been accumulated, this dispersion becomes small, and the BNE-BPD model is obviously a better choice.

It is worth noting that we focus on myopic players in order to study the convergence of equilibrium behavior.
A natural direction for future research is to consider strategic players who maximize a discounted sum of payoffs and to investigate how distribution learning influences the resulting equilibrium behavior; see e.g., \cite{noguchi2015bayesian,norman2022possibility}.
Another promising direction is to develop effective solution approaches based on the ex-ante equilibrium model discussed in Section~\ref{sec:ex-ante} and \cite{liu2026games}.
It might also be interesting to couple the learning process of the unknown distributional parameter with the equilibrium-seeking process; see, e.g., \cite{lei2020asynchronous,dey2026analysis}.
We leave all these for future research.

\begin{appendices}

\section{Auxiliary results}
\begin{lemma}
\label{lem:tv-bound-exponential-tilting}
Let $(\Omega,\mathcal G,P)$ be a probability space.
Let $X$ and $X'$ be two bounded measurable real-valued random variables on
$\Omega$, and let $\gamma>0$.
Define two probability measures $Q_X$ and $Q_{X'}$ by
$$ dQ_X = \frac{\exp(\gamma X)}{\mathbb E_P[\exp(\gamma X)]}dP, \quad dQ_{X'} = \frac{\exp(\gamma X')} {\mathbb E_P[\exp(\gamma X')]} dP. $$
Then, 
\begin{eqnarray*}
    \dd_{\rm TV}(Q_X,Q_{X'})
    \leq
    \frac{\gamma\|X-X'\|_\infty}{2}.
\end{eqnarray*}
\end{lemma}

\noindent
\textbf{Proof.}
Let $R = \frac{dQ_X}{dQ_{X'}}.$
By the definition $dQ_X$ and $dQ_{X'}$,
\begin{eqnarray*}
\log R = \log\left( \exp(\gamma(X-X')) \frac{\mathbb E_P[\exp(\gamma X')}{\mathbb E_P[\exp(\gamma X)}\right)= \gamma(X-X') + \log \frac{ \mathbb E_P[\exp(\gamma X')]} {\mathbb E_P[\exp(\gamma X)]} .
\end{eqnarray*}
The second term on the right-hand side is a constant.
Hence
$$
    \operatorname*{ess\,sup}\log R
    -
    \operatorname*{ess\,inf}\log R
    \leq 2\gamma\|X-X'\|_\infty,
$$
and thus
$$
\frac{\operatorname*{ess\,sup} R}{\operatorname*{ess\,inf}R} \leq e^{2\gamma\|X-X'\|_\infty}.
$$
Let
$
    m:=\operatorname*{ess\,inf}R,
    M:=\operatorname*{ess\,sup}R .
$
Then $M/m\leq e^{2\gamma\|X-X'\|_\infty}$.
Moreover, since $\mathbb E_{Q_{X'}}[R]=1$, we have $m\leq1\leq M$ unless
$Q_X=Q_{X'}$.

For any measurable set $A$, we have 
\begin{eqnarray*}
    Q_X(A)-Q_{X'}(A)
    &=&
    \int_A \,dQ_{X} - \int_A \,dQ_{X'} = \int_A R\,dQ_{X'} - \int_A \,dQ_{X'}\\
    &=&
    \mathbb E_{Q_{X'}}[(R-1)\mathbf 1_A].
\end{eqnarray*}
By the definition of total-variation distance,
\begin{eqnarray*}
    \dd_{\rm TV}(Q_X,Q_{X'})=\sup_{A\in \mathcal{G}}|Q_X(A)-Q_{X'}(A)| = \sup_{A\in \mathcal{G}} \mathbb E_{Q_{X'}}[(R-1)\mathbf 1_A],
\end{eqnarray*}
where the supremum is obtained with $A^*=\{R:R-1\geq0\}$.
Since $R\leq M$ on $A^*$ and $R\geq m$ on $(A^*)^c$, we have 
\begin{eqnarray*}
    \mathbb E_{Q_{X'}}[R\mathbf 1_{A^*}]
    &\leq&
    MQ_{X'}(A^*),
    \\
    \mathbb E_{Q_{X'}}[R\mathbf 1_{A^*}]
    &=&
    1-\mathbb E_{Q_{X'}}[R\mathbf 1_{{(A^*)}^c}]
    \leq
    1-m(1-Q_{X'}(A^*)).
\end{eqnarray*}
Therefore
\begin{eqnarray*}
    \dd_{\rm TV}(Q_X,Q_{X'})= \mathbb E_{Q_{X'}}[(R-1)\mathbf 1_{A^*}]
    \leq
    \min\{(M-1)Q_{X'}(A^*),\,(1-m)(1-Q_{X'}(A^*))\}.
\end{eqnarray*}
Since $(M-1)Q_{X'}(A^*)$ is increasing in $Q_{X'}(A^*)\in [0,1]$ and $(1-m)(1-Q_{X'}(A^*))$ is decreasing in $Q_{X'}(A^*)\in [0,1]$, $\min\{(M-1)Q_{X'}(A^*),\,(1-m)(1-Q_{X'}(A^*))\}$ has its maximum at $(M-1)Q_{X'}(A^*) = (1-m)(1-Q_{X'}(A^*))$, which gives 
$Q_{X'}(A^*) = \frac{1-m}{M-m}$ and thus
\begin{eqnarray}
\label{eqn:tv-bound-m-M}
    \dd_{\rm TV}(Q_X,Q_{X'})
    &\leq&
    \frac{(M-1)(1-m)}{M-m}.
\end{eqnarray}
Let $K:=M/m$ and thus $M=Km$. Since $m\leq1\leq M$, we have $m\in[1/K,1]$. 
The right-hand side of \eqref{eqn:tv-bound-m-M} becomes
\begin{eqnarray*}
    F(m,K)
    &:=&
    \frac{(Km-1)(1-m)}{m(K-1)}
    =
    \frac{K+1-Km-1/m}{K-1}.
\end{eqnarray*}
Differentiating with respect to $m$ gives
\begin{eqnarray*}
    \frac{dF(m,K)}{dm}
    &=&
    \frac{-K+1/m^2}{K-1}.
\end{eqnarray*}
Thus $F(\cdot,K)$ is maximized at $m=1/\sqrt K\in[1/K,1]$, and the maximum is
$\frac{\sqrt K-1}{\sqrt K+1}.$
Since this expression is increasing in $K$ and $K=M/m\leq e^{2\gamma\|X-X'\|_\infty}$, we obtain
\begin{eqnarray}
\label{eqn:tv-tanh-bound-general}
    \dd_{\rm TV}(Q_X,Q_{X'})
    &\leq&
    \frac{e^{\gamma\|X-X'\|_\infty}-1}{e^{\gamma\|X-X'\|_\infty}+1}
    = \tanh\left(\frac{\gamma\|X-X'\|_\infty}{2}\right) \leq
    \frac{\gamma\|X-X'\|_\infty}{2}.
\end{eqnarray}
The proof is complete.
\hfill$\Box$
\end{appendices}

\bibliographystyle{apalike}
\bibliography{sample}

\end{document}